\documentclass{amsart}

\usepackage{hyperref}
\hypersetup{
     colorlinks=true,
     linkcolor=blue,
     citecolor=blue,
     urlcolor=blue,
}

\makeatletter
\newcommand\org@maketitle{}
\newcommand\@authors{}
\let\org@maketitle\maketitle
\def\maketitle{%
	\let\@authors\authors
	\nxandlist{; }{ and }{; }\@authors
	\hypersetup{
		linktocpage=true,
		pdftitle={\@title},
                pdfauthor={\@authors},
                pdfsubject={\subjclassname. \@subjclass},
		pdfkeywords={\@keywords}
	}%
	\org@maketitle
}
\makeatother

\usepackage{amssymb}
\usepackage{mathtools}
\usepackage{enumitem}
\usepackage{rsfso}
\usepackage{microtype}

\usepackage[alphabetic, msc-links]{amsrefs}
\usepackage{doi}
\renewcommand{\PrintDOI}[1]{\doi{#1}}

\let\arXiv\arxiv

\numberwithin{equation}{section}
\newtheorem{maintheorem}{Theorem}

\newtheorem{theorem}{Theorem}[section]
\newtheorem{lemma}[theorem]{Lemma}
\newtheorem{proposition}[theorem]{Proposition}
\newtheorem{corollary}[theorem]{Corollary}
\theoremstyle{definition}
\newtheorem{definition}[theorem]{Definition}
\theoremstyle{remark}
\newtheorem{remark}[theorem]{Remark}
\newtheorem{example}[theorem]{Example}

\newcommand{\cR}{{\mathcal R}}

\newcommand{\cC}{{\mathcal C}}

\newcommand{\cK}{{\mathcal K}}
\newcommand{\cM}{{\mathcal M}}

\newcommand{\La}{\Delta}
\newcommand{\Ga}{\Gamma}
\newcommand{\Ld}{\Lambda}
\newcommand{\Si}{\Sigma}
\newcommand{\D}{\nabla}
\newcommand{\ra}{\rightarrow}
\newcommand{\pa}{\partial}

\newcommand{\R}{\mathbb{R}}
\newcommand{\vp}{\varphi}

\newcommand{\al}{\alpha}
\newcommand{\be}{\beta}
\newcommand{\g}{\gamma}
\newcommand{\de}{\delta}
\newcommand{\e}{\varepsilon}

\newcommand{\la}{\lambda}
\newcommand{\ka}{\kappa}
\newcommand{\si}{\sigma}
\newcommand{\om}{\omega}

\newcommand{\h}{\widetilde{h}}
\newcommand{\K}{\widetilde{K}}
\newcommand{\N}{\widetilde{N}}

\newcommand{\wDh}{\widehat{\D h}}
\newcommand{\wDu}{\widehat{\D u}}

\newcommand{\bx}{\bar{x}}
\newcommand{\by}{\bar{y}}

\newcommand{\loc}{\mathrm{loc}}
\newcommand{\dist}{\operatorname{dist}}
\newcommand{\mean}[1]{\langle{#1}\rangle}
\renewcommand{\Re}{\operatorname{Re}}

\def\Xint#1{\mathchoice
  {\XXint\displaystyle\textstyle{#1}}%
  {\XXint\textstyle\scriptstyle{#1}}%
  {\XXint\scriptstyle\scriptscriptstyle{#1}}%
  {\XXint\scriptscriptstyle\scriptscriptstyle{#1}}%
  \!\int}
\def\XXint#1#2#3{{\setbox0=\hbox{$#1{#2#3}{\int}$}
    \vcenter{\hbox{$#2#3$}}\kern-.5\wd0}}

\def\dashint{\Xint-}

\mathchardef\ordinarycolon\mathcode`\:
\mathcode`\:=\string"8000
\begingroup \catcode`\:=\active
  \gdef:{\mathrel{\mathop\ordinarycolon}}
\endgroup

\allowdisplaybreaks[4]

\author{Seongmin Jeon}
\address{Department of Mathematics, Purdue University, West Lafayette,
  IN 47907, USA}
\email[S.J.]{jeon54@purdue.edu}
\author{Arshak Petrosyan}
\address{Department of Mathematics, Purdue University, West Lafayette,
  IN 47907, USA}
\email[A.P.]{arshak@purdue.edu}
\thanks{The second author is supported in part by NSF Grant DMS-1800527.}
\title{Almost minimizers for the thin obstacle problem}

\begin{document}

\begin{abstract} We consider Anzellotti-type almost minimizers for the
  thin obstacle (or Signorini) problem with zero thin obstacle and
  establish their $C^{1,\beta}$ regularity on the either side of the
  thin manifold, the optimal growth away from the free boundary, the
  $C^{1,\gamma}$ regularity of the regular part of the free boundary,
  as well as a structural theorem for the singular set. The analysis
  of the free boundary is based on a successful adaptation of energy
  methods such as a one-parameter family of Weiss-type monotonicity
  formulas, Almgren-type frequency formula, and the epiperimetric and
  logarithmic epiperimetric inequalities for the solutions of the thin
  obstacle problem.
\end{abstract}
\maketitle

\tableofcontents

\newpage


\section{Introduction and main results}

\subsection{The thin obstacle (or Signorini) problem}

Let $D\subset\R^n$ be an open set and $\cM\subset\R^n$ a smooth
$(n-1)$-dimensional manifold (the \emph{thin space}) and consider the
problem of minimizing the Dirichlet energy
\begin{equation}\label{eq:DirEner}
  J_D(u):=\int_D |\nabla u(x)|^2dx
\end{equation}
among all functions $u\in W^{1,2}(D)$ satisfying
$$
u=g\text{ on }\partial D,\quad u\geq \psi\text{ on }\cM\cap D,
$$
where $\psi:\cM\to\R$ is the so-called \emph{thin obstacle} and
$g:\partial D\to \R$ is the prescribed boundary data with $g\geq\psi$
on $\cM\cap\partial D$. This problem is known as \emph{the thin
  obstacle problem}. In other words, it is a constrained minimization
problem for the energy functional $J_D$ on a closed convex set
$$
\mathfrak{K}_{\psi,g}(D,\mathcal{M}):=\{u\in W^{1,2}(D): \text{$u=g$
  on $\partial D$, $u\geq \psi$ on $\cM\cap D$}\}.
$$
This problem can be viewed as a scalar version of the Signorini
problem with unilateral constraint from elastostatics \cite{Sig59} and
is often referred to as \emph{the Signorini problem}.  It goes back to
the origins of variational inequalities and is considered as one of
the prototypical examples of such problems, see \cite{DuvLio76}. An
equivalent formulation is given in the form
\begin{align*}
  \Delta u=0&\quad\text{on }D\setminus\mathcal{M},\\
  u=g &\quad\text{on }\partial D,\\
  u\geq \psi,\quad \partial_{\nu^+}u+\partial_{\nu^-}u\geq 0,\quad
  (\partial_{\nu^+}u+\partial_{\nu^-}u)(u-\psi)=0&\quad\text{on
                                                   }\mathcal{M}\cap D,
\end{align*}
where the conditions on $\cM\cap D$ are known as the \emph{Signorini
  complementarity} (or \emph{ambiguous}) \emph{conditions}. Here,
$\partial_{\nu^\pm}$ are the exterior normal derivatives from the
either side of $\mathcal{M}$. In particular, at points on $\cM\cap D$
we must have one of the two boundary conditions satisfied: either
$u=\psi$ or $\partial_{\nu^+}u+\partial_{\nu^-}u=0$.  The set
\begin{equation}
  \label{eq:free-bdry}
  \Gamma(u):=\partial_{\cM}\{x\in\cM\cap D:u(x)=\psi(x)\},
\end{equation}
which separates the regions where different boundary conditions are
satisfied, is known as the \emph{free boundary} and plays a central
role in the analysis of the problem.

Because of the presence of the thin obstacle, it is not hard to
realize that the solutions $u$ of the Signorini problem are at most
Lipschitz across $\cM$, even if both $\cM$ and $\psi$ are smooth, as
we may have $\partial_{\nu^+}u+\partial_{\nu^-}u>0$ at some points on
$\cM$\footnote{This can be seen on the explicit example
  $u(x)=\Re(x_1+i|x_n|)^{3/2}$, which is a solution of the obstacle
  problem with $\psi=0$ on $\mathcal{M}=\{x_n=0\}$.}.  However, it has
been known since the works \cites{Caf79,Kin81,Ura85} that the
solutions of the thin obstacle problem are $C^{1,\beta}$ on
$\mathcal{M}$ and consequently on the either side of $\mathcal{M}$, up
to $\mathcal{M}$.  In recent years, there has been a renewed interest
in this problem, following the breakthrough result of Athanasopoulos
and Caffarelli \cite{AthCaf04} on the optimal $C^{1,1/2}$ regularity
of the minimizers (on the either side of $\mathcal{M}$) as well as its
relation to the obstacle-type problems for the fractional Laplacian
through the Caffarelli-Silvestre extension \cite{CafSil07}. There has
also been a significant effort in understanding the structure and the
regularity of the {free boundary}.  The results have been obtained in
many settings, such as for the equations with variable coefficients,
time-dependent versions, problems for fractional Laplacian and other
nonlocal equations, both regarding the regularity of minimizers, as
well as the properties of the free boundary; see e.g.,
\cites{Sil07,AthCafSal08,CafSalSil08,GarPet09,GarSVG14,KocPetShi15,PetPop15,PetZel15,DeSSav16,GarPetSVG16,KocRueShi16,BanSVGZel17,CafRosSer17,DanGarPetTo17,KocRueShi17a,KocRueShi17b,ColSpoVel17,AthCafMil18,DanPetPop18,FocSpa18}
and many others.

\subsection{Almost minimizers} In \cite{Anz83}, Anzellotti introduced
the notion of almost minimizers for energy functionals. Given $r_0>0$,
we say that $\omega:(0,r_0)\to [0,\infty)$ is a modulus of continuity
or a \emph{gauge function}, if $\omega(r)$ is monotone nondecreasing
in $r$ and $\omega(0+)=0$.

\begin{definition}[Almost minimizers]
  \label{def:alm-min}
  Given $r_0>0$ and a gauge function $\omega(r)$ on $(0,r_0)$, we say
  that $u\in W^{1,2}_\loc(D)$ is an \emph{almost minimizer} (or
  $\omega$-\emph{minimizer}) for the functional $J_D$, if, for any
  ball $B_r(x_0)\Subset D$ with $0<r<r_0$, we have
  \begin{equation}\label{eq:alm-min-Dir}
    J_{B_r(x_0)}(u)\leq (1+\omega(r))J_{B_r(x_0)}(v)\quad\text{for any }v\in u+
    W^{1,2}_0(B_{r}(x_0)).
  \end{equation}
\end{definition}

The idea is that the Dirichlet energy of $u$ on the ball $B_r(x_0)$ is
not necessarily minimal among all competitors
$v\in u+W^{1,2}_0(B_r(x_0))$ but \emph{almost minimal} in the sense
that it cannot decrease more than by a factor of $1+\omega(r)$. In the
specific case of the energy functional $J_D$ in \eqref{eq:DirEner},
i.e., the Dirichlet energy, we refer to the almost minimizers of $J_D$
as \emph{almost harmonic functions} in $D$.

Results on almost minimizers for more general energy functionals can
be found in \cites{DuzGasGro00,EspLeoMin04,EspMin99, Min06}.  Similar
notions were considered earlier in the context of the geometric
measure theory \cites{Alm76,Bom82}, see also \cite{Amb97}. Almost
minimizers are also related to quasiminimizers, introduced in
\cites{GiaGiu82,GiaGiu84}, see also \cites{Giu03}.  For energy
functionals exhibiting free boundaries, almost minimizers have been
considered only recently in \cites{DavTor15,DavEngTor17,deQTav18,DeSSav18,DeSSav19}.

Almost minimizers can be viewed as perturbations of minimizers of
various nature, but their study is motivated also by the observation
that the minimizers with certain constrains, such as the ones with
fixed volume or solutions of the obstacle problem, are realized as
almost minimizers of unconstrained problems, see e.g.\
\cites{Anz83}. Yet another motivation is that the study of almost
minimizers reveals a unique perspective on the problem and leads to
the development of methods relying on less technical assumptions, thus
allowing further generalization.

In this paper we extend the notion of almost minimizers to the thin
obstacle problem. Essentially, in \eqref{eq:alm-min-Dir}, we restrict
the function $u$ and its competitors $v$ to stay above the thin
obstacle $\psi$ on $\cM$.

\begin{definition}[Almost minimizer for the thin obstacle (or
  Signorini) problem]
  \label{def:alm-min-Sig}
  Given $r_0>0$ and a gauge function $\omega(r)$ on $(0,r_0)$, we say
  that $u\in W^{1,2}_\loc(D)$ is an \emph{almost minimizer for the
    thin obstacle (\emph{or} Signorini) problem}, if $u\geq\psi$ on
  $\cM\cap D$ and, for any ball $B_r(x_0)\Subset D$ with $0<r<r_0$, we
  have
  \begin{equation}\label{eq:alm-min-Sig}
    J_{B_r(x_0)}(u)\leq (1+\omega(r))J_{B_r(x_0)}(v),\quad\text{for
      any }v\in \mathfrak{K}_{\psi,u}(B_r(x_0),\cM).
  \end{equation}
\end{definition}

Note that in the case when $\cM\cap B_r(x_0)=\emptyset$, the condition
\eqref{eq:alm-min-Sig} is the same as \eqref{eq:alm-min-Dir} and thus
almost minimizers of the Signorini problem are almost harmonic in
$D\setminus \cM$. As in the case of the solutions of the Signorini
problem, we are interested in the regularity properties of almost
minimizers as well as the structure and the regularity of the free
boundary $\Gamma(u)\subset\cM$ as defined in \eqref{eq:free-bdry}.

Some examples of almost minimizers are given in
Appendix~\ref{sec:ex-drift}. We would also like to mention here that a
related notion of almost minimizers for the fractional obstacle
problem has been considered by the authors in \cite{JeoPet19b}.

\subsection{Main results} Because of the technical nature of the
problem, in this paper we restrict ourselves only to the case when
$\omega(r)=r^\alpha$ for some $\alpha>0$, $\cM$ is flat, specifically
$\cM= \R^{n-1}\times\{0\}$, and the thin obstacle $\psi=0$. As we are
mainly interested in local properties of almost minimizers and their
free boundaries, we assume that $D$ is the unit ball $B_1$,
$u\in W^{1,2}(B_1)$, and the constant $r_0=1$ in
Definition~\ref{def:alm-min-Sig}. We also assume that $u$ is even in
$x_n$-variable:
$$
u(x',x_n)=u(x',-x_n),\quad \text{for any }x=(x',x_n)\in B_1.
$$

\medskip Our first main result is then as follows.

\begin{maintheorem}[$C^{1,\beta}$-regularity of almost minimizers]
  \label{mthm:A}
  Let $u$ be an almost minimizer for the Signorini problem in $B_1$,
  under the assumptions above. Then,
  $u\in C^{1,\beta}_\loc(B_1^\pm\cup B_1')$ for
  $\beta=\beta(\alpha,n)$ and
  $$
  \|u\|_{C^{1,\beta}(K)}\leq C\|u\|_{W^{1,2}(B_1)},
$$
for any $K\Subset B_1^\pm\cup B_1'$ and $C=C(n,\alpha,K)$.
\end{maintheorem}
The proof is obtained by using Morrey and Campanato space estimates,
following the original idea of Anzellotti \cite{Anz83}. However, in
our case the proof is much more elaborate and, in a sense, based on
the idea that the solutions of the Signorini problem are $2$-valued
harmonic functions, as we have to work with both even and odd
extensions of $u$ and $\nabla u$ from $B_{1}^+$ to $B_1$.

While the optimal regularity for the minimizer (or solutions) of the
Signorini problem is $C^{1,1/2}$, we do not expect such regularity for
almost minimizers. However, we are able to establish the optimal
growth for almost minimizers, which then allows to study the local
properties of the free boundary
$$
\Gamma(u)=\partial\{u(\cdot,0)=0\}\cap B_1'.
$$
\begin{maintheorem}[Optimal growth near free boundary]
  \label{mthm:B} Let $u$ be as in Theorem~\ref{mthm:A}. Then,
  $$
  \int_{\partial B_r(x_0)} u^2\leq
  C(n,\alpha)\|u\|^2_{W^{1,2}(B_1)}r^{n+2},
  $$
  for $x_0\in B_{1/2}'\cap \Gamma(u)$, $0<r<r_0(n,\alpha)$.
\end{maintheorem}
One of the ingredients in the proof is an Almgren-type monotonicity
formula, which we describe below. For an almost minimizer $u$,
\emph{Almgren's frequency} \cite{Alm00} is defined by
$$
N(r,u,x_0):=\frac{r\int_{B_r(x_0)}|\nabla u|^2}{\int_{\partial
    B_r(x_0)}u^2},\quad x_0\in \Gamma(u).
$$
It is one of the most important monotone quantities in the analysis of
the free boundary for the Signorini problem, see e.g.\
\cite{PetShaUra12}*{Chapter~9}. We show that for almost minimizers a
small modification of $N$ is monotone.

\begin{maintheorem}[Monotonicity of the truncated frequency]
  \label{mthm:C}
  Let $u$ be as in Theorem~\ref{mthm:A}. Then for any
  $\kappa_0\geq 2$, there is $b=b(n,\alpha,\kappa_0)$ such that
$$
r\mapsto\widehat{N}_{\kappa_0}(r,u,x_0) :=
\min\left\{\frac{1}{1-br^\alpha}N(r,u,x_0),\kappa_0\right\}
$$ 
is monotone for $x_0\in B_{1/2}'\cap\Gamma(u)$, and
$0<r<r_0(n,\alpha,\kappa_0)$. Moreover, we have that either
$$
\widehat{N}_{\kappa_0}(0+,u,x_0)=3/2\quad\text{or}\quad
\widehat{N}_{\kappa_0}(0+,u,x_0)\geq 2.
$$
\end{maintheorem}
We give an indirect proof of this fact, based on an one-parametric
family of Weiss-type energy functionals
$\{W_\kappa\}_{0<\kappa<\kappa_0}$, see Theorem~\ref{thm:weiss}, that
go back to the work \cite{GarPet09} for the solutions of the Signorini
problem and Weiss \cite{Wei99b} for the classical obstacle
problem. The fact that $\widehat N\geq 3/2$ at free boundary points is
crucial for the proof of the optimal growth (Theorem~\ref{mthm:B}),
however, the proof of Theorem~\ref{mthm:B} requires also an
application of so-called \emph{epiperimetric inequality} for Weiss's
energy functional $W_{3/2}$ (see \cite{GarPetSVG16}), to remove a
remaining logarithmic term.

Our next result concerns the subset of the free boundary
$$
\mathcal{R}(u):=\{x_0\in \Gamma(u):\widehat{N}(0+,u,x_0)=3/2\},
$$
where Almgren's frequency is minimal, known as the \emph{regular set}
of $u$.

\begin{maintheorem}[Regularity of the regular set]
  \label{mthm:D}
  Let $u$ be as in Theorem~\ref{mthm:A}. Then $\cR(u)$ is a relatively
  open subset of the free boundary $\Gamma(u)$ and is a
  $(n-2)$-dimensional manifold of class $C^{1,\gamma}$.
\end{maintheorem}

Our proof of this theorem is based on the use of the epiperimetric
inequality and is similar to the one for the solutions of the
Signorini problem in \cite{GarPetSVG16}.

Finally, we state our main result for the so-called \emph{singular
  set}. A free boundary point $x_0\in \Gamma(u)$ is called
\emph{singular} if the \emph{coincidence set}
$\Lambda(u):=\{u(\cdot,0)=0\}$ has $H^{n-1}$-density zero at $x_0$,
i.e.,
$$
\lim_{r\to 0+} \frac{H^{n-1}(\Lambda(u)\cap
  B_r'(x_0))}{H^{n-1}(B_r')}=0.
$$
If $\widehat{N}_{\kappa_0}(0+,u,x_0)=\kappa<\kappa_0$, then $x_0$ is
singular if and only if $\kappa=2m$, $m\in\mathbb{N}$ (see
Proposition~\ref{prop:sing-char}). For such $\kappa$, we then define
$$
\Sigma_{\kappa}(u):=\{x_0\in\Gamma(u):\widehat{N}_{\kappa_0}(0+,u,x_0)=\kappa\}.
$$

\begin{maintheorem}[Structure of the singular set]
  \label{mthm:E}
  Let $u$ be as in Theorem~\ref{mthm:A}. Then, for any
  $\kappa=2m<\kappa_0$, $m\in\mathbb{N}$, $\Sigma_\kappa(u)$ is
  contained in a countable union of $(n-2)$-dimensional manifolds of
  class $C^{1,\log}$.
\end{maintheorem}
A more refined version of this theorem is given in
Theorem~\ref{thm:sing}. The proof is based on the \emph{logarithmic
  epiperimetric inequality} of Colombo-Spolaor-Velichkov
\cite{ColSpoVel17} for Weiss's energy functional $W_{\kappa}$, with
$\kappa=2m$, $m\in\mathbb{N}$. We also point out that this inequality
is instrumental in the proof of the optimal growth at singular points,
which is rather immediate for the solutions of the Signorini problem,
but far more complicated for the almost minimizers (see
Lemmas~\ref{lem:W-est}--\ref{lem:opt-est}).

\subsubsection{Proofs of Theorems~\ref{mthm:A}--\ref{mthm:E}} While we
don't give formal proofs of Theorems~\ref{mthm:A}--\ref{mthm:E}, in
the main body of the paper, they follow from the combination of
results there. More specifically,
\begin{enumerate}[label=$\circ$,leftmargin=*,labelindent=\parindent]
\item Theorem~\ref{mthm:A} follows by combining
  Theorems~\ref{thm:holder} and \ref{thm:grad-holder}.
\item The statement of Theorem~\ref{mthm:B} is contained in that of
  Lemma~\ref{lem:opt-growth}.
\item Theorem~\ref{mthm:C} follows by combining
  Theorem~\ref{thm:Almgren} and Corollary~\ref{cor:gap}.
\item The statement of Theorem~\ref{mthm:D} is contained in that of
  Theorem~\ref{thm:C1g-regset}.
\item The statement of Theorem~\ref{mthm:E} is contained in that of
  Theorem~\ref{thm:sing}.
\end{enumerate}

\subsection{Notation}

Throughout the paper we use the following notation. $\R^n$ stands for
the $n$-dimensional Euclidean space. We denote the points of $\R^n$ by
$x=(x', x_n)$, where $x'=(x_1, \ldots, x_{n-1})\in \R^{n-1}$. We
routinely identify $x'\in \R^{n-1}$ with
$(x', 0)\in \R^{n-1}\times \{0\}$. $\R^{n}_\pm$ stand for open
halfspaces $\{x\in\R^n: \pm x_n>0\}$.

For $x\in\R^n$, $r>0$, we use the following notations for balls of
radius $r$, centered at $x$.
\begin{alignat*}{2}
  B_r(x)&=\{y\in \R^n:|x-y|<r\},&\quad&\text{ball in $\R^n$,}\\
  B^{\pm}_r(x')&=B_r(x',0)\cap \{\pm x_n>0\},&& \text{half-ball in $\R^n$,}\\
  B'_r(x')&=B_r(x',0)\cap \{ x_n=0\}, &&\text{ball in $\R^{n-1}$, or
    thin ball.}
\end{alignat*}
We typically drop the center from the notation if it is the
origin. Thus, $B_r=B_r(0)$, $B'_r=B'_r(0)$, etc.

Next, for a direction $e\in\R^n$, we denote
$$
\partial_e u=\nabla u\cdot e,
$$
the directional derivative of $u$ in the direction $e$.  For the
standard coordinate directions $e_i$, $i=1,\ldots,n$, we simply
write $$ u_{x_i}=\pa_{x_i}u= \pa_{e_i}u.
$$
Moreover, by $\pa_{x_n}^{\pm}u(x', 0)$ we mean the limit of
$\pa_{x_n}u$ from within $B_r^{\pm}$, specifically,
\begin{align*}
  \pa^+_{x_n}u(x', 0)&=\lim_{\substack{y\ra (x', 0)\\y\in
  B_r^+}} \pa_{x_n}u(y)=-\pa_{\nu^+}u(x', 0),\\
  \pa^-_{x_n}u(x', 0)&=\lim_{\substack{y\ra (x', 0)\\y\in B_r^-}}\pa_{x_n}u(y)=\pa_{\nu^-}u(x', 0),
\end{align*}
where $\nu^{\pm}=\mp e_n$ are unit outward normal vectors for
$B_r^{\pm}$ on $B_r'$.

In integrals, we often drop the variable and the measure of
integration if it is with respect to the Lebesgue measure or the
surface measure. Thus,
$$
\int_{B_r} u=\int_{B_r} u(x)dx,\quad \int_{\partial B_r}
u=\int_{\partial B_r} u(x)dS_x,
$$
where $S_x$ stands for the surface measure.

We indicate by $\mean{u}_{x, r}$ the integral mean value of a function
$u$ over $B_r(x)$. That is,
$$
\mean{u}_{x, r}:=\dashint_{B_r(x)}u=\frac{1}{\om_nr^n}\int_{B_r(x)}u,
$$ where $\omega_n=|B_1|$ 
is the volume of unit ball in $\R^n$. Similarly to the other
notations, we drop the origin if it is $0$ and write $\mean{u}_r$ for
$\mean{u}_{0,r}$.


\section{Almost harmonic functions}
\label{sec:almost-harm-funct}

In this section we recall some results of Anzellotti \cite{Anz83} on
almost harmonic functions, i.e., almost minimizers of the Dirichlet
integral $J_D(v)=\int_D|\nabla v|^2$. We also state here some of the
relevant auxiliary results that we will need also in the treatment of
almost minimizers for the Signorini problem.

\begin{theorem}\label{thm:Anz} Let $u$ be an almost harmonic function in
  an open set $D$ with a gauge function $\omega$. Then
  \begin{enumerate}[label=\textup{(\roman*)},leftmargin=*,widest=ii]
  \item $u$ is locally almost Lipschitz, i.e.,
    $u\in C^{0, \sigma}_{\loc}(D)$ for all $\sigma\in (0, 1)$.
  \item If $\om(r)\le Cr^{\al}$ for some $\al\in (0, 2)$, then
    $u\in C^{1, \al/2}_{\loc}(D)$.
  \end{enumerate}
\end{theorem}

While we refer to \cite{Anz83} for the full proof of this theorem, we
would like to outline the key steps in Anzellotti's argument. The idea
to prove $C^{0,\sigma}$ and $C^{1,\alpha/2}$ regularity of $u$ is
through the Morrey and Campanato space estimates, namely, by
establishing that
\begin{align}
  \label{eq:Anz-Mor-est} &\int_{B_\rho(x_0)}|\nabla u|^2\leq C \rho^{n-2+2\sigma}\\
  \label{eq:Anz-Camp-est}&\int_{B_\rho(x_0)}|\nabla u -\mean{\nabla u}_{x_0,\rho}|^2\leq C \rho^{n+\alpha}
\end{align}
for $x_0\in K\Subset D$, and $0<\rho<\rho_0$, with $C$ and $\rho_0$
depending on $n$, $r_0$, $d=\dist(K,\partial D)$, the gauge function
$\omega$, and $\|u\|_{W^{1,2}(D)}$.

To obtain the estimates above, one starts by choosing a special
competitor $v$ in \eqref{eq:alm-min-Dir}. Namely, we take $v=h$ which
solves the Dirichlet problem
$$
\Delta h=0\quad\text{in }B_r(x_0),\quad h=u\quad\text{on }\partial
B_r(x_0).
$$
Equivalently, $h$ is the minimizer of the Dirichlet energy
$\int_{B_r(x_0)}|\nabla v|^2$ among all functions in
$u+W^{1,2}_0(B_r(x_0))$. We call this $h$ the \emph{harmonic
  replacement} of $u$ in $B_r(x_0)$.  We then have the following
concentric ball estimates for $h$.
\begin{proposition} Let $h$ be harmonic in $B_r(x_0)$ and
  $0<\rho<r$. Then
  \begin{align}
    \label{eq:harm-Mor-est}&\int_{B_\rho(x_0)}|\nabla h|^2\leq
                             \left(\frac{\rho}{r}\right)^n\int_{B_r(x_0)}|\nabla h|^2\\
    \label{eq:harm-Camp-est}&\int_{B_\rho(x_0)}|\nabla h-\mean{\nabla h}_{x_0,\rho}|^2\leq
                              \left(\frac{\rho}{r}\right)^{n+2}\int_{B_r(x_0)}|\nabla h-\mean{\nabla h}_{x_0,r}|^2.
  \end{align}
\end{proposition}
\begin{proof} The estimates above follow from the monotonicity in
  $\rho$ of the quantities
  $$
  \frac{1}{\rho^n}\int_{B_\rho(x_0)}|\nabla
  h|^2,\quad\frac{1}{\rho^{n+2}}\int_{B_\rho(x_0)}|\nabla
  h-\mean{\nabla h}_{x_0,\rho}|^2.
  $$
  Noticing that $\mean{\nabla h}_{x_0,\rho}=\nabla h(x_0)$, an easy
  proof is obtained by decomposing $h$ into the sum of the series of
  homogeneous harmonic polynomials.
\end{proof}

We next use the almost minimizing property of $u$ to deduce perturbed
versions of the estimates above.
\begin{proposition}\label{prop:Anz-Mor-Camp-est} Let $u$ be an almost
  harmonic function in $D$. Then for any ball $B_r(x_0)\Subset D$ with
  $r<r_0$ and $0<\rho<r$ we have
  \begin{align}
    \label{eq:Anz-Mor-alm-min} &\int_{B_\rho(x_0)}|\nabla u|^2\leq
                                 2\left[\left(\frac{\rho}{r}\right)^n+\omega(r)\right]\int_{B_r(x_0)}|\nabla
                                 u|^2\\
                               &
                                 \label{eq:Anz-Camp-alm-min}\begin{multlined}[t]
                                   \int_{B_{\rho}(x_0)}|\D u-\mean{\D
                                     u}_{x_0, \rho}|^2 \le
                                   9\left(\frac{\rho}{r}\right)^{n+2}\int_{B_r(x_0)}|\D
                                   u-\mean{\D u}_{x_0,
                                     r}|^2\\+24\,\omega(r)\int_{B_r(x_0)}|\D
                                   u|^2.
                                 \end{multlined}
  \end{align}
\end{proposition}
\begin{proof} If $h$ is a harmonic replacement of $u$ in $B_r(x_0)$,
  we first note that
  \begin{align*}
    \int_{B_r(x_0)}|\nabla (u-h)|^2&=\int_{B_r(x_0)}|\nabla u|^2-|\nabla
                                     h|^2-2\int_{B_r(x_0)}\nabla h\nabla(u-h)\\
                                   &=\int_{B_r(x_0)}|\nabla u|^2-|\nabla
                                     h|^2\leq \omega(r)\int_{B_r(x_0)}|\nabla h|^2\leq \omega(r)\int_{B_r(x_0)}|\nabla u|^2. 
  \end{align*}
  Then, combined with \eqref{eq:harm-Mor-est}, we estimate
  \begin{align*}
    \int_{B_\rho(x_0)}|\nabla u|^2&
                                    \leq 2\int_{B_{\rho}(x_0)}|\nabla
                                    h|^2+2\int_{B_{\rho}(x_0)}|\nabla (u-h)|^2\\
                                  &\leq
                                    2\left[\left(\frac{\rho}{r}\right)^n +
                                    \omega(r)\right]\int_{B_r(x_0)}|\nabla u|^2,
  \end{align*}
  which gives \eqref{eq:Anz-Mor-alm-min}. To obtain
  \eqref{eq:Anz-Camp-alm-min}, we argue very similarly by using
  additionally that by Jensen's inequality
$$
\int_{B_\rho(x_0)}|\mean{\nabla u}_{x_0,\rho}-\mean{\nabla
  h}_{x_0,\rho}|^2\leq \int_{B_\rho(x_0)}|\nabla u-\nabla h|^2.
$$
For more details we refer to the proof of
Theorem~\ref{thm:grad-holder}, Case 1.1.
\end{proof}

From here, one deduces the estimates
\eqref{eq:Anz-Mor-est}--\eqref{eq:Anz-Camp-est} with the help of the
following useful lemma. The proof can be found e.g.\ in
\cite{HanLin97}.
\begin{lemma}\label{lem:HL}
  Let $r_0>0$ be a positive number and let
  $\vp:(0,r_0)\to (0, \infty)$ be a nondecreasing function. Let $a$,
  $\beta$, and $\gamma$ be such that $a>0$, $\gamma >\beta >0$. There
  exist two positive numbers $\e=\e(a,\gamma,\beta)$,
  $c=c(a,\gamma,\beta)$ such that, if
$$
\vp(\rho)\le
a\Bigl[\Bigl(\frac{\rho}{r}\Bigr)^{\gamma}+\e\Bigr]\vp(r)+b\, r^{\be}
$$ for all $\rho$, $r$ with $0<\rho\leq r<r_0$, where $b\ge 0$,
then one also has, still for $0<\rho<r<r_0$,
$$
\vp(\rho)\le
c\Bigl[\Bigl(\frac{\rho}{r}\Bigr)^{\be}\vp(r)+b\rho^{\be}\Bigr].
$$
\end{lemma}

We can now give a formal proof of Theorem~\ref{thm:Anz}.

\begin{proof}[Proof of Theorem~\ref{thm:Anz}]
  (i) Taking $r_0$ small enough so that $\omega(r_0)<\e$, a direct
  application of Lemma~\ref{lem:HL} to \eqref{eq:Anz-Mor-alm-min}
  produces the estimate \eqref{eq:Anz-Mor-est}, which in turn implies
  that $u\in C^{0,\sigma}_\loc(D)$, by the Morrey space embedding
  theorem.

  \medskip
  \noindent (ii) Using that $\omega(r)\leq C r^\alpha$, combined with
  the estimate \eqref{eq:Anz-Mor-est}, we first obtain
  \begin{multline*}
    \int_{B_{\rho}(x_0)}|\D u-\mean{\D u}_{x_0, \rho}|^2 \le
    9\left(\frac{\rho}{r}\right)^{n+2}\int_{B_r(x_0)}|\D u-\mean{\D
      u}_{x_0, r}|^2+C r^{n-2+2\sigma+\alpha}.
  \end{multline*}
  If $\sigma$ is so that $\alpha'=-2+2\sigma+\alpha>0$,
  Lemma~\ref{lem:HL} implies that
$$
\int_{B_{\rho}(x_0)}|\D u-\mean{\D u}_{x_0, \rho}|^2 \le C
\rho^{n+\alpha'}.
$$
By the Campanato space embedding, we therefore obtain that
$\nabla u\in C_\loc^{0,\alpha'}(D)$. However, it is easy to bootstrap
the regularity up to $C_\loc^{0,\alpha/2}$ by noticing that we now
know that $\nabla u$ is locally bounded in $D$ and thus
$\int_{B_r(x_0)}|\nabla u|^2\leq C r^n$. Plugging that in the last
term of \eqref{eq:Anz-Camp-alm-min}, we obtain that
\begin{multline*}
  \int_{B_{\rho}(x_0)}|\D u-\mean{\D u}_{x_0, \rho}|^2 \le
  9\left(\frac{\rho}{r}\right)^{n+2}\int_{B_r(x_0)}|\D u-\mean{\D
    u}_{x_0, r}|^2+C r^{n+\alpha}
\end{multline*}
and repeating the arguments above conclude that
$u\in C^{1,\alpha/2}_\loc$.
\end{proof}


\section{Almost Lipschitz regularity of almost minimizers}
\label{sec:almost-lipsch-regul}

In this section we prove the first regularity results for the almost
minimizers for the Signorini problem, see
Definition~\ref{def:alm-min-Sig}. Recall that we assume $D=B_1$,
$\cM=\R^{n-1}\times\{0\}$, $\psi=0$, $r_0=1$, and $\omega(r)=r^\alpha$
for some $\alpha>0$. Furthermore we assume that $u$ is even symmetric
in $x_n$-variable.

\begin{theorem}\label{thm:holder}
  Let $u$ be an almost minimizer for the Signorini problem in $B_1$.
  Then $u\in C^{0, \sigma}(B_1)$ for all $0<\sigma<1$. Moreover, for
  any $K\Subset B_1$,
  \begin{align}
    \|u\|_{C^{0, \sigma}(K)} &\le C\|u\|_{W^{1, 2}(B_1)} \label{eq:holder-6}
  \end{align}
  with $C=C(n, \al, \sigma, K)$.
\end{theorem}

The idea of the proof is to follow that of Anzellotti \cite{Anz83}
that we outlined in Section~\ref{sec:almost-harm-funct} and to prove
an estimate similar to \eqref{eq:Anz-Mor-alm-min}. The proof of the
latter estimate followed by a perturbation argument from a similar
estimate for the harmonic replacement of $u$. However, in the case of
the Signorini problem, the harmonic replacements are not necessarily
admissible competitors. Instead, for $B_r(x_0)\Subset B_1$, we
consider the \emph{Signorini replacements} $h$ of $u$ in $B_r(x_0)$,
which solve the Signorini problem in $B_r(x_0)$ with the thin obstacle
$0$ on $\cM$ and boundary values $h=u$ on $\partial
B_r(x_0)$. Equivalently, Signorini replacements are the minimizers of
$J_{B_r(x_0)}$ on the constraint set
$\mathfrak{K}_{0,u}(B_r(x_0),\cM)$ and they also satisfy the
variational inequality\footnote{which follows from the inequality
  $ \int_{B_r(x_0)}|\nabla h|^2\leq \int_{B_r(x_0)}|\nabla( (1-\e)h+\e
  v)|^2$, $\e\in (0,1)$ by a first variation argument.}
\begin{equation}\label{eq:Sig-repl-var-ineq}
  \int_{B_r(x_0)}\nabla h\cdot\nabla(v-h)\geq 0\quad\text{for any }v\in\mathfrak{K}_{0,u}(B_r(x_0),\cM).
\end{equation}
We then have the following concentric ball estimates for Signorini
replacements similar to the one for harmonic replacements, at least
when the center of the balls is on $\cM=\R^{n-1}\times\{0\}$.

\begin{proposition}\label{prop:Sig-Mor-est}
  Let $x_0\in\cM$ and let $h$ be a solution of the Signorini problem
  in $B_r(x_0)$ with zero obstacle on $\cM$, even in
  $x_n$-variable. Then,
  \begin{align}\label{eq:holder-4}\int_{B_{\rho}(x_0)}|\D h|^2\le
    \Bigl(\frac{\rho}{r}\Bigr)^n\int_{B_r(x_0)}|\D h|^2,\quad
    0<\rho<r.
  \end{align}
\end{proposition}
\begin{proof} We claim that $|\D h|^2$ is subharmonic in
  $B_r(x_0)$. This follows from the fact that $h_{x_i}^\pm$,
  $i=1,\ldots,n-1$, are subharmonic in $B_r(x_0)$, see
  \cite{PetShaUra12}, and similarly that the even extensions
  $\widetilde h_{x_n}^\pm $ of $h_{x_n}^\pm$ in $x_n$-variable from
  $B_R^+(x_0)$ to all of $B_R(x_0)$ are also subharmonic. These are
  all consequences of the fact that a continuous nonnegative function,
  subharmonic in its positivity set is subharmonic, see
  \cite{PetShaUra12}*{Ex.~2.6}.

  The subharmonicity of $|\nabla h|^2$ in $B_r(x_0)$ then implies, by
  the sub-mean value property, that the function
$$
\rho\mapsto \frac{1}{\rho^n}\int_{B_\rho(x_0)}|\nabla h|^2
$$
is monotone nondecreasing. This readily implies \eqref{eq:holder-4}.
\end{proof}

We next have the perturbed version of
Proposition~\ref{prop:Sig-Mor-est}.

\begin{proposition}
  \label{prop:alm-min-Sig-Mor-est}
  Let $u$ be an almost minimizer for the Signorini problem in $B_1$,
  and $B_r(x_0)\subset B_1$. Then, there is $C_1=C_1(n)>1$ such that
  \begin{equation} \int_{B_{\rho}(x_0)}|\D u|^2 \le
    C_1\left[\left(\frac{\rho}{r}\right)^n
      +r^{\al}\right]\int_{B_r(x_0)}|\D u|^2,\quad
    0<\rho<r.\label{eq:holder-1}
  \end{equation}
\end{proposition}

\begin{proof}
  By using the continuity argument, we may assume that
  $B_r(x_0)\Subset B_1$. We first prove the estimate when $x_0$ is in
  the thin space, i.e.,\ $x_0\in B_1'$ and then extend it to arbitrary
  $x_0\in B_1$.

  \medskip
  \noindent
  \emph{Case 1.} Suppose $x_0\in B'_1$ and let $h$ be the Signorini
  replacement of $u$ in $B_r(x_0)$. Recall that $h$ satisfies
  \eqref{eq:Sig-repl-var-ineq}. Then, plugging $v=u$, we obtain
  \begin{equation}\label{eq:holder-2}
    \int_{B_{r}(x_0)}\nabla h\cdot\nabla u-|\nabla h|^2\geq 0.
  \end{equation}
  Using this, we can estimate
  \begin{equation}
    \begin{split}
      \int_{B_r(x_0)}|\D(u-h)|^2&= \int_{B_r(x_0)}\Bigl(|\D u|^2+|\D h|^2-2\D u\cdot \D h\Bigr)\\
      &\le \int_{B_r(x_0)}|\D u|^2-\int_{B_r(x_0)}|\D h|^2 \\
      &\le \big(1+r^{\al}\big)\int_{B_r(x_0)}|\D h|^2-\int_{B_r(x_0)}|\D h|^2 \\
      &= r^{\al}\int_{B_r(x_0)}|\D h|^2 \le r^{\al}\int_{B_r(x_0)}|\D
      u|^2,\label{eq:holder-3}
    \end{split}
  \end{equation}
  where in the very last step we have used that $h$ minimizes the
  Dirichlet integral among all functions in
  $\mathfrak{K}_{0,u}(B_r(x_0),\cM)$.

  Next, we use the same perturbation argument as in the proof of
  \eqref{eq:Anz-Mor-alm-min}.  By using \eqref{eq:holder-4} and
  \eqref{eq:holder-3}, we estimate
  \begin{align*}
    \int_{B_{\rho}(x_0)}|\D u|^2
    &\le 2\int_{B_{\rho}(x_0)}|\D h|^2+2\int_{B_{\rho}(x_0)}|\D(u-h)|^2 \\
    &\le 2\Bigl(\frac{\rho}{r}\Bigr)^n\int_{B_r(x_0)}|\D h|^2+2r^{\al}\int_{B_r(x_0)}|\D u|^2\\
    &\le2\Bigl[\Bigl(\frac{\rho}{r}\Bigr)^n+r^{\al}\Bigr]\int_{B_r(x_0)}|\D u|^2. 
  \end{align*}
  Thus, \eqref{eq:holder-1} follows in this case.

  \medskip\noindent \emph{Case 2.} Consider now the case
  $x_0 \in B_1^+$. If $\rho\ge r/4$, then we simply have
$$
\int_{B_{\rho}(x_0)}|\D u|^2 \le
4^n\left(\frac{\rho}{r}\right)^n\int_{B_r(x_0)}|\D u|^2.$$ Thus, we
may assume $\rho<r/4$. Then, let $d:=\dist(x_0, B'_1)>0$ and choose
$x_1\in \pa B_d(x_0)\cap B'_1$.

\medskip\noindent \emph{Case 2.1.} If $\rho\ge d$, then we use
$B_{\rho}(x_0)\subset B_{2\rho}(x_1)\subset B_{r/2}(x_1)\subset
B_r(x_0)$ and the result of Case 1 to write
\begin{align*}
  \int_{B_{\rho}(x_0)}|\D u|^2
  &\le \int_{B_{2\rho}(x_1)}|\D u|^2
    \le C\left[\left(\frac{2\rho}{r/2}\right)^n +
    (r/2)^{\al}\right]\int_{B_{r/2}(x_1)}|\D u|^2\\
  &\le
    C\left[\left(\frac{\rho}{r}\right)^n+r^{\al}\right]\int_{B_r(x_0)}|\D u|^2.
\end{align*}

\medskip\noindent \emph{Case 2.2.}  Suppose now $d>\rho$. If $d>r$,
then $B_r(x_0)\Subset B_1^+$. Since $u$ is almost harmonic in $B_1^+$,
we can apply Proposition~\ref{prop:Anz-Mor-Camp-est} to obtain
$$
\int_{B_{\rho}(x_0)}|\D u|^2\le2\left[\left(\frac \rho
    r\right)^n+r^{\al}\right]\int_{B_r(x_0)}|\D u|^2.
$$
Thus, we may assume $d\le r$. Then we note that
$B_d(x_0)\subset B_1^+$ and by a limiting argument from the previous
estimate, we obtain
$$
\int_{B_{\rho}(x_0)}|\D u|^2 \le 2\left[\left(\frac \rho
    d\right)^n+r^{\al}\right]\int_{B_d(x_0)}|\D u|^2.
$$

\medskip\noindent \emph{Case 2.2.1.} If $r/4\le d$, then
$$
\int_{B_d(x_0)}|\D u|^2 \le
4^n\left(\frac{d}{r}\right)^n\int_{B_r(x_0)}|\D u|^2,
$$ 
which immediately implies \eqref{eq:holder-1}.

\medskip\noindent \emph{Case 2.2.2.}  It remains to consider the case
$\rho<d<r/4$. Using Case 1 again, we have
\begin{align*}
  \int_{B_d(x_0)}|\D u|^2
  &\le \int_{B_{2d}(x_1)}|\D u|^2 \le C\left[\left(\frac{2d}{r/2}\right)^n+(r/2)^{\al}\right]\int_{B_{r/2}(x_1)}|\D u|^2 \\
  &\le C\left[\left(\frac{d}{r}\right)^n+r^{\al}\right]\int_{B_r(x_0)}|\D u|^2,
\end{align*}
which also implies \eqref{eq:holder-1}.  This concludes the proof of
the proposition.
\end{proof}

We can now give the proof of the almost Lipschitz regularity of almost
minimizers.

\begin{proof}[Proof of Theorem~\ref{thm:holder}]
  Let $K\Subset B_1$ and $x_0\in K$. Take
  $\delta=\delta(n, \al,\sigma, K)>0$ such that
  $\delta<\dist(K, \pa B_1)$ and
  $\delta^{\al}\le \e(C_1, n, n+2\sigma-2)$, where
  $\e=\e(C_1, n, n+2\sigma-2)$ is as in Lemma~\ref{lem:HL}. Then for
  all $0<\rho<r<\delta$, by \eqref{eq:holder-1},
  \begin{align*}
    \int_{B_{\rho}(x_0)}|\D u|^2 &\le  C_1\Bigl[\Bigl(\frac{\rho}{r}\Bigr)^n+\e\Bigr]\int_{B_r(x_0)}|\D u|^2.
  \end{align*}
  By applying Lemma~\ref{lem:HL}, we obtain
  \begin{align*}
    \int_{B_{\rho}(x_0)}|\D u|^2
    &\le C(n, \sigma)\Bigl(\frac{\rho}{r}\Bigr)^{n+2\sigma-2}\int_{B_{r}(x_0)}|\D u|^2.
  \end{align*}
  Taking $r\nearrow \delta$, we can therefore conclude
  \begin{align}
    \int_{B_{\rho}(x_0)}|\D u|^2 &\le C(n, \al, \sigma, K)\|\D u\|_{L^2(B_1)}^2\rho^{n+2\sigma-2}.\label{eq:holder-5}
  \end{align}
  From here, we use the Morrey space embedding to obtain
  $u\in C^{0, \sigma}(K)$ with the norm estimate
  \begin{align*}
    \|u\|_{C^{0, \sigma}(K)} &\le C(n, \al, \sigma,K)\|u\|_{W^{1, 2}(B_1)},
  \end{align*}
  as required.
\end{proof}


\section{$C^{1, \beta}$ regularity of almost minimizers}
\label{sec:c1-beta-regularity}

In this section we establish the $C^{1, \be}$ regularity of almost
minimizers for some $\beta>0$. The idea is again to use Signorini
replacements and an appropriate version of the concentric ball
estimate \eqref{eq:harm-Camp-est} for solutions of the Signorini
problem.

As we saw in the proof of the almost Lipschitz regularity of almost
minimizers, it is enough to obtain such estimates when balls are
centered at $x_0$ on the thin space $\cM=\R^{n-1}\times\{0\}$. It
turns out that to prove a proper version of \eqref{eq:harm-Camp-est},
we have to work with both even and odd extensions in $x_n$-variable of
Signorini replacements $h$ from $B_r^+(x_0)$ to $B_r(x_0)$. The reason
is that even extensions are harmonic across the positivity set
$\{h(\cdot,0)>0\}$, while the odd extensions are harmonic across the
interior of the coincidence set $\{h(\cdot,0)=0\}$.

We start with the borderline case when the center $x_0$ of concentric
balls is on the free boundary
$\Gamma(h)= B_r'(x_0)\cap\partial_{\R^{n-1}}\{h(\cdot,0)=0\}$.

\begin{proposition}\label{prop:sig-est-free-bdry}
  Let $h$ be a solution of the Signorini problem in $B_r(x_0)$ with
  $x_0\in \cM$, even in $x_n$, and define
  $$
  \h(x',x_n):=\begin{cases}\phantom{-}h(x', x_n),& x_n\ge 0 \\-h(x',
    -x_n), &x_n<0,
  \end{cases}
  $$
  the odd extension in $x_n$-variable of $h$ from $B_r^+(x_0)$ to
  $B_r(x_0)$.
  
  Suppose that $x_0\in\Gamma(h)$. Then, for any $0<\alpha<1$, there
  exists $C=C(n, \al)$ such that for any $0<\rho<s<(3/4)r$ we have
  \begin{align}
    \begin{multlined}[t]
      \int_{B_{\rho}(x_0)}|\D h-\mean{\D h}_{x_0, \rho}|^2 \le
      \Bigl(\frac{\rho}{s}\Bigr)^{n+\al}\int_{B_s(x_0)}|\D h-\mean{\D
        h}_{x_0, s}|^2\\+C\biggl(\int_{B_r(x_0)}h^2\biggr)
      \frac{s^{n+1}}{r^{n+3}},
    \end{multlined}\\
    \begin{multlined}[t]\int_{B_{\rho}(x_0)}|\D\h-\mean{\D\h}_{x_0,
        \rho}|^2 \le
      \Bigl(\frac{\rho}{s}\Bigr)^{n+\al}\int_{B_s(x_0)}|\D\h-\mean{\D\h}_{x_0,
        s}|^2\\+C\biggl(\int_{B_r(x_0)}h^2\biggr)
      \frac{s^{n+1}}{r^{n+3}}.
    \end{multlined}
  \end{align}
\end{proposition}
\begin{remark} We note here that $\mean{\nabla h}_{x_0,\rho}$ has its
  $n$-th component zero because of odd symmetry of $h_{x_n}$, while
  $\mean{\nabla \h}_{x_0,\rho}$ has its first $(n-1)$ components zero
  because of odd symmetry of $\h_{x_i}$, $i=1,\ldots,n-1$.
\end{remark}

\begin{proof} Without loss of generality we may assume $x_0=0$.  For
  $0<t<(3/4)r$, define
  $$
  \vp(t):=\frac{1}{t^{n+\al}}\int_{B_t}|\D h-\mean{\D h}_{t}|^2.
  $$
  Then
  \begin{align*}
    \vp(t) &= \frac{1}{t^{n+\al}}\int_{B_t}|\D h-\mean{\D h}_t|^2\\
           &= \frac{1}{t^{n+\al}}\biggl[\int_{B_t}|\D h|^2-2\mean{\D h}_t\int_{B_t}\D h+\int_{B_t}\mean{\D h}_t^2\biggr]\\
           &= \frac{1}{t^{n+\al}}\biggl[ \int_{B_t}|\D h|^2-\frac{1}{\omega_nt^n}\Bigl(\int_{B_t}\D h\Bigr)^2\biggr].
  \end{align*}
  Thus,
  \begin{multline}\label{eq:sig-est-free-bdry-1}
    \vp'(t) = \frac{1}{t^{n+\alpha}}\biggl[-\frac{n+\al}{t}\int_{B_t}|\D h|^2+\int_{\pa B_t}|\D h|^2\\
    +\frac{2n+\al}{\omega_nt^{n+1}}\Bigl(\int_{B_t}\D
    h\Bigr)^2-\frac{2}{\omega_nt^{n}}\Bigl( \int_{B_t}\D
    h\Bigr)\Bigl(\int_{\pa B_t}\D h\Bigr)\biggr].
  \end{multline}
  We next recall that $h$ is $C^{1,1/2}$ regular in $B_r^\pm\cup B'_r$
  and we have the estimate
  \begin{align}\label{eq:sig-est-free-bdry-2}
    \|\D h\|_{C^{0,1/2}\left(B_{(3/4)r}^\pm\cup B'_{(3/4)r}\right)} \le C(n) r^{-\frac{n+3}{2}}\|h\|_{L^2(B_r^+)},
  \end{align}
  see e.g.\ Theorem~9.13 in \cite{PetShaUra12}. Then, using
  $\D h(0)=0$, we obtain
  \begin{align*}
    \frac{n+\al}{t^{n+\al+1}}\int_{B_t}|\D h|^2
    &\le \frac{C(n, \al) }{t^{\al}r^{n+3}}\Bigl(\int_{B_r}h^2\Bigr).
  \end{align*}
  We can similarly estimate the other term with a negative sign in
  \eqref{eq:sig-est-free-bdry-1} to obtain
  \begin{align*}
    \vp'(t) &\ge -\frac{C(n, \al)}{t^{\al}r^{n+3}}\Bigl(\int_{B_r}h^2\Bigr).
  \end{align*}
  Thus,
  \begin{align*}
    \vp(s)-\vp(\rho) &\ge -C\Bigl(\int_{B_r}h^2\Bigr)\frac{1}{r^{n+3}}\int_{\rho}^{s}t^{-\al}dt\\
                     &\ge -C\Bigl(\int_{B_r}h^2\Bigr)\frac{s^{1-\al}}{r^{n+3}}.
  \end{align*}
  Therefore
  \begin{align*}
    \int_{B_{\rho}}|\D h-\mean{\D h}_{\rho}|^2
    &= \rho^{n+\al}\vp(\rho)\le \rho^{n+\al}\biggl(\vp(s)+C\Bigl(\int_{B_r}h^2\Bigr)\frac{s^{1-\al}}{r^{n+3}}\biggr)\\
    &\le \Bigl(\frac{\rho}{s}\Bigr)^{n+\al}\int_{B_s}|\D h-\mean{\D h}_s|^2+C\Bigl(\int_{B_r}h^2\Bigr)\frac{s^{n+1}}{r^{n+3}}.
  \end{align*}
  This proves the first estimate in the statement of the lemma. For
  the second estimate, involving $\D\h$, we essentially repeat the
  above argument with
  \[
    \widetilde{\vp}(t):=\frac{1}{t^{n+\al}}\int_{B_t}|\D\h-\mean{\D\h}_t|^2.\qedhere
  \]
\end{proof}

We next consider the case when the center $x_0\notin \Gamma(h)$. We
have to distinguish the cases $x_0$ is in $\{h(\cdot,0)>0\}$ or the
interior of $\{h(\cdot,0)=0\}$.

\begin{proposition}\label{prop:sig-est-not-free-bdry}
  Let $h$ be a solution of the Signorini problem in $B_r(x_0)$ with
  $x_0\in\cM$, even in $x_n$-variable. Suppose
  $x_0\in B_r'(x_0)\setminus \Gamma(h)$.  Let
  $d:=\dist(x_0, \Gamma(h))>0$ and fix $\al\in(0, 1)$. Then there are
  $C_1=C_1(n, \al)$, $C_2=C_2(n, \al)$ such that for $0<\rho<s<r$ the
  following inequalities hold.
  \begin{enumerate}[label=\textup{(\roman*)},leftmargin=*,widest=ii,align=left]
  \item If $B'_d(x_0)\subset \{h(\cdot, 0)>0\}$, then
    \begin{multline*}
      \int_{B_{\rho}(x_0)}|\D h-\mean{\D h}_{x_0, \rho}|^2 \le
      C_1\Bigl(\frac{\rho}{s}\Bigr)^{n+\al}\int_{B_s(x_0)}|\D
      h-\mean{\D h}_{x_0, s}|^2\\+C_2
      \Bigl(\int_{B_r(x_0)}h^2\Bigr)\frac{s^{n+1}}{r^{n+3}}.
    \end{multline*}
  \item If $B'_d(x_0) \subset \{h(\cdot, 0)=0\}$, then
    \begin{multline*}
      \int_{B_{\rho(x_0)}}|\D\h-\mean{\D\h}_{x_0, \rho}|^2 \le
      C_1\Bigl(\frac{\rho}{s}\Bigr)^{n+\al}\int_{B_s(x_0)}|\D\h-\mean{\D\h}_{x_0,s}|^2\\
      +C_2 \Bigl(\int_{B_r(x_0)}h^2\Bigr)\frac{s^{n+1}}{r^{n+3}}.
    \end{multline*}
  \end{enumerate}
\end{proposition}

\begin{proof}
  Without loss of generality we may assume $x_0=0$.

  \medskip\noindent (i) First consider the case
  $B'_d\subset \{h(\cdot, 0)>0\}$. If $d\ge s$, then $h$ is harmonic
  in $B_s$ and hence
  \begin{align*}
    \int_{B_{\rho}}|\D h-\mean{\D h}_{\rho}|^2 &\le\Bigl(\frac{\rho}{s}\Bigr)^{n+2}\int_{B_s}|\D h-\mean{\D h}_s|^2\le\Bigl(\frac{\rho}{s}\Bigr)^{n+\al}\int_{B_s}|\D h-\mean{\D h}_s|^2.
  \end{align*}
  We can therefore assume $0<d<s$
 
  \medskip\noindent\emph{Case 1.} $s/4\le d<s$.

  \medskip\noindent \emph{Case 1.1.} Suppose $0<\rho<d$. We first
  observe that
$$
\int_{B_d}|\D h-\mean{\D h}_d|^2 = \min_{C \in \R^n}\int_{B_d}|\D
h-C|^2 \le \int_{B_d}|\D h-\mean{\D h}_s|^2 \le \int_{B_s}|\D
h-\mean{\D h}_s|^2.
$$
Now using that $h$ is harmonic in $B_d$, we obtain
\begin{align*}
  \int_{B_{\rho}}|\D h-\mean{\D h}_{\rho}|^2
  &\le \Bigl(\frac{\rho}{d}\Bigr)^{n+2}\int_{B_d}|\D h-\mean{\D h}_d|^2\\
  &\le \Bigl(\frac{4\rho}{s}\Bigr)^{n+\al}\int_{B_s}|\D h-\mean{\D h}_s|^2\\
  &= 4^{n+\al}\Bigl(\frac{\rho}{s}\Bigr)^{n+\al}\int_{B_s}|\D h-\mean{\D h}_s|^2.
\end{align*}

\medskip\noindent \emph{Case 1.2.} If $\rho \ge d$, then we use
$\rho/s\ge 1/4$ to obtain
\begin{align*}
  \int_{B_{\rho}}|\D h-\mean{\D h}_{\rho}|^2
  &\le \int_{B_s}|\D h-\mean{\D h}_s|^2\\
  &\le 4^{n+\al}\Bigl(\frac{\rho}{s}\Bigr)^{n+\al}\int_{B_s}|\D h-\mean{\D h}_s|^2.
\end{align*}

\medskip\noindent \emph{Case 2.} $0<d<s/4$.

\medskip\noindent \emph{Case 2.1.} If $0<\rho<d$, take
$x_1\in \pa B'_d\cap\Gamma(h)$. We first use the harmonicity of $h$ in
$B_d$ and then applying Proposition~\ref{prop:sig-est-free-bdry} for
balls $B_{2d}(x_1) \subset B_{s/2}(x_1)\subset B_{(2/3)r}(x_1)$, to
obtain
\begin{align*}
  \int_{B_{\rho}}|\D h-\mean{\D h}_{\rho}|^2
  &\le \Bigl(\frac{\rho}{d}\Bigr)^{n+\al}\int_{B_d}|\D h-\mean{\D h}_d|^2\\
  &\le \Bigl(\frac{\rho}{d}\Bigr)^{n+\al}\int_{B_{2d}(x_1)}|\D h-\mean{\D h}_{x_1, 2d}|^2\\
  &\begin{multlined}\le \Bigl(\frac{\rho}{d}\Bigr)^{n+\al}\biggl[ \Bigl(\frac{4d}{s}\Bigr)^{n+\al}\int_{B_{s/2}(x_1)}|\D h-\mean{\D h}_{x_1, s/2}|^2\\+ C_2 \Bigl(\int_{B_{(2/3)r}(x_1)}h^2\Bigr)\frac{s^{n+1}}{r^{n+3}}\biggr]
  \end{multlined}\\
  &\le C\Bigl(\frac{\rho}{s}\Bigr)^{n+\al}\int_{B_s}|\D h-\mean{\D
    h}_s|^2+C_2 \Bigl(\int_{B_r}h^2\Bigr)\frac{s^{n+1}}{r^{n+3}},
\end{align*}
where is the last step we have used that $B_{s/2}(x_1)\subset B_s$ and
$B_{(2/3)r}(x_1)\subset B_r$.

\medskip\noindent \emph{Case 2.2.} If $d\le \rho<s/4$, then we apply
Proposition~\ref{prop:sig-est-free-bdry} with
$B_{2\rho}(x_1)\subset B_{s/2}(x_1)\subset B_{(2/3)r}(x_1)$ as in Case
2.1:
\begin{align*}
  \int_{B_{\rho}}|\D h-\mean{\D h}_{\rho}|^2
  &\le \int_{B_{2\rho}(x_1)}|\D h-\mean{\D h}_{x_1, 2\rho}|^2\\
  &\begin{multlined}\le
    \Bigl(\frac{4\rho}{s}\Bigr)^{n+\al}\int_{B_{s/2}(x_1)}|\D
    h-\mean{\D h}_{x_1, s/2}|^2\\+C_2
    \Bigl(\int_{B_{(2/3)r}(x_1)}h^2\Bigr)\frac{s^{n+1}}{r^{n+3}}
  \end{multlined}\\
  &\le C\Bigl(\frac{\rho}{s}\Bigr)^{n+\al}\int_{B_s}|\D h-\mean{\D
    h}_s|^2+C_2 \Bigl(\int_{B_r}h^2\Bigr)\frac{s^{n+1}}{r^{n+3}}.
\end{align*}

\medskip\noindent \emph{Case 2.3.} If $\rho\ge s/4$, then
$\rho/s\ge 1/4$ and can we repeat the argument in Case 1.2.

\medskip This completes the proof of part (i).

\medskip\noindent (ii) Now suppose that $h=0$ in $B_d'(x_0)$. Notice
that in the proof of part (i) only harmonicity of $h$ in $B_d$ and
Proposition~\ref{prop:sig-est-free-bdry} were used. In the present
case, it is the odd reflection $\h$ that is harmonic in $B_d$, and we
can repeat the same arguments as in part (i) with $\D h$ replaced by
$\D\h$.
\end{proof}

Now we have the following estimate combining the two preceding ones.

\begin{proposition}\label{prop:sig-est} Let $h$ be a solution of the
  Signorini problem in $B_r(x_0)$ with $x_0\in\cM$, even in
  $x_n$-variable.  Define
$$
\wDh:=
\begin{cases}\D h(x', x_n),& x_n\ge 0 \\
  \D h(x', -x_n), &x_n<0,
\end{cases}
$$
the even extension of $\D h$ from $B_r^+(x_0)$ to $B_r(x_0)$.  Then
for $0<\al <1$, there are $C_1=C_1(n, \al)$, $C_2=C_2(n, \al)$ such
that for all $0<\rho <s\le (3/4)r$,
\begin{multline*}\int_{B_{\rho}(x_0)}|\wDh-\mean{\wDh}_{x_0, \rho}|^2
  \le C_1
  \Bigl(\frac{\rho}{s}\Bigr)^{n+\al}\int_{B_s(x_0)}|\wDh-\mean{\wDh}_{x_0,
    s}|^2\\+C_2\biggl(\int_{B_r(x_0)}h^2\biggr)
  \frac{s^{n+1}}{r^{n+3}}.
\end{multline*}
\end{proposition}
\begin{remark} We explicitly warn the reader that $\D\h$ should not be
  confused with $\wDh$. In $\D\h$, the first $n-1$ components are odd
  and the last one is even in $x_n$-variable, while in $\wDh$ all
  components are even in $x_n$.
\end{remark}

\begin{proof} Let $d:=\dist(x_0, \Gamma(h))$. Without loss of
  generality we may assume that $d>0$, as the case $d=0$ will follow
  by continuity. Also, without loss of generality, assume $x_0=0$.

  \medskip\noindent (i) First consider the case when
  $B'_d \subset \{h>0\}$. By the odd symmetry of $h_{x_n}$ in
  $x_n$-variable, we have $\mean{h_{x_n}}_{\rho}=0$. Thus, for
$$
\widehat{h_{x_n}}(x)=\begin{cases}h_{x_n}(x', x_n), &x_n\ge 0 \\
  h_{x_n}(x', -x_n),& x_n<0 \end{cases},
$$
we obtain
\begin{align*}
  \int_{B_{\rho}}|\widehat{h_{x_n}}-\langle{\widehat{h_{x_n}}}\rangle_{\rho}|^2&= \int_{B_{\rho}}\widehat{h_{x_n}}^2-\frac{1}{|B_{\rho}|}\biggl( \int_{B_{\rho}}\widehat{h_{x_n}}\biggr)^2\\
                                                                               &= \int_{B_{\rho}}|h_{x_n}-\langle{h_{x_n}}\rangle_{\rho}|^2-\frac{1}{|B_{\rho}|}\biggl(\int_{B_{\rho}}\widehat{h_{x_n}}\biggr)^2.
\end{align*}
Further, if $\widehat{h_{x_i}}$ denotes the $i$-th component of
$\widehat{\nabla h}$, we have $\widehat{h_{x_i}}=h_{x_i}$ for
$i=1,\ldots, n-1$, and hence arrive at
\begin{align}
  \int_{B_{\rho}}|\wDh-\mean{\wDh}_{\rho}|^2
  &=\int_{B_{\rho}}|\D h-\mean{\D h}_{\rho}|^2-\frac{1}{|B_{\rho}|}\biggl(\int_{B_{\rho}}\widehat{h_{x_n}}\biggr)^2.\label{lem:sig-est-1}
\end{align}
Similarly, we have
\begin{align}
  \int_{B_{s}}|\wDh-\mean{\wDh}_s|^2&=\int_{B_s}|\D h-\mean{\D
                                      h}_s|^2-\frac{1}{|B_s|}\biggl(\int_{B_s}\widehat{h_{x_n}}\biggr)^2.\label{lem:sig-est-2}
\end{align}
Then, by \eqref{lem:sig-est-1}, \eqref{lem:sig-est-2}, and
Proposition~\ref{prop:sig-est-not-free-bdry}, we obtain
\begin{align*}
  &\int_{B_{\rho}}|\wDh-\mean{\wDh}_{\rho}|^2 \le \int_{B_{\rho}}|\D h-\mean{\D h}_{\rho}|^2\\
  &\qquad\le C_1\Bigl(\frac{\rho}{s}\Bigr)^{n+\al}\int_{B_s}|\D h-\mean{\D h}_{s}|^2+C_2\biggl(\int_{B_r}h^2\biggr) \frac{s^{n+1}}{r^{n+3}}\\
  &\qquad\le C_1\Bigl(\frac{\rho}{s}\Bigr)^{n+\al}\int_{B_s}|\wDh-\mean{\wDh}_{s}|^2 + \frac{C_1}{|B_s|}\biggl(\int_{B_s}\widehat{h_{x_n}}\biggr)^2 + C_2\biggl(\int_{B_r}h^2\biggr) \frac{s^{n+1}}{r^{n+3}}.
\end{align*}
From $h(0)>0$, we have $h_{x_n}(0)=0$. Thus, using
\eqref{eq:sig-est-free-bdry-2}, we obtain
\begin{align*}
  \frac{1}{|B_s|}\biggl(\int_{B_s}\widehat{h_{x_n}}\biggr)^2 &\le C\biggl(\int_{B_r}h^2\biggr)\frac{s^{n+1}}{r^{n+3}}.
\end{align*}
This completes the proof in this case.

\medskip\noindent (ii) Suppose now $B_d'\subset\{h=0\}$. In this case,
we use Proposition~\ref{prop:sig-est-free-bdry} for $\D\h$, which
differs from $\wDh$ in the first $(n-1)$ components by their symmetry,
and has the same even $n$-th component. Arguing as above, we obtain
error terms
$$
\frac{1}{|B_s|}\biggl(\int_{B_s}h_{x_i}\biggr)^2,\quad i=1,\ldots,n-1,
$$
up to a factor of $C(n,\alpha)$.  Then, using that $h_{x_i}(0)=0$,
$i=1,\ldots,n-1$ and \eqref{eq:sig-est-free-bdry-2}, we conclude that
\begin{align*}
  \frac{1}{|B_s|}\biggl(\int_{B_s}{h_{x_i}}\biggr)^2 &\le C\biggl(\int_{B_r}h^2\biggr)\frac{s^{n+1}}{r^{n+3}}.
\end{align*}
This completes the proof.
\end{proof}

We now prove the $C^{1, \be}$ regularity of almost minimizers.

\begin{theorem}\label{thm:grad-holder}
  Let $u$ be an almost minimizer of the Signorini problem in
  $B_1$. Define
$$
\wDu(x', x_n):=\begin{cases}\D u(x', x_n), &x_n\ge 0\\
  \D u(x', -x_n),&x_n<0.
\end{cases}
$$
Then
$$
\wDu\in C^{0, \be}(B_1) \quad\text{with } \be=\frac{\al}{4(2n+\al)}.
$$
Moreover, for any $K\Subset B_1$ there holds
\begin{equation}\label{eq:grad-holder-1}
  \|\wDu\|_{C^{0, \be}(K)} \le C(n, \al, K)\|u\|_{W^{1, 2}(B_1)}.
\end{equation}
\end{theorem}

\begin{proof}
  Without loss of generality, we may assume that $K$ is a ball
  centered at $0$. Fix a small $r_0=r_0(n, \al, K)>0$ to be chosen
  later. Particularly, we will ask
  $R_0:=r_0^{\frac{2n}{2n+\al}}\le (1/2)\dist(K, \pa B_1)$, which will
  imply that
  $$
  \K:=\left\{x\in B_1: \dist(x, K)\le R_0\right\}\Subset B_1.
 $$
 Our goal now is to show that for $x_0\in K$, $0<\rho<r<r_0$,
 \begin{multline}\label{eq:grad-holder-estimate}
   \int_{B_{\rho}(x_0)}|\wDu-\mean{\wDu}_{x_0, \rho}|^2 \le C(n,
   \al)\left(\frac{\rho}{r}\right)^{n+\al}\int_{B_r(x_0)}|\wDu-\mean{\wDu}_{x_0,
     r}|^2\\+C(n, \al, K)\|u\|^2_{W^{1, 2}(B_1)}r^{n+2\be},
 \end{multline}
 which readily gives the estimate \eqref{eq:grad-holder-1} by applying
 Lemma~\ref{lem:HL} and using the Campanato space embedding.

 We first prove \eqref{eq:grad-holder-estimate} for
 $x_0\in K\cap B'_1$, by taking the advantage of the symmetry of
 $\wDu$, and then we argue as in the proof of
 Proposition~\ref{prop:alm-min-Sig-Mor-est} to extend it to all
 $x_0\in K$.

 \medskip\noindent \emph{Case 1}. Suppose $x_0\in K\cap B_1'$. For
 notational simplicity, we assume $x_0=0$ (by shifting the center of
 the domain $D=B_1$ to $-x_0$) and let $0<r<r_0$ be given. Let us also
 denote
$$
\al':=1-\frac{\al}{8n}\in (0, 1),\quad R:=r^{\frac{2n}{2n+\al}}.
$$
We then split our proof into two cases:
$$
\sup_{\pa B_{R}}|u|\le C_3 R ^{\al'}\quad\text{and}\quad\sup_{\pa
  B_{R}}|u|> C_3R ^{\al'},$$ with
$C_3=2[u]_{0, \al', \K}=2\sup_{\substack{x, y\in \K\\ x\neq
    y}}\frac{|u(x)-u(y)|}{|x-y|^{\al'}}$.

\medskip\noindent \emph{Case 1.1.}  Assume first that
$\sup_{\pa B_{R}}|u|\le C_3 R ^{\al'}$. Let $h$ be the Signorini
replacement of $u$ on $B_{R}$. Then, for any $0<\rho<r$, we have
\begin{multline*}
  \int_{B_{\rho}}|\wDu-\mean{\wDu}_{\rho}|^2 \le
  3\int_{B_{\rho}}|\wDh-\mean{\wDh}_{\rho}|^2\\+3\int_{B_{\rho}}|\wDu-\wDh|^2+3\int_{B_{\rho}}|\mean{\wDu}_{\rho}-\mean{\wDh}_{\rho}|^2.
\end{multline*}
Besides, by Jensen's inequality, we have
$$
\int_{B_{\rho}}|\mean{\wDu}_{\rho}-\mean{\wDh}_{\rho}|^2 \le
\int_{B_{\rho}}|\wDu-\wDh|^2.
$$
Hence, combining the estimates above, we obtain
\begin{equation}\label{eq:grad-holder-2}
  \int_{B_{\rho}}|\wDu-\mean{\wDu}_{\rho}|^2 \le 3\int_{B_{\rho}}|\wDh-\mean{\wDh}_{\rho}|^2+6\int_{B_{\rho}}|\wDu-\wDh|^2.
\end{equation}
Similarly
\begin{equation}\label{eq:grad-holder-3}
  \int_{B_r}|\wDh-\mean{\wDh}_r|^2 \le 3\int_{B_r}|\wDu-\mean{\wDu}_r|^2+6\int_{B_r}|\wDu-\wDh|^2.
\end{equation}
Next, note that if $r_0\le (3/4)^{\frac{2n+\al}{\al}}$, then
$r\le (3/4) R$, and thus by Proposition~\ref{prop:sig-est},
\begin{multline}\label{eq:grad-holder-4}
  \int_{B_{\rho}}|\wDh-\mean{\wDh}_{\rho}|^2 \le C(n,
  \al)\Bigl(\frac{\rho}{r}\Bigr)^{n+\al}\int_{B_r}|\wDh-\mean{\wDh}_r|^2\\+C(n,
  \al)\Bigl(\int_{B_{R}}h^2\Bigr)\frac{r^{n+1}}{R^{n+3}}.
\end{multline}
Then, using \eqref{eq:grad-holder-2}, \eqref{eq:grad-holder-3}, and
\eqref{eq:grad-holder-4}, we obtain
\begin{equation}\label{eq:grad-holder-5}
  \begin{split}
    \int_{B_{\rho}}|\wDu-\mean{\wDu}_{\rho}|^2 &\le 3\int_{B_{\rho}}|\wDh-\mean{\wDh}_{\rho}|^2 + 6\int_{B_{\rho}}|\wDu-\wDh|^2 \\
    &\begin{multlined} \le C(n,
      \al)\Bigl(\frac{\rho}{r}\Bigr)^{n+\al}\int_{B_r}|\wDh-\mean{\wDh}_r|^2\\+C(n,
      \al)\Bigl(\int_{B_{R}}h^2\Bigr)\frac{r^{n+1}}{R^{n+3}}
      +6\int_{B_{\rho}}|\wDu-\wDh|^2
    \end{multlined}
    \\
    &\begin{multlined} \le C(n,
      \al)\Bigl(\frac{\rho}{r}\Bigr)^{n+\al}\int_{B_r}|\wDu-\mean{\wDu}_r|^2\\+C(n,
      \al)\Bigl(\int_{B_{R}}h^2\Bigr)\frac{r^{n+1}}{R^{n+3}}+C(n,\alpha)\int_{B_r}|\wDu-\wDh|^2.
    \end{multlined}
  \end{split}
\end{equation}
Now take $\delta=\delta(n,\al, K)>0$ such that
$\delta<\dist(K, \pa B_1)$ and
$\delta^{\al} \le \e=\e(C_1, n, n+2\al'-2)$, where $C_1$ is as in
Theorem~\ref{thm:holder} and $\e$ is as in Lemma~\ref{lem:HL}. If
$r_0\le \delta^{\frac{2n+\al}{2n}}$, then $R<\delta$, and therefore by
\eqref{eq:holder-5},
\begin{align*}
  \int_{B_{R}}|\wDu|^2 \le C(n, \al, K)\|\D u\|^2_{L^2(B_1)} R^{n+2\al'-2}.
\end{align*}
Thus, using the above inequality, combined with \eqref{eq:holder-2},
we obtain
\begin{equation}\label{eq:grad-holder-6}
  \begin{split}
    \int_{B_r}|\wDu-\wDh|^2 & \le \int_{B_{R}}|\wDu-\wDh|^2  \le \int_{B_{R}}|\wDu|^2 -\int_{B_{R}}|\wDh|^2   \\
    & \le R^{\al}\int_{B_{R}}|\wDh|^2  \le R^{\al}\int_{B_{R}}|\wDu|^2 \\
    & \le  C(n, \al, K)\|\D u\|^2_{L^2(B_1)}R^{n+\al+2\al'-2} \\
    &= C(n, \al, K)\|\D
    u\|^2_{L^2(B_1)}r^{n+\frac{\al}{2n+\al}(n-\frac{1}{2})}.
  \end{split}
\end{equation}
We next use that $h^2$ is subharmonic in $B_{R}$. (This can be seen
for instance by a direct computation
$\La(h^2)=2\left(|\D h|^2+h\La h\right)=2|\D h|^2\ge 0$, or by using
the fact that $h^\pm$ are subharmonic.) Then,
\begin{equation}\label{eq:grad-holder-7}\mean{h^2}_R \le
  \sup_{B_{R}}h^2 = \sup_{\pa B_{R}}h^2 = \sup_{\pa B_{R}}u^2
  \le C_3^2 R^{2\al'}.
\end{equation}
Also note that by \eqref{eq:holder-6},
$C_3\le C(n, \al, K)\|u\|_{W^{1, 2}(B_1)}$. Hence,
\begin{equation}\label{eq:grad-holder-8}
  \begin{split}
    \biggl(\int_{B_{R}}h^2\biggr)\frac{r^{n+1}}{R^{n+3}}
    &= C(n)\mean{h^2}_R\frac{r^{n+1}}{R^{3}}\\
    &\le C(n, \al, K)\|u\|_{W^{1,
        2}(B_1)}^2r^{n+\frac{\al}{2(2n+\al)}}.
  \end{split}
\end{equation}
Now \eqref{eq:grad-holder-5}, \eqref{eq:grad-holder-6},
\eqref{eq:grad-holder-8} give
\begin{equation}
  \begin{aligned}\label{eq:grad-holder-9}
    \int_{B_{\rho}}|\wDu-\mean{\wDu}_{\rho}|^2 &\begin{multlined}[t]
      \le C(n,\al)
      \Bigl(\frac{\rho}{r}\Bigr)^{n+\al}\int_{B_r}|\wDu-\mean{\wDu}_r|^2\\+
      C(n, \al, K)\|u\|_{W^{1, 2}(B_1)}^2
      r^{n+\frac{\al}{2(2n+\al)}}\\\shoveright{+ C(n, \al, K)\|\D
        u\|^2_{L^2(B_1)}r^{n+\frac{\al}{2n+\al}(n-\frac{1}{2})}}
    \end{multlined}\\
    &\begin{multlined} \le C(n,
      \al)\Bigl(\frac{\rho}{r}\Bigr)^{n+\al}\int_{B_r}|\wDu-\mean{\wDu}_r|^2\\+
      C(n, \al, K)\|u\|_{W^{1, 2}(B_1)}^2r^{n+\frac{\al}{2(2n+\al)}}.
    \end{multlined}
  \end{aligned}
\end{equation}

\medskip\noindent \emph{Case 1.2.}  Now suppose
$\sup_{\pa B_{R}}|u|> C_3 R ^{\al'}$. By the choice of
$C_3=2[u]_{0, \al', \K}$, we have either $u\ge (C_3/2) R^{\al'}$ in
all of $B_{R}$ or $u\le -(C_3/2) R^{\al'}$ in all of $B_{R}$. However,
from the inequality $u(0)\ge 0$, the only possibility is
$$
u\ge \frac{C_3}{2}\,R^{\al'}\quad \text{in } B_{R}.
$$
Let $h$ again be the Signorini replacement of $u$ in $B_R$. Then from
positivity of $h=u>0$ on $\pa B_{R}$ and superharmonicity of $h$ in
$B_{R}$, it follows that $h>0$ in $B_{R}$ and is therefore harmonic
there. Thus,
$$
\int_{B_{\rho}}|\D h-\mean{\D h}_{\rho}|^2 \le
\Bigl(\frac{\rho}{r}\Bigr)^{n+2}\int_{B_r}|\D h-\mean{\D h}_r|^2,\quad
0<\rho<r.
$$
Using \eqref{lem:sig-est-1} and \eqref{lem:sig-est-2} with $r$ in lieu
of $s$, we have for all $0<\rho<r$
\begin{equation}
  \begin{split}\label{eq:grad-holder-10}
    \quad \int_{B_{\rho}}|\wDh-\mean{\wDh}_{\rho}|^2 &\le \int_{B_{\rho}}|\D h-\mean{\D h}_{\rho}|^2 \le \Bigl(\frac{\rho}{r}\Bigr)^{n+2}\int_{B_r}|\D h-\mean{\D h}_r|^2 \\
    &\le
    \Bigl(\frac{\rho}{r}\Bigr)^{n+2}\int_{B_r}|\wDh-\mean{\wDh}_r|^2+\frac{1}{|B_r|}\biggl(\int_{B_r}\widehat{h_{x_n}}\biggr)^2.
  \end{split}
\end{equation}
Next, note that if $r_0\le (1/2)^{\frac{2n+\al}{\al}}$, then
$r\le R/2$. Then, for $\g:=1-\frac{3\al}{8n}$,
\begin{align*}
  \sup_{B_{R/2}}|D^2h|
  &\le \frac{C(n)}{R}\sup_{B_{(3/4)R}}|\D h| \le \frac{C(n)}{R^{1+\frac{n}{2}}}\biggl(\int_{B_{R}}|\D h|^2\biggr)^{1/2} \\
  &\le \frac{C(n)}{R^{1+\frac{n}{2}}}\biggl(\int_{B_{R}}|\D
    u|^2\biggr)^{1/2} \le C(n, \al, K)\|\D u\|_{L^2(B_1)} R^{\g-2},
\end{align*}
where the last inequality follows from \eqref{eq:holder-5}. Thus, for
$x=(x', x_n)\in B_r$, we have
\begin{align*}
  |h_{x_n}| &\le |x_n|\sup_{B_{R/2}}|D^2h|\\
            &\le C(n, \al, K)\|\D u\|_{L^2(B_1)}rR^{\g-2}\\
            &\le C(n, \al, K)\|\D u\|_{L^2(B_1)}r^{1+\frac{2n}{2n+\al}(\g-2)},
\end{align*}
and hence
\begin{multline}
  \frac{1}{|B_r|}\biggl(\int_{B_r}\widehat{h_{x_n}}\biggr)^2 \le C(n,
  \al, K)\|\D u\|_{L^2(B_1)}^2r^{n+2+\frac{4n}{2n+\al}(\g-2)}\\=C(n,
  \al, K)\|\D
  u\|_{L^2(B_1)}^2r^{n+\frac{\al}{2(2n+\al)}}\label{eq:grad-holder-11}.
\end{multline}
Combining \eqref{eq:grad-holder-10} and \eqref{eq:grad-holder-11}, we
obtain
\begin{multline}\label{eq:grad-holder-12}
  \int_{B_{\rho}}|\wDh-\mean{\wDh}_{\rho}|^2 \le
  \Bigl(\frac{\rho}{r}\Bigr)^{n+2}\int_{B_r}|\wDh-\mean{\wDh}_r|^2\\+
  C(n, \al, K)\|\D u\|_{L^2(B_1)}^2r^{n+\frac{\al}{2(2n+\al)}}.
\end{multline}
Finally, \eqref{eq:grad-holder-2}, \eqref{eq:grad-holder-3},
\eqref{eq:grad-holder-6}, and \eqref{eq:grad-holder-12} give
\begin{equation}\label{eq:grad-holder-13}
  \begin{multlined}
    \hspace{-2em}\int_{B_{\rho}}|\wDu-\mean{\wDu}_{\rho}|^2\\
    \hspace{2em}
    \begin{aligned}
      &\le 3\int_{B_{\rho}}|\wDh-\mean{\wDh}_{\rho}|^2 + 6\int_{B_{\rho}}|\wDu-\wDh|^2 \\
      &\le 3\Bigl(\frac{\rho}{r}\Bigr)^{n+2}\int_{B_r}|\wDh-\mean{\wDh}_r|^2\\
      &\qquad+C(n, \al, K)\|\D
      u\|_{L^2(B_1)}^2 r^{n+\frac{\al}{2(2n+\al)}}+6\int_{B_{\rho}}|\wDu-\wDh|^2\\
      &\le
      9\Bigl(\frac{\rho}{r}\Bigr)^{n+2}\int_{B_r}|\wDu-\mean{\wDu}_r|^2\\
      &\qquad+C(n, \al, K)\|\D u\|_{L^2(B_1)}^2
      r^{n+\frac{\al}{2(2n+\al)}}+24\int_{B_r}|\wDu-\wDh|^2\\
      &\le
      9\Bigl(\frac{\rho}{r}\Bigr)^{n+2}\int_{B_r}|\wDu-\mean{\wDu}_r|^2\\
      &\qquad+ C(n, \al, K)\|\D
      u\|_{L^2(B_1)}^2r^{n+\frac{\al}{2(2n+\al)}}\\
      &\qquad+ C(n, \al, K)\|\D
      u\|_{L^2(B_1)}^2r^{n+\frac{\al}{2n+\al}(n-\frac{1}{2})}\\
      &\le
      9\Bigl(\frac{\rho}{r}\Bigr)^{n+2}\int_{B_r}|\wDu-\mean{\wDu}_r|^2\\
      &\qquad+ C(n, \al, K)\|\D
      u\|_{L^2(B_1)}^2r^{n+\frac{\al}{2(2n+\al)}}
    \end{aligned}
  \end{multlined}
\end{equation}
From \eqref{eq:grad-holder-9} and \eqref{eq:grad-holder-13} we obtain
\eqref{eq:grad-holder-estimate} for $x_0\in K\cap B'_1$.

\medskip\noindent \emph{Case 2.}  To extend
\eqref{eq:grad-holder-estimate} to any $x_0 \in K$, we now assume
$x_0\in K\cap B^+_1$. We use an argument similar to the one in Case~2
in the proof of Proposition~\ref{prop:alm-min-Sig-Mor-est}.

Now, if $\rho\ge r/4$, then
\begin{align*}
  \int_{B_{\rho}(x_0)}|\wDu-\mean{\wDu}_{x_0, \rho}|^2 &\le \int_{B_{\rho}(x_0)}|\wDu-\mean{\wDu}_{x_0, r}|^2\\&\le 4^{n+\al}\left(\frac \rho r\right)^{n+\al}\int_{B_r(x_0)}|\wDu-\mean{\wDu}_{x_0, r}|^2,
\end{align*}
and thus we may assume $\rho< r/4$.  Let $d:=\dist(x_0, B'_1)>0$ and
choose $x_1\in \pa B_d(x_0)\cap B'_1$. Note that from the assumption
that $K$ is a ball centered at $0$, we have $x_1\in K\cap B_1'$.

\medskip\noindent \emph{Case 2.1.} If $\rho\ge d$, then from
$B_{\rho}(x_0)\subset B_{2\rho}(x_1)\subset B_{r/2}(x_1)\subset
B_r(x_0)$, we have
\begin{align*}
  \int_{B_{\rho}(x_0)}|\wDu-\mean{\wDu}_{x_0, \rho}|^2
  &\le \int_{B_{2\rho}(x_1)}|\wDu-\mean{\wDu}_{x_1, 2\rho}|^2\\
  &\begin{multlined}\le C(n, \al)\left(\frac \rho r\right)^{n+\al}\int_{B_{r/2}(x_1)}|\wDu-\mean{\wDu}_{x_1, r/2}|^2 \\+C(n, \al, K)\|u\|^2_{W^{1, 2}(B_1)}r^{n+2\be}
  \end{multlined}\\
  &\begin{multlined}\le C(n, \al)\left(\frac \rho
     r\right)^{n+\al}\int_{B_r(x_0)}|\wDu-\mean{\wDu}_{x_0,
     r}|^2\\+C(n, \al, K)\|u\|^2_{W^{1, 2}(B_1)}r^{n+2\be},
 \end{multlined}
\end{align*}
which gives \eqref{eq:grad-holder-estimate} in this case.

\medskip\noindent \emph{Case 2.2.} Now we suppose $d>\rho$. If also
$d>r$, then $B_{r}(x_0)\subset B_1^+$ and since $u$ is almost harmonic
in $B_1^+$, we can apply Proposition~\ref{prop:Anz-Mor-Camp-est},
together with the growth estimate \eqref{eq:holder-5} in the proof of
Theorem~\ref{thm:holder}, to conclude
\begin{multline*}
  \int_{B_{\rho}(x_0)}|\wDu-\mean{\wDu}_{x_0, \rho}|^2 \le C(n,
  \al)\left(\frac \rho
    r\right)^{n+\al}\int_{B_r(x_0)}|\wDu-\mean{\wDu}_{x_0,
    r}|^2\\+C(n, \al, K)\|u\|^2_{W^{1, 2}(B_1)}r^{n+2\be}.
\end{multline*}
Thus, we may assume $d\leq r$. Then, $B_d(x_0)\subset B_1^{+}$, and
hence, again by the combination of
Proposition~\ref{prop:Anz-Mor-Camp-est} and the growth estimate
\eqref{eq:holder-5}, we have
\begin{multline*}
  \int_{B_{\rho}(x_0)}|\wDu-\mean{\wDu}_{x_0, \rho}|^2 \le C(n,
  \al)\left(\frac \rho
    d\right)^{n+\al}\int_{B_d(x_0)}|\wDu-\mean{\wDu}_{x_0,
    d}|^2\\+C(n, \al, K)\|u\|^2_{W^{1, 2}(B_1)}d^{n+2\be}.
\end{multline*}
We need to consider further subcases.

\medskip\noindent\emph{Case 2.2.1.} If $r/4\le d$, then (since also
$d\leq r$)
$$\int_{B_d(x_0)}|\wDu-\mean{\wDu}_{x_0, d}|^2 \le 4^{n+\al}\left(\frac{d}{r}\right)^{n+\al}\int_{B_r(x_0)}|\wDu-\mean{\wDu}_{x_0, r}|^2$$
and combined with the previous inequality, we obtain
\eqref{eq:grad-holder-estimate} in this subcase.

\medskip\noindent \emph{Case 2.2.2.} If $d<r/4$, then we also have
\begin{align*}
  \int_{B_d(x_0)}|\wDu-\mean{\wDu}_{x_0, d}|^2
  &\le \int_{B_{2d}(x_1)}|\wDu-\mean{\wDu}_{x_1, 2d}|^2 \\
  &\begin{multlined}
    \le C(n, \al)\left(\frac{d}r\right)^{n+\al}\int_{B_{r/2}(x_1)}|\wDu-\mean{\wDu}_{x_1, r/2}|^2\\
    +C(n, \al, K)\|u\|^2_{W^{1, 2}(B_1)}r^{n+2\be}
  \end{multlined}\\
  &\begin{multlined}\le C(n,
     \al)\left(\frac{d}{r}\right)^{n+\al}\int_{B_{r}(x_0)}|\wDu-\mean{\wDu}_{x_0,
     r}|^2\\+C(n, \al, K)\|u\|^2_{W^{1, 2}(B_1)}r^{n+2\be}.
 \end{multlined}
\end{align*}
Hence, the estimate \eqref{eq:grad-holder-estimate} has been
established in all possible cases.

\medskip

To complete the proof of the theorem, we now apply Lemma~\ref{lem:HL}
to the estimate \eqref{eq:grad-holder-estimate} to obtain
\begin{multline*}
  \int_{B_{\rho}(x_0)}|\wDu-\mean{\wDu}_{x_0, \rho}|^2 \le C(n,
  \al)\biggl[
  \Bigl(\frac{\rho}{r}\Bigr)^{n+2\be}\int_{B_{r}(x_0)}|\wDu-\mean{\wDu}_{x_0,
    r}|^2\\+C(n, \al, K)\|u\|_{W^{1, 2}(B_1)}^2\rho^{n+2\be}\biggr].
\end{multline*}
Taking $r\nearrow r_0=r_0(n, \al, K)$, we have
\begin{align*}
  \int_{B_{\rho}(x_0)}|\wDu-\mean{\wDu}_{x_0, \rho}|^2 &\le C(n, \al, K)\|u\|_{W^{1, 2}(B_1)}^2\rho^{n+2\be}.
\end{align*}
Then by the Campanato space embedding we conclude that
$$\wDu\in C^{0, \be}(K)$$ with
\begin{equation*}
  \|\wDu\|_{C^{0,\be}(K)}\le C(n, \al, K)\|u\|_{W^{1, 2}(B_1)}.\qedhere
\end{equation*}
\end{proof}

Having the $C^{1,\beta}$ regularity of almost minimizers, we can now
talk about pointwise values of
$$
\partial_{x_n}^+u(x',0)=\lim_{\substack{y\ra (x', 0)\\y\in B_r^+}}
\pa_{x_n}u(y)
$$
for $x'\in B_1'$. The following complementarity condition is of
crucial importance in the study of the free boundary.

\begin{lemma}[Complementarity condition]\label{lem:comp-cond}
  Let $u$ be an almost minimizer for the Signorini problem in $B_1$.
  Then $u$ satisfies the following \emph{complementarity condition}
$$
u\,\pa_{x_n}^+u=0 \quad{on}\quad B'_1.
$$
Moreover, if $x_0\in \Gamma(u)$ then
$$
u(x_0)=0\quad\text{and}\quad |\widehat{\nabla u}(x_0)|=0.
$$
\end{lemma}

\begin{proof} Since $u\geq0$ on $B_1'$, the complementarity condition
  will follow once we show that $\partial_{x_n}^+ u$ vanishes where
  $u>0$ on $B_1'$. To this end, let $u(x', 0)>0$ for some
  $x'\in B'_1$. By the continuity of $u$ in $B_1$, (see
  Theorem~\ref{thm:holder}), we have $u>0$ in some open neighborhood
  $U \subset B_1$ of $(x', 0)$. If $B_r(y)\Subset U$ (not necessarily
  centered on $B_1'$) and $v$ is a harmonic replacement of $u$ in
  $B_r(y)$, then by the minimum principle $v>0$ in
  $\overline{B_r(y)}$, and particularly $v>0$ on set
  $B_r(y)\cap B_1'$. Then $v\in\mathfrak{K}_{0,u}(B_r(y),\cM)$ and
  therefore we must have
  $$
  \int_{B_r(y)}|\D u|^2\le \left(1+r^{\al}\right)\int_{B_r(y)}|\D
  v|^2.
  $$
  This means that $u$ is an almost harmonic function in $U$. Hence
  $u\in C^{1, \alpha/2}(U)$ by Theorem~\ref{thm:Anz}. From the even
  symmetry of $u$ in $x_n$, it is then immediate that
  $\pa_{x_n}^+u(x', 0)=\pa_{x_n}u(x', 0)=0$.

  The second part of the lemma now follows by the $C^{1,\beta}$
  regularity and the complementarity condition.
\end{proof}


\section{Weiss- and Almgren-type monotonicity formulas}
\label{sec:weiss-almgren-type}

In the rest of the paper we study the free boundary of almost
minimizers. In this section we introduce important technical tools,
so-called \emph{Weiss-} and \emph{Almgren-type monotonicity formulas},
which play a significant role in our analysis.

We start with Weiss-type monotonicity formulas. They go back to the
works of Weiss \cites{Wei99b,Wei99a} in the case of the classical
obstacle problem and Alt-Caffarelli minimum problem, respectively, and
to \cite{GarPet09} for the solutions of the thin obstacle problems. In
the context of almost minimizers, this type of monotonicity formulas
has been used in a recent paper \cite{DavEngTor17}.

\begin{theorem}[Weiss-type monitonicity formula]\label{thm:weiss}
  Let $u$ be an almost minimizer for the Signorini problem in
  $B_1$. Suppose $x_0\in B'_{1/2}$ and $u(x_0)=0$. For $0<\ka<\ka_0$
  with a fixed $\kappa_0\geq 2$ set
  \begin{align*}
    W_{\ka}(t, u, x_0):=
    &\frac{e^{a t^{\al}}}{t^{n+2\ka-2}}\left[\int_{B_{t}(x_0)}|\D
      u|^2-\kappa\frac{1-b t^\alpha}{t}\int_{\pa B_{t}(x_0)}u^2\right], 
  \end{align*}
  with
$$
a=a_\kappa=\frac{n+2\ka-2}{\al},\quad b=\frac{n+2\ka_0}\alpha.
$$
Then, for $0<t<t_0=t_0(n,\alpha,\kappa_0)$,
\begin{align*}
  \frac{d}{dt}W_\kappa(t, u, x_0)\geq \frac{e^{a t^{\al}}}{t^{n+2\ka-2}}\int_{\partial
  B_t(x_0)} \left(u_\nu-\frac{\kappa(1-b
  t^\alpha)}{t}u\right)^2.
\end{align*}
In particular, $W_{\ka}(t, u, x_0)$ is nondecreasing in $t$ for
$0<t<t_0$.
\end{theorem}

\begin{remark} It is important to observe that while $a=a_\kappa$
  depends on $\kappa$, the constant $b$ depends only on $\alpha$, $n$
  and $\ka_0$. We also note that in our version of Weiss's
  monotonicity formula, perturbations (from the case of the thin
  obstacle problem) appear in the form of multiplicative factors,
  rather than additive errors as in \cite{DavEngTor17}. Because of the
  multiplicative nature of the perturbations, we can then use the
  one-parametric family of monotonicity formulas
  $\{W_\kappa\}_{0<\kappa<\kappa_0}$ to derive an Almgren-type
  monotonicity formula, see Theorem~\ref{thm:Almgren}.
\end{remark}

\begin{remark}\label{rem:bulky-notat} To avoid bulky notations, we
  will write $W_\kappa(t,u)$ for $W_\kappa(t,u,x_0)$ when $x_0=0$ or
  even simply $W_\kappa(t)$, when both $u$ and $x_0$ are clear from
  the context.
\end{remark}

\begin{proof}
  The proof uses an argument similar to the one in Theorem~1.2 in
  \cite{Wei99a}. Essentially, it follows from a comparison
  \eqref{eq:alm-min-Sig} with special competitors, described below. Without loss
  of generality, assume $x_0=0$. Then for $t\in (0, 1/2)$, define $w$
  by
$$
w(x):=\left(\frac{|x|}{t}\right)^{\ka}u\left(t\frac{x}{|x|}\right),\quad\text{for
}x\in B_t.
$$
Note that $w$ is $\kappa$-homogeneous in $B_t$, i.e.,
$w(\lambda x )=\lambda^\kappa w(\lambda x)$ for $\lambda>0$,
$x,\lambda x\in B_t$, and coincides with $u$ on $\partial B_t$. Also
note that $w\geq 0$ on $B_t'$ and is therefore a valid competitor for
$u$ in \eqref{eq:alm-min-Sig}. We refer to this $w$ as the
\emph{$\kappa$-homogeneous replacement} of $u$ in $B_t$.

Now, in $B_t$, we have
\begin{align*}
  \nabla w(x)=
  \left(\frac{|x|}{t}\right)^{\ka-1}\left[\frac{\ka}{t}u\left(t\frac{x}{|x|}\right)\frac{x}{|x|}+\nabla
  u\left(t\frac{x}{|x|}\right)-\D u\left(t\frac{x}{|x|}\right)\cdot\frac{x}{|x|}\frac{x}{|x|}\right],
\end{align*}
which gives
\begin{align*}
  \int_{B_t}|\D w|^2dx
  &= \int_0^t\int_{\pa B_r}|\D w(x)|^2dS_xdr\\
  &= \int_0^t\int_{\pa B_r}\Bigl(\frac{r}{t}\Bigr)^{2\ka-2}\bigg|\frac{\ka}{t}u\Bigl(t\frac{x}{r}\Bigr) \nu-\Bigl(\D u\Bigl(t\frac{x}{r}\Bigr)\cdot\nu\Bigr)  \nu+\D u\Bigl(t\frac{x}{r}\Bigr)\bigg|^2dS_xdr \\
  &= \int_0^t\int_{\pa B_t}\Bigl(\frac{r}{t}\Bigr)^{n+2\ka-3}\Big| \frac{\ka}{t}u\nu-\big(\D u\cdot \nu\big)\nu+\D u\Big|^2dS_xdr\\
  &= \frac{t}{n+2\ka-2}\int_{\pa B_t}\Big|\D u-\big(\D u\cdot \nu\big)\nu+\frac{\ka}{t}u\nu\Big|^2dS_x \\
  &= \frac{t}{n+2\ka-2}\int_{\pa B_t}\Bigl(|\D u|^2-\big(\D u\cdot \nu\big)^2+\Bigl(\frac{\ka}{t}\Bigr)^2u^2\Bigr)dS_x.
\end{align*}
The latter equality can be rewritten as
\begin{multline}\label{eq:weiss-1}
  \int_{\pa B_t}u^2dS_x =
  \Bigl(\frac{t}{\ka}\Bigr)^2\biggl[\frac{n+2\ka-2}{t}\int_{B_t}|\D
  w|^2dx+ \int_{\pa B_t}\left(u_\nu^2-|\D u|^2\right)dS_x\biggr].
\end{multline}
Since $w$ is a competitor for $u$, we have
\begin{align}\label{eq:weiss-2}
  \int_{B_t}|\D w|^2dx \ge\frac{1}{1+t^{\al}}\int_{B_t}|\D u|^2dx
  \ge (1-t^{\al})\int_{B_t}|\D u|^2 dx
\end{align}
and combining \eqref{eq:weiss-1} and \eqref{eq:weiss-2} yields
\begin{multline}
  \int_{\pa B_t}u^2dS_x
  \ge\Bigl(\frac{t}{\ka}\Bigr)^2\biggl[(n+2\ka-2)\frac{1-t^{\al}}{t}\int_{B_t}|\D
  u|^2dx\\+\int_{\pa B_t}\left(u_\nu^2-|\D
    u|^2\right)dS_x\biggr].\label{eq:weiss-3}
\end{multline}
Multiplying this by $\ka^2e^{at^{\al}}t^{-n-2\ka}$ and rearranging
terms, we obtain
\begin{equation}
  \begin{aligned}
    &\frac{d}{dt}\left(e^{at^{\al}}t^{-n-2\ka+2}\right)\int_{B_t}|\D u|^2dx\\
    &\qquad = -(n+2\ka-2)e^{at^{\al}}t^{-n-2\ka}\left(t-t^{\al+1}\right)\int_{B_t}|\D u|^2dx \\
    &\qquad\ge e^{at^{\al}}t^{-n-2\ka+2}\int_{\pa
      B_t}\left(u_{\nu}^2-|\D
      u|^2\right)dS_x-\ka^2e^{at^{\al}}t^{-n-2\ka}\int_{\pa
      B_t}u^2dS_x.\label{eq:weiss-4}
  \end{aligned}
\end{equation}
Define now an auxiliary function
$$
\psi(t)=\frac{\ka e^{a t^{\al}}(1-bt^{\alpha})}{t^{n+2\ka-1}}.
$$
Then we write
$$
W_\kappa(t, u,0)=e^{at^{\al}}t^{-n-2\ka+2}\int_{B_t}|\D
u|^2dx-\psi(t)\int_{\partial B_t}u^2dS_x
$$
and, using \eqref{eq:weiss-4}, obtain
\begin{align*}
  \frac{d}{dt}W_{\ka}(t, u, 0)
  &=\frac{d}{dt}\left(e^{at^{\al}}t^{-n-2\ka+2}\right)\int_{B_t}|\D
    u|^2dx+e^{at^{\al}}t^{-n-2\ka+2}\int_{\pa B_t}|\D u|^2dS_x\\&\qquad-\psi'(t)\int_{\pa B_t}u^2dS_x-2\psi(t)\int_{\pa B_t}uu_{\nu}dS_x-(n-1)\frac{\psi(t)}{t}\int_{\pa B_t}u^2dS_x
  \\&\ge e^{at^{\al}}t^{-n-2\ka+2}\int_{\pa B_t}\left(u_{\nu}^2-|\D u|^2\right)dS_x-\ka^2e^{at^{\al}}t^{-n-2\ka}\int_{\pa B_t}u^2dS_x \\
  &\qquad+e^{at^{\al}}t^{-n-2\ka+2}\int_{\pa B_t}|\D u|^2dS_x-\psi'(t)\int_{\pa B_t}u^2dS_x\\
  &\qquad-2\psi(t)\int_{\pa B_t}uu_{\nu}dS_x-(n-1)\frac{\psi(t)}{t}\int_{\pa B_t}u^2dS_x\\
  &= e^{at^{\al}}t^{-n-2\ka+2}\int_{\pa B_t}u_{\nu}^2dS_x-2\psi(t)\int_{\pa B_t}uu_{\nu}dS_x\\
  &\qquad-\left(\ka^2e^{at^{\al}}t^{-n-2\ka}+\psi'(t)+(n-1)\frac{\psi(t)}{t}\right)\int_{\pa
    B_t}u^2dS_x.
\end{align*}
Now observe that $\psi(t)$ satisfies the inequality
$$
-\frac{e^{at^{\al}}}{t^{n+2\ka-2}}\left(\ka^2e^{at^{\al}}t^{-n-2\ka}+\psi'(t)+(n-1)\frac{\psi(t)}{t}
\right)-\psi^2(t)\geq0
$$
for $0<t<t_0(n, \alpha,\kappa_0)$ and $0<\kappa<\ka_0$. Indeed, a
direct computation shows that the above inequality is equivalent to
$$
2\alpha^2(1+\kappa_0-\kappa)-(n+2\kappa_0)[(n+2\kappa_0)\kappa-\alpha(n+2\kappa-2)]t^\alpha\geq
0,
$$
which holds for $0<\kappa<\ka_0$ and small $t>0$ such that
$$
2\alpha^2 - 4(n+2\kappa_0)^2\kappa_0t^\alpha\geq 0.
$$
Hence, recalling also the formula for $\psi(t)$, we can conclude that
\begin{align*}
  \frac{d}{dt}W_\kappa(t, u, 0)
  &\geq \frac{e^{at^{\al}}}{t^{n+2\ka-2}}\biggl[\int_{\pa
    B_t}u_{\nu}^2dS_x-2\frac{\kappa(1-bt^\alpha)}{t}\int_{\pa B_t}uu_{\nu}dS_x\\
  &\qquad +\Bigl(\frac{\kappa(1-bt^\alpha)}{t}\Bigr)^2\int_{\partial
    B_t} u^2 dS_x\biggr]\\
  &= \frac{e^{a t^{\al}}}{t^{n+2\ka-2}}\int_{\partial
    B_t} \left(u_\nu-\frac{\kappa(1-b t^\alpha)}{t}u\right)^2,
\end{align*}
for $0<t<t_0(n,\alpha,\kappa_0)$.
\end{proof}

Next, for an almost minimizer $u$ in $B_1$ and $x_0\in B'_{1/2}$,
consider the quantity
$$
N(t, u, x_0):=\frac{t\int_{B_t(x_0)}|\D u|^2}{\int_{\pa
    B_t(x_0)}u^2},\quad 0<t<1/2
$$
which is known as \emph{Almgren's frequency} and goes back to
Almgren's \emph{Big Regularity Paper} \cite{Alm00}. This kind of
quantities have also been used in unique continuation for a class of
elliptic operators \cites{GarLin86,GarLin87} and have been
instrumental in thin obstacle-type problems, starting with the works
\cites{AthCafSal08,CafSalSil08,GarPet09}.

Before proceeding, we observe that Almgren's frequency is well defined
when $x_0$ is a free boundary point, since $\int_{\pa
  B_t(x_0)}u^2>0$. Indeed, otherwise $u=0$ on $\pa B_t(x_0)$ and we
can use $h\equiv0$ in $B_t(x_0)$ as a competitor, to obtain that
$\int_{B_t(x_0)}|\nabla u|^2\leq (1+t^\alpha)0=0$, implying $u\equiv0$
in $B_t(x_0)$, contradicting the assumption that $x_0$ is a free
boundary point.  Next, we also consider a modification of $N$:
\begin{align*}
  \N(t,u,x_0) &:=\frac{1}{1-bt^\alpha}N(t,u,x_0),
   \intertext{where $b$ is as in Theorem~\ref{thm:weiss}, as well as} 
  \widehat N_{\ka_0}(t,u,x_0)&:=\min\{\widetilde N(t),\ka_0\},\quad 0<t<t_0,
\end{align*}
which we call the \emph{truncated frequency}.

For the frequencies $N$, $\N$, and $\widehat N_{\kappa_0}$, we will
follow the same notational conventions as outlined in
Remark~\ref{rem:bulky-notat} for Weiss's functionals $W_\kappa$.

\medskip

With the Weiss type monotonicity formula at hand, we easily obtain the
following monotonicity of $\widehat N_{\ka_0}$.

\begin{theorem}[Almgren-type monotonicity formula]\label{thm:Almgren}
  Let $u$, $\kappa_0$, and $t_0$ be as in Theorem~\ref{thm:weiss}, and
  $x_0$ a free boundary point. Then $\widehat N_{\ka_0}(t,u,x_0)$ is
  nondecreasing in $0<t<t_0$.
\end{theorem}

\begin{proof}
  We assume $x_0=0$. It is quite important to observe that $t_0$
  depends only on $n$, $\alpha$, and $\kappa_0$. Then, if $\N(t)< \ka$
  for some $t\in (0, t_0)$ and $\ka \in (0, \ka_0)$, then
  \begin{align*}
    W_{\ka}(t)
    &=\frac{e^{a t^{\al}}}{t^{n+2\ka-1}}\left(\int_{\pa
      B_t}u^2\right)\left(N(t)-\ka(1-b t^\alpha)\right)\\
    &=\frac{e^{a t^{\al}}}{t^{n+2\ka-1}}\left(\int_{\pa
      B_t}u^2\right)(1-b t^\alpha)\left(\N(t)-\ka\right)     <0.
  \end{align*}
  By Theorem~\ref{thm:weiss} we also have
  $W_{\ka}(s) \le W_{\ka}(t)<0$ for all $s\in (0, t)$, and thus
  $\N(s)< \ka$. This completes the proof.
\end{proof}
\begin{remark} The proof above is rather indirect and establishes the
  monotonicity of $\widehat{N}_{\kappa_0}$ from that of Weiss-type
  formulas in one-parametric family
  $\{W_\kappa\}_{0<\kappa<\kappa_0}$.  This kind of relation has been
  first observed in \cite{GarPet09}.
\end{remark}


\section{Almgren rescalings and blowups}
\label{sec:almgr-resc-blow}

In this section we prove a lower bound on Almgren's frequency for
almost minimizers at free boundary points. The idea is to consider
appropriate rescalings and blowups of almost minimizers to obtain
solutions of the Signorini problem, for which a bound $N(0+)\geq 3/2$
is known.

Now, let $u$ be an almost minimizer for the Signorini problem in
$B_1$, and $x_0\in B_{1/2}'$ a free boundary point. For $0<r<1/2$
consider the \emph{Almgren rescaling}\footnote{We use the superscript
  $A$ to distinguish this rescaling from the other rescalings, namely,
  homogeneous and almost homogeneous rescalings that we consider
  later.} of $u$ at $x_0$
\begin{align*}
  u_{x_0, r}^A(x):=\frac{u(rx+x_0)}{\Bigl(\frac{1}{r^{n-1}}\int_{\pa
  B_r(x_0)}u^2\Bigr)^{\frac{1}{2}}}, \quad x\in B_{1/(2r)}.
\end{align*}
When $x_0=0$, we also write $u_r^A$ instead of $u_{0,r}^A$. The
Almgren rescalings have the following normalization and scaling
properties
\begin{align*}
  &\|u_{x_0, r}^A\|_{L^2(\pa B_1)}=1\\
  &N(\rho, u_{x_0,
    r}^A)=N(\rho r, u, x_0),\quad \rho<1/(2r).
\end{align*}
We will call the limits of $u_{x_0,r}^A$ over any sequence
$r=r_j\to 0+$ \emph{Almgren blowups} of $u$ at $x_0$ and denote by
$u_{x_0,0}^A$.

\begin{proposition}[Existence of Almgren blowups]\label{prop:exist-Alm-blowup}
  Let $x_0\in B_{1/2}'\cap \Gamma(u)$ be such that
  $\widehat{N}_{\kappa_0}(0+,u,x_0)=\kappa<\kappa_0$. Then every
  sequence of Almgren rescalings $u_{x_0,r_j}^A$, with $r_j\to 0+$
  contains a subsequence, still denoted $r_j$, such that for a
  function
  $u_{x_0,0}^A\in W^{1,2}(B_1)\cap C^{1}_{\loc}(B_{1}^\pm\cup B_1' )$
  \begin{align*}
    u_{x_0,r_j}^A\ra u_{x_0,0}^A &\quad\text{in }W^{1, 2}(B_1),\\
    u_{x_0,r_j}^A\ra u_{x_0,0}^A &\quad\text{in }L^2(\pa B_1),\\
    u_{x_0,r_j}^A\ra u_{x_0,0}^A &\quad\text{in }C^1_{\loc}(B_1^{\pm}\cup B_1').
  \end{align*}
  Moreover, $u_{x_0,0}^A$ is a nonzero solution of the Signorini
  problem in $B_1$, even in $x_n$, and homogeneous of degree $\kappa$
  in $B_1$, i.e.,
$$
u_{x_0,0}^A(\lambda x)=\lambda^\kappa u_{x_0,0}^A(x),
$$
for $\lambda>0$, provided $x,\lambda x\in B_1$.
\end{proposition}
\begin{proof} Without loss of generality, we assume $x_0=0$. From the
  fact that $\widehat{N}(0+,u)=\kappa<\kappa_0$, it follows also that
  $N(0+,u)=\widehat{N}(0+,u)=\kappa$. In particular,
  $N(r_j,u)<\kappa_0$ for large $j$. Then, for such $j$
  $$
  \int_{B_1}|\nabla u_{r_j}^A|^2=N(1,u_{r_j}^A)=N(r_j,u)\leq \kappa_0
$$
and combined with the normalization
$\int_{\partial B_1}(u_{r_j}^A)^2=1$, we see that the sequence
$u_{r_j}^A$ is bounded in $W^{1,2}(B_1)$. Hence, there is a function
$u_0^A\in W^{1, 2}(B_1)$ such that, over a subsequence,
\begin{align*}u_{r_j}^A \to u_0^A&\quad\text{weakly in }W^{1, 2}(B_1),\\
  u_{r_j}^A\to u_0^A&\quad\text{strongly in }L^2(\pa B_1).
\end{align*}
In particular, $\int_{\pa B_1}(u_0^A)^2=1$, implying that
$u_0^A\not\equiv 0$ in $B_1$.

Next, we observe that since $u$ is an almost minimizer in $B_{1}$ with
gauge function $\omega(t)=t^{\al}$, $u_{r}^A$ is also an almost
minimizer in $B_{1/(2r)}$ with gauge function
$\omega_r(t)=(rt)^{\al}$. This is rather easy to see, since $u_r^A(x)$
up to a positive constant factor is $u(r x)$ and the multiplication
(or the division) by a positive number preserves the almost minimizing
property. Since $\omega_r(t)\leq \omega(t)$,
Theorem~\ref{thm:grad-holder} is applicable to rescalings $u_{r_j}^A$,
from where we can deduce that over yet another subsequence,
\begin{align}
  u_{r_j}^A \ra u_0^A\quad \text{in } C^1_{\loc}(B_1^{\pm}\cup B'_1). \label{Almgren-1}
\end{align}
Now, we claim that since the gauge functions
$\omega_r(t)=(rt)^\alpha\to 0$ as $r\to 0$, the blowup $u_0^A$ is a
solution of the Signorini problem in $B_1$. Indeed, for a fixed $r_j$,
let $h_{r_j}$ be the Signorini replacement of $u_{r_j}^A$ in
$B_1$. Then, by repeating the argument as in the proof of
Proposition~\ref{prop:alm-min-Sig-Mor-est}
$$
\int_{B_1}|\nabla (u_{r_j}^A-h_{r_j})|^2\leq
r_j^\alpha\int_{B_1}|\nabla u_{r_j}^A|^2.
$$
This implies that $h_{r_j}\to u_0^A$ weakly in $W^{1,2}(B_1)$. On the
other hand, by the boundedness of the sequence $h_{r_j}$ in
$W^{1,2}(B_1)$, we have also boundedness in
$C^{1,1/2}_\loc(B_1^\pm \cup B_1')$ and hence, over a subsequence,
$h_{r_j}\to u_0^A$ in $C^1_{\loc}(B_1^\pm\cup B_1')$. By this
convergence we then conclude that $u_0^A$ satisfies
\begin{align*}
  \Delta u_0^A=0&\quad\text{in }B_1\setminus B_1'\\*
  u_0^A\geq 0,\quad -\partial_{x_n}^+u_0^A\geq 0,\quad
  u_0^A\partial_{x_n}^+u_0^A=0&\quad\text{on } B_1',
\end{align*}
and hence $u_0^A$ itself solves the Signorini problem in $B_1$.

Using the $C^1_{\loc}$ convergence again, we have that for any
$0<\rho<1$
\begin{align*}
  N(\rho, u_0^A)=\lim_{r_j\ra 0}N(\rho, u_{r_j}^A)=\lim_{r_j\ra 0}N(\rho
  r_j, u)=N(0+, u)=\ka.
\end{align*}
Thus, the Almgren frequency of $u_0^A$ is constant $\kappa$, which is
possible only if $u_0^A$ is a $\kappa$-homogeneous solution of the
Signorini problem in $B_1$, see \cite{PetShaUra12}*{Theorem~9.4}.
\end{proof}

In what follows, it will be sufficient for us to fix $\kappa_0\geq 2$
(say $\kappa_0=2$), in the definition of $\widehat{N}_{\kappa_0}$ and
we will simply write
$$
\widehat N = \widehat N_{\ka_0}.
$$

\begin{lemma}[Minimal frequency]
  \label{lem:min-freq} Let $u$ be an almost minimizer for the
  Signorini problem in $B_1$. If $x_0\in B_{1/2}'\cap \Gamma(u)$, then
$$
\widehat{N}(0+,u,x_0)=\lim_{r\to 0+}\widehat N(r,u,
x_0)\geq\frac{3}{2}.
$$
Consequently, we also have
$$
\widehat{N}(t,u,x_0)\geq 3/2\quad\text{for }0<t<t_0.
$$
\end{lemma}

\begin{proof} As before, let $x_0=0$. Assume to the contrary that
  $\widehat N(0+,u)=\kappa<3/2$. Since $\kappa<\kappa_0$ we can apply
  Proposition~\ref{prop:exist-Alm-blowup} to obtain that over a
  sequence $r_j\to 0+$, $u_{r_j}^A\to u_0^A$ in
  $C^1_{\loc}(B_1^\pm\cup B_1')$, where $u_0^A$ is a nonzero
  $\kappa$-homogeneous solution of the Signorini problem in $B_1$,
  even in $x_n$. Moreover, since $0\in \Gamma(u)$, by
  Lemma~\ref{lem:comp-cond} we have that
  $u(0)=|\widehat{\nabla u}(0)|=0$, implying that
  $u_{r_j}^A(0)=|\widehat{\nabla u_{r_j}^A}(0)|=0$ and, by passing to
  the limit, $u_{0}^A(0)=|\widehat{\nabla u_{0}^A}(0)|=0$. Now, to
  arrive at a contradiction, we argue as in the proof of
  \cite{PetShaUra12}*{Proposition~9.9} to reduce the problem to
  dimension $n=2$, where we can classify all possible homogeneous
  solutions of the Signorini problem, even in $x_n$. The only nonzero
  homogeneous solutions with $\kappa<3/2$ in dimension $n=2$ are
  possible for $\kappa=1$ and have the form $u_0^A(x)=-c x_n$ for some
  $c>0$, but they fails to satisfy the condition
  $|\widehat{\nabla u_{0}^A}(0)|=0$. Thus, we arrived at
  contradiction, implying that $\widehat{N}(0+,u)\geq 3/2$.  Finally,
  applying Theorem~\ref{thm:Almgren}, we obtain
  $\widehat{N}(t,u)\geq \widehat{N}(0+,u)\geq 3/2$, for $0<t<t_0$.
\end{proof}

\begin{corollary}\label{cor:W32-nonneg} Let $u$ be an almost minimizer
  for the Signorini problem in $B_1$ and $x_0$ a free boundary
  points. Then
  $$
  W_{3/2}(t, u, x_0)\geq 0\quad\text{for }0<t<t_0.
  $$
\end{corollary}
\begin{proof} We simply observe that
  $\N(t)\geq \widehat{N}(t)\geq 3/2$ for $0<t<t_0$ and hence
  \[
    W_{3/2}(t, u,
    x_0)=\frac{e^{at^{\al}}}{t^{n+2\ka-1}}\left(\int_{\pa
        B_t}u^2\right)(1-b t^\alpha)\left(\widetilde
      N(t)-\frac{3}{2}\right) \ge 0.\qedhere
  \]
\end{proof}


\section{Growth estimates}
\label{sec:growth-estimates}

An important step in the study of the free boundary in the Signorini
problem (and in many other free boundary problems) is the proof of the
optimal regularity of solutions, which in this case is $C^{1,1/2}$ on
each side of the thin space. This allows to make proper blowup
arguments to establish the regularity of the so-called regular part of
the free boundary. However, in the case of almost minimizers, we only
know $C^{1,\beta}$ regularity for some small $\beta>0$ and do not
expect to have anything better. Yet, in this section, we establish the
\emph{optimal growth} of the almost minimizers at free boundary points
with the help of the Weiss-type monotonicity formula and the
epiperimetric inequality.

Finally, we want to point out that the results in this section are
rather immediate in the case of minimizers, as they follow easily from
the differentiation formulas for the quantities involved in the
Almgren's frequency formula. This is completely unavailable for almost
minimizers.

We start by defining a new type of rescalings. Fix $\kappa\geq
3/2$. For a free boundary point $x_0$ in $B_{1/2}'$ and $r>0$, we
define the \emph{$\kappa$-homogeneous rescaling} by
$$
u_{x_0, r}(x):=u_{x_0, r}^{(\kappa)}(x)=
\frac{u(rx+x_0)}{r^{\kappa}},\quad x\in B_{1/(2r)}.
$$
To take advantage of the Weiss-type monotonicity formula, we need a
slight modification of this rescaling. With the help of an auxiliary
function
$$
\phi(r)=\phi_\kappa(r):=e^{-(\kappa b/\alpha)r^\alpha}r^{\kappa},\quad
r>0,
$$
which is a solution of the differential equation
$$
\phi'(r)=\kappa\,\phi(r)\frac{1-b r^\alpha}r,\quad r>0
$$
we define the \emph{$\kappa$-almost homogeneous rescalings} by
$$
u_{x_0, r}^{\phi}(x):=\frac{u(rx+x_0)}{\phi(r)},\quad x\in B_{1/(2r)}.
$$

\begin{lemma}[Weak growth estimate]
  \label{lem:almost-opt-growth}
  Let $u$ be an almost minimizer of the Signorini problem in $B_1$ and
  $x_0\in B'_{1/2}\cap\Gamma(u)$ be such that
  $\widehat{N}(0+, u,x_0)\geq \ka$ for $\kappa\leq\kappa_0$. Then
  \begin{align*}
    \int_{\pa B_t(x_0)}u^2&\le C(n, \al,\kappa_0)\|u\|^2_{W^{1, 2}(B_1)}\left(\log\frac{1}{t}\right)t^{n+2\ka-1},\\
    \int_{B_t(x_0)}|\D u|^2&\le C(n, \al, \kappa_0)\|u\|^2_{W^{1,2}(B_1)}\left(\log\frac{1}{t}\right)t^{n+2\ka-2},
  \end{align*}
  for $0<t<t_0=t_0(n, \al, \kappa_0)$.
\end{lemma}
\begin{proof} Without loss of generality, assume $x_0=0$. We first
  note that the condition $\widehat{N}(0+,u)\geq \kappa$ implies that
  $\widehat{N}(t,u)\geq \kappa$ for
  $0<t<t_0=t_0(n,\alpha,\kappa_0)$. Then also $\N(t,u)\geq \kappa$ for
  such $t$ and consequently,
  \[
    W_{\kappa}(t, u)=\frac{e^{at^{\al}}}{t^{n+2\ka-1}}\left(\int_{\pa
        B_t}u^2\right)(1-b t^\alpha)\left(\widetilde
      N(t,u)-\kappa\right) \ge 0.
  \]
  Next, for $\phi=\phi_\kappa$, we have that
  \begin{align*}
    \frac{d}{dr}u_{r}^{\phi}(x)
    &=\frac{\nabla u(rx)\cdot x}{\phi(r)}-\frac{u(rx)[\phi'(r)/\phi(r)]}{\phi(r)}\\
    &=\frac{1}{\phi(r)}\left(\nabla u(rx)\cdot
      x-\frac{\kappa(1-br^\alpha)}{r}u(rx)\right).
  \end{align*}
  Now let
$$
m(r)=\left(\int_{\partial B_1} (u_r^\phi(\xi))^2
  dS_\xi\right)^{1/2},\quad r>0.
$$
Then,
$$
m'(r)=\left(\int_{\partial B_1} u_r^\phi(\xi)
  \frac{d}{dr}u_r^\phi(\xi)dS_\xi\right)\left(\int_{\partial B_1}
  (u_r^\phi(\xi))^2 dS_\xi\right)^{-1/2}
$$
and consequently, by Cauchy-Schwarz,
$$
|m'(r)|\leq \left(\int_{\partial
    B_1}\left[\frac{d}{dr}u_r^\phi(\xi)\right]^2dS_\xi\right)^{1/2}.
$$
Hence,
\begin{align*}
  |m'(r)|
  &\leq
    \frac{1}{\phi(r)}\left(\int_{\partial B_1}\left(\nabla u(r\xi)\cdot \xi
    -\frac{\kappa(1-br^\alpha)}{r}u(r \xi)\right)^2  dS_\xi\right)^{1/2}\\
  &=\frac{1}{\phi(r)}\left(\frac{1}{r^{n-1}}\int_{\partial
    B_r}\left(\partial_\nu u(x)-\frac{\kappa(1-br^\alpha)}{r}
    u(x)\right)^2dS_x\right)^{1/2}\\
  &\leq\frac{1}{\phi(r)}\left(\frac{1}{r^{n-1}}\frac{r^{n+1}}{e^{a
    r^\alpha}}\frac{d}{dr}W_{\kappa}(r)\right)^{1/2}=\frac{e^{c
    r^\alpha}}{r^{1/2}}\left(\frac{d}{dr}W_{\kappa}(r)\right)^{1/2}
    ,\quad c=\kappa\frac{b}{\alpha}-\frac{a}{2},
\end{align*}
for $0<r<t_0=t_0(n,\alpha,\kappa_0)$.  Thus, we have shown
$$
|m'(r)|\leq
\frac{e^{cr^\alpha}}{r^{1/2}}\left(\frac{d}{dr}W_{\kappa}(r)\right)^{1/2},\quad
0<r<t_0.
$$
Integrating in $r$ over the interval $(s,t)\subset (0,t_0)$, we obtain
\begin{align*}
  |m(t)-m(s)|
  &\leq \int_{s}^{t}
    \frac{e^{cr^\alpha}}{r^{1/2}}\left(\frac{d}{dr}W_{\kappa}(r)\right)^{1/2}dr\\
  &\leq\left(\int_{s}^{t}\frac{e^{2cr^{\alpha}}}{r}dr\right)^{1/2}\left(\int_{s}^{t}
    \frac{d}{dr}W_{\kappa}(r)\right)^{1/2}\\
  &\leq C_0\left(\log\frac{t}{s}\right)^{1/2}\left[W_{\kappa}(t)-W_{\kappa}(s)\right]^{1/2}.
\end{align*}
In particular (recalling that $W_{\kappa}(s)\geq 0$), we obtain
\begin{align*}
  m(t)&\leq m(t_0) + C_0\left(\log\frac{t_0}{t}\right)^{1/2}\left[W_{\kappa}(t_0)\right]^{1/2}.
\end{align*}
Varying $t_0$ by an absolute factor, we can guarantee that
$$
m(t_0)\leq C(n,\alpha,\kappa_0)\|u\|_{L^2(B_1)},\quad
W_{\kappa}(t_0)\leq C(n,\alpha,\kappa_0)\|u\|_{W^{1,2}(B_1)}^2.
$$
Hence, we can conclude
\[
  \int_{\partial B_t} u^2\leq
  C(n,\alpha,\kappa_0)\|u\|_{W^{1,2}(B_1)}^2
  \left(\log\frac{1}{t}\right)t^{n+2\kappa-1},
\]
for $0<t<t_0=t_0(n,\alpha,\kappa_0)$. This implies the first
bound. The second bound follows immediately from the first one by
using that $W_{\kappa}(t, u)\leq W_{\kappa}(t_0, u)$:
\begin{align*}
  \frac{1}{t^{n+2\kappa-2}}\int_{B_t}|\nabla u|^2
  &\leq\frac{\kappa(1-bt^\alpha)}{t^{n+2\kappa-1}}\int_{\partial B_t}u^2+e^{-at^\alpha}W_{\kappa}(t_0, u)\\
  &\leq
    C(n,\alpha,\kappa_0)\|u\|_{W^{1,2}(B_1)}^2\left(\log\frac{1}{t}\right)
    +\frac{e^{at_0^\alpha}}{t_0^{n+2\kappa-2}}\int_{B_{t_0}}|\nabla u|^2\\
  &\leq C(n,\alpha,\kappa_0)\|u\|_{W^{1,2}(B_1)}^2\left(\log\frac{1}{t}\right).\qedhere
\end{align*}
\end{proof}

The logarithmic term in Lemma~\ref{lem:almost-opt-growth} does not
allow to conclude that the sequence of $\kappa$-homogeneous or almost
homogeneous rescaling is uniformly bounded say in $W^{1,2}(B_1)$. In
the rest of this section we show that in the case of the minimal
frequency $\kappa=3/2$ we can do that with the help of the so-called
epiperimetric inequality for the Signorini problem for the Weiss
energy
$$
W^0_{3/2}(w):=\int_{B_1}|\D w|^2-\frac 32 \int_{\pa B_1}w^2.
$$
To state this result, we let
\begin{equation}\label{eq:epiper-A}
  \mathcal{A}:=\{w\in W^{1, 2}(B_1): w\ge 0\text{ on }B'_1,\ w(x', x_n)=w(x', -x_n)\}
\end{equation}

\begin{theorem}[Epiperimetric inequality]
  There exists $\eta\in(0, 1)$ such that if $w\in\mathcal{A}$ is
  homogeneous of degree $3/2$ in $B_1$, then there exists
  $v\in\mathcal{A}$ with $v=w$ on $\pa B_1$ such that
  \[
    W^0_{3/2}(v)\le(1-\eta)W^0_{3/2}(w).
  \]
\end{theorem}
This kind of inequalities go back to the work of Weiss \cite{Wei99b},
in the case of the classical obstacle problem. For the Signorini
problem, a version of this theorem was proved in \cite{GarPetSVG16}
and \cite{FocSpa16}. In fact, the theorem above is the version in
\cite{RueShi17}.  The inequality in \cite{GarPetSVG16} and
\cite{FocSpa16} requires $w$ to be close to the blowup profile, but
this can be easily removed by a scaling argument (see
\cite{RueShi17}). We also refer to \cite{ColSpoVel17}, for a more
direct proof of this inequality with an explicit constant
$\eta=1/(2n+3)$.

\medskip Now, with the help of the epiperimetric inequality, we can
prove a decay estimate for the Weiss-type energy functional
$W_{3/2}$. For the rest of the section, we will assume
$$
\kappa_0=2,
$$
which will make some of the constants independent of $\kappa_0$, but the results hold also for any other value of $\kappa_0\geq2$, with possible added dependence of constants on $\kappa_0$.

\begin{lemma}\label{lem:W-gr-est} Let $x_0\in B_{1/2}'$ be a free
  boundary point. Then, there exist $\de=\de(n,\alpha)>0$ such that
$$
0\leq W_{3/2}(t,u,x_0)\leq C t^\de,\quad 0<t<t_0=t_0(n,\alpha),
$$
with $C=C(n,\alpha)\|u\|_{W^{1,2}(B_1)}^2$.
\end{lemma}

\begin{proof} As before, without loss of generality we assume that
  $x_0=0$.

  The proof will follow from a differential inequality that we derive
  by using our earlier computations and the epiperimetric
  inequality. Recalling the proof of the Weiss-type monotonicity
  formula (Theorem~\ref{thm:weiss}), for small $t>0$, we have
  \begin{align*}
    \frac{d}{dt}W_{3/2}(t,u)
    &=\frac{e^{at^\alpha}}{t^{n+1}}\int_{\partial
      B_t} |\nabla
      u|^2-\frac{(n+1)(1-t^\alpha)e^{at^\alpha}}{t^{n+2}}\int_{B_t}|\nabla
      u|^2\\
    &\qquad-\psi'(t)\int_{\partial B_t}
      u^2-(n-1)\frac{\psi(t)}{t}\int_{\partial B_t}
      u^2-2\psi(t)\int_{\partial B_t} u \partial_\nu u\\
    &=-\frac{(n+1)(1-t^\alpha)}{t}W_{3/2}(t,
      u)+\frac{e^{at^\alpha}}{t^{n+1}}\int_{\partial B_t} |\nabla u|^2\\
    &\qquad-\left([(n+1)(1-t^\alpha)+(n-1)]\frac{\psi(t)}{t} +
      \psi'(t)\right)\int_{\partial B_t}u^2\\
    &\qquad-2\psi(t)\int_{\partial B_t}u\partial_\nu u\\
    &\geq -\frac{(n+1)(1-t^\alpha)}{t}W_{3/2}(t, u)\\
    &\qquad +\frac{e^{at^\alpha}(1-bt^\alpha)}{t^{n+1}}
      \begin{multlined}[t]\int_{\partial
          B_t}\biggl(|\nabla u|^2-\frac{3}{t} u\partial_\nu u\\
        -\frac{3}{2t}\left[\frac{(n+1)(1-t^\alpha)+(n-1)}{t}+\frac{\psi'(t)}{\psi(t)}\right]u^2\biggr).
      \end{multlined}
  \end{align*}
  To proceed, note that
$$
\frac{(n+1)(1-t^\alpha)+(n-1)}{t}+\frac{\psi'(t)}{\psi(t)}=\frac{(n-2)+O(t^\alpha)}{t}.
$$
Now, for the homogeneous rescalings
$$
u_t(x)=\frac{u(tx)}{t^{3/2}},
$$
we can write
\begin{align*}
  &\int_{\partial B_t}|\nabla u|^2-\frac{3}{t} u\partial_\nu
    u-\frac{3}{2}\frac{(n-2)+O(t^\alpha)}{t^2} u^2\\
  &\qquad=t^{n}\int_{\partial B_1}|\nabla u_t|^2
    -3u_t\partial_\nu u_t-\frac{3}{2}[(n-2)+O(t^\alpha)]u_t^2\\
  &\qquad=t^{n}\int_{\partial B_1}\left(\partial_\nu
    u_t-\frac{3}{2}u_t\right)^2+(\partial_\tau u_t)^2-\frac{3}{2}\left[\left(n-\frac12\right)+O(t^\alpha)\right]u^2_t,
\end{align*}
where $\partial_\tau u_t$ is the tangential component of $\nabla u_t$
on the unit sphere.  We can summarize for now that
\begin{align*}
  \frac{d}{dt}W_{3/2}(t,u)
  &\geq -\frac{(n+1)(1-t^\alpha)}{t}W_{3/2}(t, u)\\
  &\qquad+\frac{e^{a
    t^\alpha}(1-bt^\alpha)}{t}\int_{\partial B_1}\left[\left(\partial_\nu
    u_t-\frac{3}{2}u_t\right)^2+(\partial_\tau
    u_t)^2-\frac{3}{2}\left(n-\frac12\right)u^2_t\right]\\
  &\qquad+O(t^{\alpha-1})\int_{\partial
    B_1}u_t^2.
\end{align*}
On the other hand, if $w_t$ is a $3/2$-homogeneous replacement of
$u_t$ in $B_1$, i.e.,
$$
w_t(x)=|x|^{3/2}u_t(x/|x|)
$$
then
\begin{align*}
  \int_{\partial B_1} (\partial_\tau
  u_t)^2-\frac{3}{2}\left(n-\frac12\right)u^2_t
  &=\int_{\partial B_1}(\partial_\tau w_t)^2-\frac{3}{2}\left(n-\frac12\right)w^2_t\\&=(n+1)W_{3/2}^0(w_t),
\end{align*} where
$$
W^0_{3/2}(w_t)=\int_{B_1}|\nabla w_t|^2-\frac{3}{2}\int_{\partial
  B_1}w_t^2.
$$
The last equality follows by repeating the arguments in the beginning
of the proof of Theorem~\ref{thm:weiss} with $\kappa=3/2$. Let $v_t$
be the solution of the Signorini problem in $B_1$ with $v_t=u_t=w_t$
on $\partial B_1$. Then by the epiperimetric inequality
$$
W^0_{3/2}(v_t)\leq (1-\eta)W^0_{3/2}(w_t).
$$
On the other hand, since $u$ is an almost minimizer, we have
$$
\int_{B_1}|\nabla u_t|^2\leq (1+t^\alpha)\int_{B_1}|\nabla v_t|^2
$$
and since also $u_t=v_t$ on $\partial B_1$, we have
\begin{align*}
  W_{3/2}(t,u)
  &=\frac{e^{at^\alpha}}{t^{n+1}}\left[\int_{B_t}|\nabla
    u|^2-\frac{(3/2)(1-bt^\alpha)}{t}\int_{\partial B_t}u^2\right]\\
  &\leq
    (1+O(t^\alpha))W^0_{3/2}(v_t)+O(t^\alpha)\int_{\partial
    B_1}u_t^2\\
  &\leq \left(1-\frac{\eta}{2}\right)W_{3/2}^0(w_t)+O(t^\alpha)\int_{\partial B_1}
    u_t^2,\quad\text{for }0<t<t_0=t_0(n,\alpha).
\end{align*}
We can therefore write
\begin{align*}
  \frac{d}{dt}W_{3/2}(t,u)&\geq
                            -\frac{(n+1)(1-t^\alpha)}{t}W_{3/2}(t, u)\\
                          &\qquad+\frac{(n+1)e^{a
                            t^\alpha}(1-bt^\alpha)}{t}W_{3/2}^0(w_t)+O(t^{\alpha-1})\int_{\partial
                            B_1}u_t^2\\
                          &\geq\frac{n+1}{t}\left(-1+\frac{1}{1-\eta/2}+O(t^\alpha)\right)W_{3/2}(t, u)+\frac{O(t^\alpha)}{t^{n+3}}\int_{\partial
                            B_t} u^2\\
                          &\geq \frac{\eta}{4t}\,W_{3/2}(t, u)-Ct^{\alpha/2-1},
\end{align*}
for small $t$, where we have also used the growth estimate in
Lemma~\ref{lem:almost-opt-growth}. Taking now $\de$ such that
$$
0<\de<\min\left\{\frac{\eta}4,\frac{\alpha}{2}\right\},
$$
we have
\begin{multline*}
  \frac{d}{dt}\left[W_{3/2}(t,u)t^{-\de}+\frac{C}{\alpha/2-\de}
    t^{\alpha/2-\de}\right]\\
  \begin{aligned}
    &=t^{-\de}\left(\frac{d}{dt}W_{3/2}(t,u)-\frac{\de}{t}W_{3/2}(t,u)\right)+Ct^{\alpha/2-\de-1}\\
    &\geq t^{-\de-1}\left[\frac{\eta}{4}-\de\right]W_{3/2}(t,u)-Ct^{\alpha/2-\de-1}+Ct^{\alpha/2-\de-1}\\
    &\geq 0,
  \end{aligned}
\end{multline*}
for small $t$, where we have used again that $W_{3/2}(t,u)\geq
0$. Thus, we can conclude that
$$
0\leq W_{3/2}(t,u)\leq C t^\de,\quad 0<t<t_0=t_0(n,\alpha),
$$
with $C=C(n,\alpha)\|u\|_{W^{1,2}(B_1)}^2$.
\end{proof}

Using the estimate on $W_{3/2}(t,u)$ in Lemma~\ref{lem:W-gr-est}, we
can improve on Lemma~\ref{lem:almost-opt-growth} in the case
$\kappa=3/2$.
\begin{lemma}[Optimal growth estimate]
  \label{lem:opt-growth} Let $x_0\in B_{1/2}'$ be a free boundary
  point. Then, for $0<t<t_0=t_0(n,\alpha)$,
  \begin{align*}
    \int_{\partial B_t(x_0)} u^2
    &\leq C(n,\alpha)\|u\|_{W^{1,2}(B_1)}^2t^{n+2},\\
    \int_{B_t(x_0)} |\nabla u|^2&\leq C(n,\alpha)\|u\|_{W^{1,2}(B_1)}^2t^{n+1}.  
  \end{align*}
  
\end{lemma}

\begin{proof} We proceed as in the proof of
  Lemma~\ref{lem:almost-opt-growth} up to the estimate
  \begin{align*}
    |m(t)-m(s)|
    &\leq
      C_0\left(\log\frac{t}{s}\right)^{1/2}[W_{3/2}(t,u)-W_{3/2}(s,u)]^{1/2}.
  \end{align*}
  From there, using Lemma~\ref{lem:W-gr-est}, we now have an improved
  bound
  \begin{align*}
    |m(t)-m(s)|&\leq C\left(\log\frac{t}{s}\right)^{1/2}
                 t^{\de/2},\quad s<t<t_0,
  \end{align*}
  with $C=C(n,\alpha)\|u\|_{W^{1,2}(B_1)}$.  Then, by a dyadic
  argument, we can conclude that
$$
|m(t)-m(s)|\leq C t^{\de/2}.
$$
Indeed, let $k=0,1,2,\ldots$ be such that $t/2^{k+1}\leq
s<t/2^{k}$. Then,
\begin{align*}
  |m(t)-m(s)|
  &\leq \sum_{j=1}^k |m(t/2^{j-1})-m(t/2^j)|+|m(t/2^k)-m(s)|\\
  &\leq C(\log2)^{1/2}\sum_{j=1}^{k+1} (t/2^{j-1})^{\de/2}\leq
    C(\log2)^{1/2} \frac{t^{\de/2}}{1-2^{-\de/2}}=C t^{\de/2}.
\end{align*}
In particular, we have
$$
m(t)\leq m(t_0)+Ct_0^{\de/2}\leq C(n,\alpha)\|u\|_{W^{1,2}(B_1)},\quad
t<t_0.
$$
This implies the first bound. The second bound follows immediately
from the first one by using that $W_{3/2}(t, u)\leq W_{3/2}(t_0, u)$:
\begin{align*}
  \frac{1}{t^{n+1}}\int_{B_t}|\nabla u(x)|^2dx
  &\leq\frac{(3/2)(1-bt^\alpha)}{t^{n+2}}\int_{\partial B_t}u(x)^2dS_x+e^{-at^\alpha}W_{3/2}(t_0, u)\\
  &\leq
    C(n,\alpha)\|u\|_{W^{1,2}(B_1)}^2
    +\frac{e^{at_0^\alpha}}{t_0^{n+1}}\int_{B_{t_0}}|\nabla u(x)|^2dx\\*
  &\leq C(n,\alpha)\|u\|_{W^{1,2}(B_1)}^2.\qedhere
\end{align*}
\end{proof}


\section{$3/2$-Homogeneous blowups}
\label{sec:32-homog-blow}

For a free boundary point $x_0\in B'_{1/2}$, we consider again the
$3/2$-almost homogeneous rescalings
$$ u^{\phi}_{x_0, t}(x)=\frac{u(tx+x_0)}{\phi(t)}, \quad x\in B_{1/(2 t)},
$$
with $\phi=\phi_{3/2}$. We now observe that the optimal growth
estimates in Lemma~\ref{lem:opt-growth} implies the boundedness of
this family of rescalings in $W^{1,2}(B_R)$ for any $R>1$. Indeed, the
rescalings above will be defined in $B_R$ if $t<1/(2R)$, and by
Lemma~\ref{lem:opt-growth}, we will have
\begin{align*}
  \int_{B_R}|\D u^{\phi}_{x_0, t}|^2
  &=\frac{e^{{\frac{3b}{\al}t^\al}}}{t^{n+1}}\int_{B_{Rt}(x_0)}|\D u|^2\le C(n, \al)\|u\|^2_{W^{1, 2}(B_1)}R^{n+1},\\
  \int_{\pa B_R}(u^{\phi}_{x_0, t})^2
  &= \frac{e^{{\frac{3b}{\al}t^\al}}}{t^{n+2}}\int_{\pa B_{Rt}(x_0)}u^2\le C(n, \al)\|u\|^2_{W^{1, 2}(B_1)}R^{n+2},
\end{align*}
for $0<t<t_0/R$. Arguing as in the proof of
Proposition~\ref{prop:exist-Alm-blowup}, we have for a sequence
$t=t_j\ra 0+$
$$
u^{\phi}_{x_0, t_j}\ra u^{\phi}_{x_0, 0}\quad\text{in}\quad
C^1_{\loc}(B^{\pm}_R\cup B'_R).
$$
By letting $R\to \infty$ and using Cantor's diagonal argument, we
therefore have that over a subsequence $t=t_j\to 0+$
$$
u^{\phi}_{x_0, t_j}\ra u^{\phi}_{x_0, 0}\quad\text{in}\quad
C^1_{\loc}(\R^{n}_\pm\cup\R^{n-1}).
$$
We call such $u^{\phi}_{x_0, 0}$ a \emph{$3/2$-homogeneous blowup} of
$u$ at $x_0$. The name is explained by the fact that
$$
\lim_{t\to 0}\frac{\phi(t)}{t^{3/2}}=1,
$$
which implies that if we consider the $3/2$-homogeneous rescalings
$$
u_{x_0, t}^{(3/2)}(x)=\frac{u(tx+x_0)}{t^{3/2}},
$$
then we will have
$$
u^{\phi}_{x_0, 0}=\lim_{t_j\ra 0}u^{\phi}_{x_0, t_j}=\lim_{t_j\ra
  0}u_{x_0, t_j}^{(3/2)}=:u_{x_0, 0}^{(3/2)}
$$
and thus $u^{\phi}_{x_0, 0}=u_{x_0, 0}^{(3/2)}$.

\begin{remark}
  \label{rem:kappa-blowup}
  Because of the logarithmic term in the weak growth estimates in
  Lemma~\ref{lem:almost-opt-growth}, at the moment we are unable to
  consider $\kappa$-homogeneous blowups as above for frequencies other
  than $\kappa=3/2$. However, once the logarithmic term is removed,
  the same construction as for $\kappa=3/2$ applies. In particular, we
  note that in Lemma~\ref{lem:opt-est} we prove the optimal growth
  estimates for frequencies $\kappa=2m<\kappa_0$, $m\in\mathbb{N}$,
  enabling us to consider the $\kappa$-homogeneous blowups for these
  values of $\kappa$.
\end{remark}

We show next that the $3/2$-homogeneous blowups are unique at free
boundary points. This is achieved by the control on the ``rotation''
of the rescalings $u^\phi_{x_0, r}(x)$.

\begin{lemma}[Rotation estimate]
  \label{lem:rotation-est} Let $u$ be an almost minimizer for the
  Signorini problem in $B_1$, $x_0\in B_{1/2}'$ a free boundary point,
  and $\de$ as in Lemma~\ref{lem:W-gr-est}. Then for $\kappa=3/2$ and
  $\phi=\phi_{3/2}$
  $$
  \int_{\partial B_1} |u^\phi_{x_0, t}-u^\phi_{x_0, s}|\leq C
  t^{\de/2},\quad s<t<t_0=t_0(n,\alpha),
$$
for $C=C(n,\alpha)\|u\|_{W^{1,2}(B_1)}$.
\end{lemma}
\begin{proof} The proof uses computations similar to the proof of
  Lemma~\ref{lem:almost-opt-growth} combined with the growth estimated
  for $W_{3/2}(t, u)$ in Lemma~\ref{lem:W-gr-est}.  We assume $x_0=0$,
  and have
  \begin{align*}
    \int_{\partial B_1} |u^\phi_t-u^\phi_s|
    &\leq \int_{\partial B_1}
      \int_s^t\left|\frac{d}{dr}u_r^\phi\right|dr=
      \int_s^t\int_{\partial B_1}\left|\frac{d}{dr}u_r^\phi\right| dr\\
    &\leq C_n\int_s^t\left(\int_{\partial B_1}
      \left|\frac{d}{dr}u_r^\phi\right|^2\right)^{1/2}\\
    &\leq C_n
      \left(\int_{s}^{t}\frac{1}{r}dr\right)^{1/2}\left(\int_{s}^{t}r\int_{\partial
      B_1}\left|\frac{d}{dr}u_r^\phi\right|^2\right)^{1/2}\\
    &\leq C_ne^{c
      t^\alpha}\left(\log\frac{t}{s}\right)^{1/2}\left(\int_{s}^t\frac{d}{dr}W_{3/2}(r, u)dr\right)^{1/2}
      ,\qquad c=\frac{3b}{2\alpha}-\frac{a}{2},
  \end{align*}
  where we have re-used the computation made in the proof of
  Lemma~\ref{lem:almost-opt-growth}.  Thus, we obtain
  \begin{align*}
    \int_{\partial B_1} |u^\phi_t-u^\phi_s|
    &\leq C(n,\alpha)\left(\log\frac{t}{s}\right)^{1/2} (W_{3/2}(t,
      u)-W_{3/2}(s, u))^{1/2}\\
    &\leq C \left(\log\frac{t}{s}\right)^{1/2}t^{\de/2}. 
  \end{align*}
  Then, using a dyadic argument as Lemma~\ref{lem:opt-growth}, we can
  conclude that
  $$
  \int_{\partial B_1} |u^\phi_t-u^\phi_s|\leq C t^{\de/2},\quad
  s<t<t_0,
  $$
  as required. Indeed, let $k=0,1,2,\ldots$ be such that
  $t/2^{k+1}\leq s<t/2^{k}$. Then
  \begin{align*}
    \int_{\partial B_1} |u^\phi_t-u^\phi_s|
    &\leq \sum_{j=1}^{k}\int_{\partial B_1}
      \left|u^\phi_{t/2^{j-1}}-u^\phi_{t/2^j}\right|+\int_{\partial B_1}
      \left|u^\phi_{t/2^{k}}-u^\phi_{s}\right|\\
    &\leq
      C(\log2)^{1/2}\sum_{j=1}^{k+1} (t/2^{j-1})^{\de/2}\leq
      C(\log2)^{1/2} \frac{t^{\de/2}}{1-2^{-\de/2}}.
  \end{align*}
  This completes the proof.
\end{proof}

The uniqueness of $3/2$-homogeneous blowup now follows.

\begin{lemma}\label{lem:blowup-rot-est} Let $u_{x_0,0}^\phi$ be a
  blowup at a free boundary point $x_0\in B_{1/2}'$. Then for
  $\kappa=3/2$
  $$
  \int_{\partial B_1} |u_{x_0,t}^\phi - u_{x_0,0}^\phi|\leq C
  t^{\de/2},\quad 0<t<t_0,
$$
where $C=C\left(n,\alpha,\|u\|_{W^{1,2}(B_1)}\right)$ and
$\de=\de(n,\alpha)>0$ are as in Lemma~\ref{lem:rotation-est}. In
particular, the blowup $u_{x_0,0}^\phi$ is unique.
\end{lemma}
\begin{proof} If $u_{x_0,0}$ is the limit of $u^\phi_{x_0,t_j}$ for
  $t_j\to 0$, then first part of the lemma follows immediately from
  Lemma~\ref{lem:rotation-est}, by taking $s=t_j\to 0$ and passing to
  the limit.

  To see the uniqueness of blowup, we observe that $u_{x_0,0}^\phi$ is
  a solution of the Signorini problem in $B_1$, by arguing as in the
  proof of Lemma~\ref{lem:min-freq} for Almgren blowups. Now, if
  $\tilde{u}_{x_0,0}^\phi$ is another blowup, then from the first part
  of the lemma we will have
$$
\int_{\partial B_1} |\tilde{u}_{x_0,0}^\phi - u_{x_0,0}^\phi|^2=0,
$$
implying that both $\tilde{u}_{x_0,0}^\phi$ and $u_{x_0,0}^\phi$ are
solutions of the Signorini problem in $B_1$ with the same boundary
values on $\partial B_1$. By the uniqueness of such solutions, we have
$\tilde{u}_{x_0,0}^\phi=u_{x_0,0}^\phi$ in $B_1$. The equality
propagates to all of $\R^n$ by the unique continuation of harmonic
functions in $\R^n_\pm$.
\end{proof}

We next show that not only the blowups are unique, but also depend
continuously on a free boundary point.

\begin{lemma}[Continuous dependence of blowups]
  \label{lem:blowup-est} There exists $\rho=\rho(n,\alpha)>0$ such
  that if $x_0,y_0\in B_{\rho}'$ are free boundary points, then
$$
\int_{\partial B_1} |u_{x_0,0}^\phi - u_{y_0,0}^\phi|\leq
C|x_0-y_0|^\g,
$$
with $C=C\left(n,\alpha,\|u\|_{W^{1,2}(B_1)}\right)$ and
$\g=\g(n,\alpha)>0$.
\end{lemma}

\begin{proof} Let $d=|x_0-y_0|$ and $d^\mu\leq r\leq 2 d^\mu$ with
  $\mu\in(0,1]$ to be determined. By Lemma~\ref{lem:blowup-rot-est} we
  have
  \begin{align*}
    \int_{\partial B_1} |u_{x_0,0}^\phi - u_{y_0,0}^\phi|
    &\leq 2C r^{\de/2}
      +\int_{\partial B_1}|u_{x_0,r}^\phi-u_{y_0,r}^\phi|\\
    &\leq C d^{\mu\de/2}
      +\frac{C}{d^{\mu(n+1/2)}}\int_{\partial B_r}|u(x_0+z)-u(y_0+z)| dS_z
  \end{align*}
  and taking the average over $d^\mu\leq r\leq 2 d^\mu$, we have
  \begin{align*}
    \int_{\partial B_1} |u_{x_0,0}^\phi - u_{y_0,0}^\phi|
    &\leq C d^{\mu\de/2}
      +\frac{C}{d^{\mu(n+3/2)}}\int_{B_{2d^\mu}\setminus B_{d^\mu}}|u(x_0+z)-u(y_0+z)|dz.
  \end{align*}
  On the other hand, by using Lemma~\ref{lem:opt-growth},
  \begin{multline*}
    \int_{B_{2d^\mu}\setminus B_{d^\mu}}|u(x_0+z)-u(y_0+z)|dz\\
    \begin{aligned}
      &\leq \int_{B_{2d^\mu}\setminus B_{d^\mu}}\left|\int_0^1
        \frac{d}{ds}
        u(z+x_0(1-s)+y_0 s)ds\right|dz\\
      &\leq |x_0-y_0|\int_0^1\int_{B_{2d^\mu}}|\nabla
      u(z+x_0(1-s)+y_0s)|dz
      ds\\
      &\leq d\int_0^1\biggl(\int_{B_{2d^\mu}(x_0(1-s)+y_0s)}|\nabla u|\biggr)ds\\
      &\leq d\int_{B_{2d^\mu+d}(x_0)}|\nabla u|\leq
      d\int_{B_{3d^\mu}(x_0)}|\nabla
      u|\\
      &\leq C d^{1+\mu n/2}\biggl(\int_{B_{3d^\mu}(x_0)}|\nabla
      u|^2\biggr)^{1/2}\leq C d^{1+\mu n/2}d^{\mu(n+1)/2}\\
      &\leq C d^{1+\mu (n+1/2)},
    \end{aligned}
  \end{multline*}
  provided $3d^\mu<t_0$, which will hold if $d<\rho(n,\alpha)$.
  
  Combining the estimates, we infer that
  \begin{align*}
    \int_{\partial B_1} |u_{x_0,0}^\phi - u_{y_0,0}^\phi|
    \leq C d^{\mu\de/2}
    +C d^{1-\mu}.
  \end{align*}
  Now choosing $\mu$ so that $\mu \de/2=1-\mu$, that is
  $\mu=1/(1+\de/2)$, we obtain
  $$
  \int_{\partial B_1} |u_{x_0,0}^\phi - u_{y_0,0}^\phi| \leq C
  |x_0-y_0|^\g,\quad x_0,y_0\in B_\rho'
$$
with
\[
  \g=\frac{\de}{\de+2}.  \qedhere
\]
\end{proof}


\section{Regularity of the regular set}
\label{sec:regul-regul-set}

In this section we establish one of the main result of this paper, the
$C^{1, \g}$ regularity of the regular set. In fact, the most technical
part of the proof has already been done in the previous section, where
we proved the uniqueness of the $3/2$-homogeneous blowups, as well as
their H\"older continuous dependence on the free boundary points.

We start by defining the regular set.

\begin{definition}[Regular points]
  \label{def:reg-set}
  For an almost minimizer $u$ for the Signorini problem in $B_1$, we
  say that a free boundary point $x_0$ is \emph{regular} if
  $$
  \widehat{N}(0+,u,x_0)=3/2.
  $$
  Note that since $3/2<2\leq\kappa_0$, we will have that
  $\widehat{N}(r)<\kappa_0$ for small $r>0$, implying that
  $\N(r)=\widehat{N}(r)$ for such $r$ and consequently that
$$
N(0+)=\N(0+)=\widehat{N}(0+)=3/2.
$$
In particular, the condition above does not depend on the choice of $\kappa_0\geq 2$.

We denote the set of all regular points of $u$ by $\mathcal{R}(u)$ and
call it the \emph{regular set}.
\end{definition}

An important ingredient in the analysis of the regular set is the
following nondegeneracy lemma.

\begin{lemma}[Nondegeneracy at regular points]\label{lem:nondeg} Let  $x_0\in B'_{1/2}\cap\cR(u)$ for an almost
  minimizer $u$ for the Signorini problem in $B_1$. Then, for
  $\kappa=3/2$,
  $$
  \liminf_{t\to 0} \int_{\partial B_1} (u^\phi_{x_0,
    t})^2=\liminf_{t\to 0} \frac{1}{t^{n+2}}\int_{\partial B_t(x_0)}
  u^2>0.
  $$
\end{lemma}

\begin{proof} As before, assume $x_0=0$. In terms of the quantities
  defined in the proofs of Lemmas~\ref{lem:almost-opt-growth} and
  \ref{lem:opt-growth}, we want to prove that
  $$
  \liminf_{t\to 0} m(t)>0.
  $$
  Assume, towards a contradiction, that $m(t_j)\to 0$ for some
  sequence $t_j\to 0$. Recall that by the proof of
  Lemma~\ref{lem:opt-growth}, we have
$$
|m(t)-m(s)|\leq C t^{\de/2},\quad 0<s<t<t_0.
$$
Now, setting $s=t_j\to 0$, we conclude that
$$
|m(t)|\leq Ct^{\de/2},\quad 0<t<t_0.
$$
Equivalently, we can rewrite this as
$$
\int_{\partial B_t} u^2\leq C t^{n+2+\de}.
$$
Next, take $\tilde{\kappa}=3/2+\de/4$ and consider Weiss's
monotonicity formula
$$
W_{\tilde{\kappa}}(t, u)=\frac{e^{a_{\tilde{\kappa}}
    t^{\al}}}{t^{n+2\tilde{\kappa}-2}}\left[\int_{B_{t}}|\D
  u|^2-\tilde{\kappa}\frac{1-b t^\alpha}{t}\int_{\pa B_{t}}u^2\right].
$$
Now observe that
$$
\frac{1}{t^{n+2\tilde{\kappa}-1}}\int_{\pa B_{t}}u^2\leq Ct^{\de/2}\to
0,
$$
which readily implies that
$$
W_{\tilde{\kappa}}(0+, u)\geq 0.
$$
In particular, by monotonicity, $W_{\tilde{\kappa}}(t, u)\geq 0$, for
small $t>0$, which also implies that
$\widetilde N(t, u)\geq \tilde{\kappa}$. But then
$N(0+, u)=\widetilde N(0+, u)\geq \tilde{\kappa}=3/2+\de/4$ contrary
to the assumption in the lemma. This completes the proof.
\end{proof}

The next result provides two important facts: a gap in possible values
of Almgren's frequency $N(0+)$ as well as the classification of
$3/2$-homogeneous blowups.

\begin{proposition}\label{prop:gap-32-hom-classif}
  If $\widehat{N}(0+, u, x_0)=\ka<2$, then $\ka=3/2$ and
  $$
  u^{\phi}_{x_0, 0}(x)=a_{x_0} \Re(x'\cdot \nu_{x_0}+i|x_n|)^{3/2}$$
  for some $a_{x_0}>0$, $\nu_{x_0}\in \partial B_1'$.
\end{proposition}

\begin{proof}
  Without loss of generality, we may assume $x_0=0$. Let $r_j\to 0+$
  be a sequence such that $u_{r_j}^\phi\to u_{0}^\phi$ in
  $C^{1}_\loc(\R^n_\pm\cup \R^{n-1})$. Comparing $3/2$-almost
  homogeneous and Almgren rescalings, we have
$$
u_r^\phi(x)=u_r^A(x)\mu(r),\quad
\mu(r):=\frac{\left(\frac{1}{r^{n-1}}\int_{\partial B_r}
    u^2\right)^{1/2}}{\phi(r)}.
$$
By the optimal growth estimate (Lemma~\ref{lem:opt-growth}) and the
nondegeneracy at regular points (Lemma~\ref{lem:nondeg}) we have
$$
0<\liminf_{r\to 0+}\mu(r)\leq \limsup_{r\to 0+}\mu(r)<\infty.
$$
Thus, we may assume that, over a subsequence of $r_j$,
$\mu(r_j)\to\mu_0\in (0,\infty)$, and therefore
$$
u_{r_j}^\phi\to \mu_0 u_0^A\quad\text{in }C^{1}_{\loc}(B_1^\pm\cup
B_1'),
$$
where $u_0^A$ is an Almgren blowup of $u$ at $x_0=0$. Now, since
$\kappa<\kappa_0$, we can apply
Proposition~\ref{prop:exist-Alm-blowup} to obtain that $u_0^A$ is a
nonzero $\kappa$-homogeneous solution of the Signorini problem in
$B_1$, even in $x_n$-variable. Next, applying
Lemma~\ref{lem:min-freq}, we have $3/2\leq\kappa<2$ and thus by
\cite{PetShaUra12}*{Proposition~9.9}, we must have $\ka=3/2$ and
$$
u_0^A(x)=C_n\Re(x'\cdot\nu_{0}+i|x_n|)^{3/2}
$$ 
for some $C_n>0$, $\nu_0\in \partial B_1'$. (The constant $C_n$ comes
from the normalization $\int_{\partial B_1} (u_0^A)^2=1$.)  Thus,
$$
u_0^\phi(x)=a_0 \Re(x'\cdot\nu_0+i|x_n|)^{3/2}\quad\text{in }B_1
$$
with $a_0=C_n\mu_0$. By the unique continuation of harmonic functions
in $\R^n_\pm$, we obtain that the above formula for $u_0^\phi$
propagates to all of $\R^n$.
\end{proof}

Proposition~\ref{prop:gap-32-hom-classif} has an immediate corollary.

\begin{corollary}[Almgren frequency gap]
  \label{cor:gap} Let $u$ be an almost minimizer for the Signorini
  problem in $B_1$ and $x_0$ a free boundary point.  Then either
  $$
  \widehat N(0+,u)=3/2\quad\text{or}\quad \widehat N(0+,u)\geq 2.
  $$
\end{corollary}

Yet another important fact is as follows.

\begin{corollary} The regular set $\mathcal{R}(u)$ is a relatively
  open subset of the free boundary.
\end{corollary}
\begin{proof} For a fixed $0<t<t_0$, the mapping
  $x\mapsto \widehat N(t,u,x)$ is continuous on $\Gamma(u)$. Then, by
  the monotonicity of $\widehat N$, the mapping
  $x\mapsto \widehat{N}(0+,u,x_0)$ is upper semicontinuous on
  $\Gamma(u)$. Moreover, by Proposition~\ref{prop:gap-32-hom-classif},
  $$
  \cR(u)=\{x\in\Gamma(u): \widehat{N}(0+,u,x)<2\},
  $$
  which implies that $\cR(u)$ is relatively open in $\Gamma(u)$.
\end{proof}

The combination of Proposition~\ref{prop:gap-32-hom-classif} and
Lemma~\ref{lem:blowup-est} implies the following lemma.

\begin{lemma}\label{lem:blowup-hol}
  Let $u$ be an almost minimizer for the Signorini problem in $B_1$,
  and $x_0\in\mathcal{R}(u)$. Then there exists $\rho>0$, depending on
  $x_0$ such that $B_\rho'(x_0)\cap \Gamma(u)\subset \mathcal{R}(u)$
  and if
  $u^{\phi}_{\bx, 0}(x)=a_{\bx} \Re(x'\cdot\nu_{\bx}+i|x_n|)^{3/2}$ is
  the unique $3/2$-homogeneous blowup at
  $\bx \in B'_{\rho}(x_0)\cap \Gamma(u)$, then
  \begin{align*}
    |a_{\bx}-a_{\by}| \le C_0|\bx-\by|^{\g},\\
    |\nu_{\bx}-\nu_{\by}| \le C_0|\bx-\by|^{\g},
  \end{align*}
  for any $\bx, \by\in B'_{\rho}(x_0) \cap \Gamma(u)$ with a constant
  $C_0$ depending on $x_0$.
\end{lemma}
\begin{proof} The proof follows by repeating the argument in Lemma~7.5
  in \cite{GarPetSVG16}.
\end{proof}

Now we are ready to prove the main result on the regularity of the
regular set.

\begin{theorem}[$C^{1,\gamma}$ regularity of the regular set]
  \label{thm:C1g-regset}
  Let $u$ be an almost minimizer for the Signorini problem in
  $B_1$. Then, if $x_0\in B_{1/2}'\cap\mathcal{R}(u)$, there exists
  $\rho>0$, depending on $x_0$ such that, after a possible rotation of
  coordinate axes in $\R^{n-1}$, one has
  $B'_{\rho}(x_0)\cap \Gamma(u) \subset \mathcal{R}(u)$, and
  \begin{align*}
    B'_{\rho}(x_0) \cap \Gamma(u)
    &= B'_{\rho}(x_0)\cap \{x_{n-1}=g(x_1,\ldots, x_{n-2})\},
  \end{align*}
  for $g\in C^{1, \g}(\R^{n-2})$ with an exponent
  $\g=\g(n, \al)\in (0, 1)$.

\end{theorem}

\begin{proof}
  The proof of the theorem is similar to that of Theorem~1.2 in
  \cite{GarPetSVG16}. However, we provide full details since there are
  technical differences.

  \medskip\noindent \emph{Step 1.}  By relative openness of $\cR(u)$
  in $\Gamma(u)$, for small $\rho>0$ we have
  $B'_{2\rho}(x_0)\cap \Ga(u)\subset \cR(u)$.  We then claim that for
  any $\e>0$, there is $r_{\e}>0$ such that for
  $\bx\in B'_{\rho}(x_0)\cap \Ga(u)$, $r<r_{\e}$, we have that for
  $\phi=\phi_{3/2}$
  $$
  \|u^{\phi}_{\bx, r}-u^{\phi}_{\bx,
    0}\|_{C^1(\overline{B_1^{\pm}})}<\e.
  $$
  Assuming the contrary, there is a sequence of points
  $\bx_j\in B'_{\rho}(x_0)\cap \Ga(u)$ and radii $r_j\ra 0$ such that
  $$
  \|u^{\phi}_{\bx_j, r_j}-u^{\phi}_{\bx_j,
    0}\|_{C^1(\overline{B_1^{\pm}})}\ge \e_0
  $$
  for some $\e_0>0$. Taking a subsequence, if necessary, we may assume
  $\bx_j\ra \bx_0\in \overline{B'_{\rho}(x_0)}\cap \Ga(u)$.  Using
  estimates \eqref{eq:holder-6}, \eqref{eq:grad-holder-1} and
  Lemma~\ref{lem:opt-growth}, we can see that $u_{\bx_j, r_j}^\phi$
  are uniformly bounded in $C^{1, \be}(B_2^{\pm}\cup B'_2)$.  Thus, we
  may assume that for some $w$
  \begin{align*}
    u^{\phi}_{\bx_j, r_j}\ra w \quad \text{in }C^1(\overline{B_1^{\pm}}).
  \end{align*}
  By arguing as in the proof of
  Proposition~\ref{prop:exist-Alm-blowup}, we see that the limit $w$
  is a solution of the Signorini problem in $B_1$.  Further, by
  Lemma~\ref{lem:blowup-rot-est}, we have
$$
\|u^{\phi}_{\bx_j, r_j} - u^{\phi}_{\bx_j, 0}\|_{L^1(\pa B_1)}\ra 0.
$$
On the other hand, by Lemma~\ref{lem:blowup-hol}, we have
$$
u^{\phi}_{\bx_j,0}\ra u^{\phi}_{\bx_0, 0}\quad \text{in }
C^1(\overline{B^{\pm}_1}),
$$
and thus
$$
w=u^{\phi}_{\bx_0, 0}\quad \text{on }\partial B_1.
$$
Since both $w$ and $u^{\phi}_{\bx_0, 0}$ are solutions of the
Signorini problem, they must coincide also in $B_1$.  Therefore
$$
u^{\phi}_{\bx_j, r_j}\ra u^{\phi}_{\bx_0, 0}\quad \text{in
}C^1(\overline{B_1^{\pm}}),
$$
implying also that
$$
\|u^{\phi}_{\bx_j,
  r_j}-u^{\phi}_{\bx_j,0}\|_{C^1(\overline{B_1^{\pm}})}\to 0,
$$
which contradicts our assumption.

\medskip\noindent \emph{Step 2.}  As \cite{GarPetSVG16}, for a given
$\e>0$ and a unit vector $\nu\in \R^{n-1}$ define the cone
$$
\cC_{\e}(\nu)=\{x'\in \R^{n-1}: x'\cdot\nu > \e|x'|\}.
$$ By Lemma~\ref{lem:blowup-hol}, we may assume
$a_{\bx}\ge\frac{a_{x_0}}{2}$ for $\bx\in B'_{\rho}(x_0)\cap \Ga(u)$
by taking $\rho$ small. For such $\rho$ we then claim that for any
$\e>0$ there is $r_{\e}>0$ such that for any
$\bx\in B'_{\rho}(x_0)\cap \Ga(u)$ we have
$$\bx+\left(\cC_{\e}(\nu_{\bx})\cap B'_{r_{\e}}\right)\subset
\{u(\cdot, 0)>0\}.$$ Indeed, denoting
$\cK_{\e}(\nu)=\cC_{\e}\cap \partial B'_{1/2}$, we have for some
universal $C_{\e}>0$
\begin{align*}
  \cK_{\e}(\nu_{\bx})\Subset\{u^{\phi}_{\bx, 0}(\cdot,0)>0\}
  \cap B'_1\quad \text{and}\quad u^{\phi}_{\bx, 0}(\cdot, 0)\ge
  a_{\bx}C_{\e}\ge \frac{a_{x_0}}{2}C_{\e}\quad \text{on }
  \cK_{\e}(\nu_{\bx}). 
\end{align*}
Since $\frac{a_{x_0}}{2}C_{\e}$ is independent of $\bx$, by Step 1 we
can find $r_{\e}>0$ such that for $r<2r_{\e}$,
$$
u^{\phi}_{\bx, r}(\cdot, 0)>0\quad \text{on } \cK_{\e}(\nu_{\bx}).
$$
This implies that for $r<2r_{\e}$,
$$
u(\cdot, 0)>0\quad \text{on}\quad
\bx+r\cK_{\e}(\nu_{\bx})=\bx+\left(\cC_{\e}(\nu_{\bx})\cap \partial
  B'_{r/2}\right).
$$
Taking the union over all $r<2r_{\e}$, we obtain
$$
u(\cdot, 0)>0\quad \text{on }\bx+\left(\cC_{\e}(\nu_{\bx})\cap
  B'_{r_{\e}}\right).
$$

\medskip\noindent \emph{Step 3.}  We claim that for given $\e>0$,
there exists $r_{\e}>0$ such that for any
$\bar{x}\in B_{\rho}'(x_0)\cap \Gamma(u)$ we have
$\bx-\left(\cC_{\e}(\nu_{\bar{x}})\cap B'_{r_{\e}}\right)\subset
\{u(\cdot, 0)=0\}$.

Indeed, we first note that
$$
-\pa^+_{x_n}u^{\phi}_{\bx, 0}\ge a_{\bx}C_{\e} >
\left(\frac{a_{x_0}}{2}\right)C_{\e}\quad \text{on
}-\mathcal{K}_{\e}(\nu_{\bx})$$ for a universal constant
$C_{\e}>0$. From Step 1, there exists $r_{\e}>0$ such that for
$r<2r_{\e}$,
$$
-\pa^+_{x_n}u^{\phi}_{\bx, r}(\cdot, 0)>0\quad \text{on
}-\mathcal{K}_{\e}(\nu_{\bx}).
$$
By arguing as in Step 2, we obtain
$$
-\pa^+_{x_n}u(\cdot, 0)>0\quad\text{on } \bx-\left(\cC(\nu_{\bx})\cap
  B'_{r_{\e}}\right).
$$
By the complementarity condition in Lemma~\ref{lem:comp-cond}, we
therefore conclude that
$$
\bx-\left(\cC(\nu_{\bx})\cap B'_{r_{\e}}\right) \subset
\{-\pa^+_{x_n}u(\cdot, 0)>0\}\subset \{u(\cdot, 0)=0\}.
$$

\medskip\noindent \emph{Step 4.}  By rotation in $\R^{n-1}$ we may
assume $\nu_{x_0}=e_{n-1}$. For any $\e>0$, by
Lemma~\ref{lem:blowup-hol} again, we can take
$\rho_{\e}=\rho(x_0, \e)$, possibly smaller than $\rho$ in the
previous steps, such that
$$\cC_{2\e}(e_{n-1})\cap B'_{r_{\e}} \subset \cC_{\e}(\nu_{\bx})\cap
B'_{r_{\e}}\quad\text{for }\bx\in B'_{\rho_{\e}}(x_0)\cap \Ga(u). $$
By Step 2 and Step 3, for $\bx\in B'_{\rho_{\e}}(x_0)\cap \Ga(u)$,
\begin{align*}
  \bx+\left(\cC_{2\e}(e_{n-1})\cap B'_{r_{\e}}\right) \subset \{u(\cdot, 0)>0\},\\
  \bx-\left(\cC_{2\e}(e_{n-1})\cap B'_{r_{\e}}\right) \subset \{u(\cdot, 0)=0\}.
\end{align*}
Now, fixing $\e=\e_0$, by the standard arguments, we conclude that
there exists a Lipschitz function $g: \R^{n-2}\ra \R$ with
$|\D g|\le C_n/\e_0$ such that
\begin{align*}
  B'_{\rho_{\e_0}}(x_0) \cap \{u(\cdot, 0)=0\}=B'_{\rho_{\e_0}}(x_0)\cap \{x_{n-1}\le g(x'')\},\\
  B'_{\rho_{\e_0}}(x_0) \cap \{u(\cdot,
  0)>0\}=B'_{\rho_{\e_0}}(x_0)\cap
  \{x_{n-1}> g(x'')\}.
\end{align*}

\medskip\noindent \emph{Step 5.}  Taking $\e\ra 0$ in Step 4, $\Ga(u)$
is differentiable at $x_0$ with normal $\nu_{x_0}$. Recentering at any
$\bx\in B'_{\rho_{\e_0}}(x_0)\cap \Ga(u)$, we see that $\Ga(u)$ has a
normal $\nu_{\bx}$ at $\bx$. By Lemma~\ref{lem:blowup-hol}, we
conclude that $g$ in Step 4 is $C^{1, \g}$. This completes the proof
of the theorem.
\end{proof}


\section{Singular points}
\label{sec:singular-points}

In this section we study the set of so-called singular free boundary
points. An important technical tool to accomplish this is the
logarithmic epiperimetric inequality of \cite{ColSpoVel17}. We use it
for two purposes: to establish the optimal growth at singular points
as well as the rate of convergence of rescalings to blowups,
ultimately implying a structural theorem for the singular set.

\begin{definition}[Singular points]
  Let $u$ be an almost minimizer for the Signorini problem in
  $B_1$. We say that a free boundary point $x_0$ is \emph{singular} if
  the coincidence set $\Lambda(u)=\{u(\cdot,0)=0\}\subset B_1'$ has
  zero $H^{n-1}$-density at $x_0$, i.e.,
$$
\lim_{r\ra 0+}\frac{H^{n-1}\left(\Ld(u)\cap
    B_r'(x_0)\right)}{H^{n-1}(B'_r(x_0))}=0.
$$
By using Almgren's rescalings $u_{x_0,r}^A$, we can rewrite this
condition as
$$
\lim_{r\ra 0+}H^{n-1}(\Ld(u_{x_0, r}^A)\cap B'_1)=0.
$$
We denote the set of all singular points by $\Sigma(u)$ and call it
the \emph{singular set}.
\end{definition}

Throughout the section we will assume that
$$
\kappa_0>2.
$$
We can take $\kappa_0$ as large as we like, however, we have to
remember that the constants in $\widehat{N}=\widehat{N}_{\kappa_0}$
and $W_\kappa$ do depend on $\kappa_0$.

We then have the following characterization of singular points,
similar to Proposition~9.22 in \cite{PetShaUra12} for the solutions of
the Signorini problem.

\begin{proposition}[Characterization of singular points]\label{prop:sing-char}
  Let $u$ be an almost minimizer for the Signorini problem in $B_1$,
  and $x_0\in B_{1/2}'\cap \Gamma(u)$ be such that
  $\widehat N(0+, u, x_0)=\ka<\kappa_0$. Then the following statements
  are equivalent.
  \begin{enumerate}[label=\textup{(\roman*)},leftmargin=*,widest=iii]
  \item $x_0\in \Sigma(u)$.
  \item any Almgren blowup of $u$ at $x_0$ is a nonzero polynomial
    from the class
    \begin{multline*}
      \mathcal{Q}_{\ka}=\{q: \text{$q$ is homogeneous polynomial of
        degree
        $\kappa$ such that}\\
      \La q=0,\ q(x', 0)\ge 0,\ q(x', x_n)=q(x', -x_n)\}.
    \end{multline*}
  \item $\ka=2m$ for some $m\in \mathbb{N}$.
  \end{enumerate}
\end{proposition}
Note that for $\kappa<\kappa_0$, the condition
$\widehat{N}(0+)=\kappa$ is equivalent to ${N}(0+)=\kappa$.

\begin{proof}
  Without loss of generality we may assume $x_0=0$. By
  Proposition~\ref{prop:exist-Alm-blowup}, any Almgren blowup $u_0^A$
  of $u$ at $0$ is a nonzero global solution of the Signorini problem,
  homogeneous of degree $\ka$. Moreover $u_0^A$ is a
  $C^1_{\text{loc}}$ limit of Almgren rescalings $u_{r_j}^A$ in
  $B_1^\pm\cup B_1'$. Because of that, most parts of the proof of this
  proposition are just the repetitions of Proposition~9.22 in
  \cite{PetShaUra12}. Thus, by following Proposition~9.22 in
  \cite{PetShaUra12}, we can easily see the implications (ii)
  $\Rightarrow$ (iii), (iii) $\Rightarrow$ (ii), (ii) $\Rightarrow$
  (i). Moreover, in the proof of the remaining implication (i)
  $\Rightarrow$ (ii), the only nontrivial part is that any blowup
  $u_0^A$ is harmonic in $B_1$. But this comes from the
  complementarity condition in Lemma~\ref{lem:comp-cond}. Indeed,
  assuming (i), we claim that
$$
\pa_{x_n}^+u_0^A=0\quad\text{in}\quad B'_1.
$$
Otherwise,
$$
H^{n-1}\left(\{-\pa_{x_n}^+u_0^A(\cdot, 0)>0\}\cap B'_1\right)\ge\de
$$ for some $\de>0$. Then using the continuity from the below we also have that for some $\rho>0$, 
$$
H^{n-1}\left(\{-\pa_{x_n}^+u_0^A(\cdot, 0)>\rho\}\cap
  B'_{1-\rho}\right)\ge\de/2.
$$
Using $C^1_{\text{loc}}$ convergence $u_{r_j}^A\ra u_0^A$ in
$B_1^\pm\cup B_1'$ and applying the complementarity condition in
Lemma~\ref{lem:comp-cond} to rescalings $u_{r_j}^A$, we obtain that
for small $r_j$,
$$ H^{n-1}\left(\Ld(u_{r_j}^A)\cap B'_1\right)\ge
H^{n-1}\left(\{-\pa_{x_n}^+u_{r_j}^A(\cdot, 0)>0\}\cap B'_1\right)\ge
\de/4,
$$
which contradicts (i).  Now recalling that $u_0^A$ is a solution of
the Signorini problem, even in $x_n$-variable, it satisfies
$$
\La u_0^A=2(\pa_{x_n}^+u_0^A)H^{n-1}|_{\Ld(u_0^A)}=0\quad
\text{in}\quad B_1.
$$
By homogeneity, we obtain that $u_0^A$ is harmonic in all of $\R^n$,
and we complete the proof as in \cite{PetShaUra12}.
\end{proof}

In order to study the singular set, in view of
Proposition~\ref{prop:sing-char}, we need to refine the growth
estimate in Lemma~\ref{lem:almost-opt-growth} by removing the
logarithmic term in the case when $\kappa=2m<\kappa_0$,
$m\in\mathbb{N}$. In the case $\kappa=3/2$ we were able to do so by
proving a decay estimate for $W_{3/2}$ with the help of the
epiperimetric inequality. In the case $\kappa=2m$ we will use the
so-called \emph{logarithmic epiperimetric inequality} for the Weiss
energy
$$
W^0_{\ka}(w)=\int_{B_1}|\D w|^2-\ka\int_{\pa B_1}w^2,\quad \kappa=2m,\
m\in\mathbb{N}
$$
that first appeared in \cite{ColSpoVel17}. To state this result, we
recall the notation
$$
\mathcal{A}=\{w\in W^{1, 2}(B_1): w\ge 0\text{ on }B'_1,\ w(x',
x_n)=w(x', -x_n)\}.
$$

\begin{theorem}[Logarithmic epiperimetric inequality]
  \label{thm:log-epi}
  Let $\kappa=2m$, $m\in\mathbb{N}$ and $w\in\mathcal{A}$ be
  homogeneous of degree $\kappa$ in $B_1$ such that
  $w\in W^{1,2}(\partial B_1)$ and
  $$
  \int_{\partial B_1} w^2\leq 1,\quad |W_{\kappa}^0(w)|\leq 1.
$$
There is constant $\e=\e(n,\kappa)>0$ and a function
$v\in \mathcal{A}$ with $v=w$ on $\partial B_1$ such that
$$
W^0_{\kappa}(v)\leq
W^0_\kappa(w)(1-\e|W^0_{\kappa}(w)|^\gamma),\quad\text{where
}\gamma=\frac{n-2}{n}.
$$
\end{theorem}

To simplify the notations, in the results below all constants will
depend on $n$, $\alpha$, $\kappa$, $\kappa_0$, as well as
$\|u\|_{W^{1,2}(B_1)}$, unless stated otherwise, in addition to other
quantities. Thus, when we write $C=C(\sigma)$, we mean
$C=C(n,\alpha,\kappa,\kappa_0,\|u\|_{W^{1,2}(B_1)},\sigma)$.

\medskip The next lemma allows to apply the logarithmic epiperimetric
inequality, without the constraints.

\begin{lemma}\label{lem:epi-ineq} Let $u$ be an almost minimizer for
  the Signorini problem in $B_1$ such that $0\in \Gamma(u)$ and
  $\widehat{N}(0+, u)=\ka<\kappa_0$, $\ka=2m$, $m\in\mathbb{N}$. For
  $0<r<1$, let
  $$
  u_r(x)=u_r^{(\kappa)}(x)=\frac{u(rx)}{r^{\ka}},\quad
  w_r(x)=|x|^{\ka}u_r\left(\frac x{|x|}\right).
  $$
  Suppose that for a given $0\leq\sigma\leq 1$, there is $C=C(\sigma)$
  such that
  $$
  \int_{\pa B_r}u^2\le C\left(\log \frac 1r\right)^{\si}r^{n+2\ka-1}.
$$
Then there is a constant $\e=\e(\sigma)>0$ and $h\in \mathcal{A}$ with
$h=w_r$ on $\pa B_1$ such that
\begin{enumerate}[label=\textup{(\roman{*})},leftmargin=*,widest=ii]
\item If $|W^0_{\ka}(w_r)| \ge \int_{\pa B_1}w_r^2$, then
  $$
  W_{\ka}^0(h)\le \left(1-\e \right)W^0_{\ka}(w_r)
  $$
\item If $|W^0_{\ka}(w_r)| \le 2\int_{\pa B_1}w_r^2$, then
  $$W_{\ka}^0(h)\le W_{\ka}^0(w_r)\left(1-\e\left(\log\frac{1}{r}\right)^{-\si\g}|W_{\ka}^0(w_r)|^{\g}\right),\quad\text{where }\g=\frac{n-2}n.$$
\end{enumerate}
\end{lemma}

\begin{proof} Let $A= \int_{\pa B_1}w_r^2\,+|W_{\ka}^0(w_r)|$. Then by
  Theorem~\ref{thm:log-epi} applied to $w_r/A^{1/2}$, there is
  $h\in \mathcal{A}$ such that $h=w_r$ on $\pa B_1$ and
$$
W_{\ka}^0(h)\le W_{\ka}(w_r)^0\left(1-\e
  A^{-\g}|W_{\ka}^0(w_r)|^{\g}\right).
$$
If $|W^0_{\ka}(w_r)| \ge \int_{\pa B_1}w_r^2$, then
$A\le 2|W^0_{\ka}(w_r)|$, implying
$$
W_{\ka}^0(h)\le W^0_{\ka}(w_r)\left(1-\e 2^{-\g}\right).
$$
If $|W^0_{\ka}(w_r)| \le 2\int_{\pa B_1}w_r^2$, then
$$ A\le 3\int_{\pa B_1}w_r^2=\frac{3}{r^{n+2\ka-1}}\int_{\pa
  B_r}u^2\le C(\sigma)\left(\log\frac{1}{r}\right)^{\si}.
$$ This completes the proof.
\end{proof}

Now we show that the logarithmic epiperimetric inequality, combined
with a growth estimate for $u$, implies a growth estimate on
$W_{\kappa}(t,u)$.  This is the first part of a bootstrapping argument
that gradually decreases the power of $\log(1/t)$ in the bound for
$u$.

\begin{lemma}\label{lem:W-est} Let $u$ be an almost minimizer for the
  Signorini problem in $B_1$ such that $0\in \Gamma(u)$ and
  $\widehat{N}(0+, u)=\ka<\kappa_0$, $\ka=2m$, $m\in
  \mathbb{N}$. Suppose that for some $0\leq\sigma\leq1$
$$
\int_{\pa B_r}u^2\le C(\sigma)\left(\log \frac
  1r\right)^{\si}r^{n+2\ka-1},\quad 0<r<r_0(\sigma).
$$
Then,
$$
0\le W_{\ka}(t, u)\le C(\sigma)\left(\log \frac 1
  t\right)^{-\frac{1-\si\g}{\g}}, \quad 0<t<t_0(\sigma).
$$
\end{lemma}

\begin{proof} We first observe that $W_{\ka}(t, u)\geq 0$ for
  $0<t<t_0$, which follows easily from the condition
  $\widehat N(0+,u)=\ka<\kappa_0$, see the beginning of the proof of
  Lemma~\ref{lem:almost-opt-growth}.

  Next, recall that in the proof of Lemma~\ref{lem:W-gr-est}, we have
  used epiperimetric inequality to show that
  $0\le W_{3/2}(t, u)\le Ct^{\de}$. This followed by obtaining a
  differential inequality for $W_{3/2}$. Thus, if for $0<t<t_0$, if
  alternative (i) holds in Lemma~\ref{lem:epi-ineq}, i.e.,\
  $W_{\ka}^0(h)\le (1-\e)W^0_{\ka}(w_t)$, by arguing in the same way,
  we can show that
  \begin{equation}\label{eq:ddtW-case-1}
    \frac{d}{dt}W_\kappa(t,u)\geq \frac{\e/4}{t}W_{\kappa}(t,u)-C t^{\alpha/2-1},
  \end{equation}
  for $C=C(\sigma)$.

  Suppose now the alternative (ii) holds in Lemma~\ref{lem:epi-ineq}
  for some $0<t<t_0$. Then, following the computations in
  Lemma~\ref{lem:W-gr-est}, we have
  \begin{align*}
    \frac d{dt}W_{\ka}(t, u)\ge
    &-\frac{(n+2\ka-2)(1-t^{\al})}{t}W_{\ka}(t, u)\\
    &\qquad +\frac{e^{at^{\al}}(1-bt^{\al})}t\int_{\pa B_1}(\pa_{\nu}u_t-\ka u_t)^2+(\pa_{\tau}u_t)^2-\ka(n+\ka-2)u_t^2\\
    &\qquad +(2\ka_0+n)t^{\al-1}\int_{\pa B_1}u_t^2.
  \end{align*}
  For $w_t$ as in the statement of Lemma~\ref{lem:epi-ineq}, by
  following the computations in the proof of Theorem~\ref{thm:weiss},
  we have the identity
$$
\int_{\pa
  B_1}(\pa_{\tau}u_t)^2-\ka(n+\ka-2)u_t^2=(n+2\ka-2)W^0_{\ka}(w_t).
$$
This gives
\begin{multline}\label{eq:W-grow-est-1}
  \frac d{dt}W_{\ka}(t, u)\ge-\frac{(n+2\ka-2)(1-t^{\al})}{t}W_{\ka}(t, u)\\
  +\frac{e^{at^{\al}}(1-bt^{\al})}t(n+2\ka-2)W^0_{\ka}(w_t)+(2\ka_0+n)t^{\al-1}\int_{\pa
    B_1}u_t^2.
\end{multline}
Let now $v_t$ be the solution of the Signorini problem in $B_1$ with
$v_t=u_t=w_t$ on $\pa B_1$. Then
\begin{equation}\label{eq:W-grow-est-2}
  \begin{aligned}
    (1+t^{\al})W_{\ka}^0(w_t) &\ge (1+t^{\al})W_{\ka}^0(v_t)\\
    &\ge \int_{B_1}|\D u_t|^2-\kappa(1+t^{\al})\int_{\pa B_1}u_t^2\\
    &= W^0_{\ka}(u_t)-\kappa t^{\al}\int_{\pa B_1}u_t^2\\
    &= e^{-at^{\al}}W_{\ka}(t, u)-\ka(b+1)t^{\al}\int_{\pa B_1}u_t^2.
  \end{aligned}
\end{equation}
Now, if
$$
e^{-at^{\al}}W_{\ka}(t, u)-\ka(b+1)t^{\al}\int_{\pa B_1}u_t^2\le 0,
$$
then by Lemma~\ref{lem:almost-opt-growth} we have
\begin{align}\label{eq:ddtW-case15}
  W_{\ka}(t, u)&\le e^{at^{\al}}\ka(b+1)t^{\al}\int_{\pa B_1}u_t^2\\
               &\le Ct^\alpha\left(\log\frac1t\right)\leq C t^{\alpha/2}.\notag
\end{align}
We then proceed under the assumption
$$
e^{-at^{\al}}W_{\ka}(t, u)-\ka(b+1)t^{\al}\int_{\pa B_1}u_t^2> 0,
$$
which also implies $$W_{\ka}^0(w_t)>0.$$ Now, applying
Lemma~\ref{lem:epi-ineq}, we have
\begin{equation}
  \begin{aligned}
    &W^0_{\ka}(w_t)\ge
    W_{\ka}^0(v_t)+\e\left(\log\frac{1}t\right)^{-\si\g}W^0_{\ka}(w_t)^{\g+1}\\
    &\ge \frac 1{1+t^{\al}}\left[e^{-at^{\al}}W_{\ka}(t, u)-\ka(b+1)t^{\al}\int_{\pa B_1}u_t^2\right]\\
    &\qquad+\e\left(\log \frac{1}t\right)^{-\si\g}\left(\frac 1{1+t^{\al}}\right)^{\g+1}\times\\&\qquad\qquad\times\left[e^{-at^{\al}}W_{\ka}(t, u)-\ka(b+1)t^{\al}\int_{\pa B_1}u_t^2\right]^{\g+1}\\
    &\ge (1-t^{\al})\left[e^{-at^{\al}}W_{\ka}(t, u)-\ka(b+1)t^{\al}\int_{\pa B_1}u_t^2\right]\\
    &\qquad +\e\left(\log \frac{1}t\right)^{-\si\g}(1-t^{\al})^{\g+1}
    \times\\&\qquad\qquad\times
    \left[\frac{\left(e^{-at^{\al}}W_{\ka}(t, u)\right)^{\g+1}}{2^{\g}}-\left( \ka(b+1)t^{\al}\int_{\pa B_1}u_t^2  \right)^{\g+1}\right]\\
    &=(1-t^{\al})e^{-at^{\al}}W_{\ka}(t, u)\\
    &\qquad+\frac{\e}{2^{\g}}\left(\log \frac{1}t\right)^{-\si\g}(1-t^{\al})^{\g+1}e^{-a(\g+1)t^{\al}}W_{\ka}(t, u)^{\g+1}\\
    &\qquad -(1-t^{\al})\ka(b+1)t^{\al}\int_{\pa B_1}u_t^2\\
    &\qquad-\e\left(\log\frac{1}t\right)^{-\si\g}(1-t^{\al})^{\g+1}\ka^{\gamma+1}(b+1)^{\g+1}t^{\al(\g+1)}\left(\int_{\pa
        B_1}u_t^2\right)^{\g+1},\label{eq:W-grow-est-3}
  \end{aligned}
\end{equation}
where we used \eqref{eq:W-grow-est-2} in the second inequality and the
convexity of $x\mapsto x^{\g+1}$ on $\R_+$ in the third inequality.
Now \eqref{eq:W-grow-est-1} and \eqref{eq:W-grow-est-3}, together with
Lemma~\ref{lem:almost-opt-growth}, yield
\begin{multline}\label{eq:ddtW-case-2}
  \frac d{dt}W_{\ka}(t, u)\ge -C_1t^{\al-1}W_{\ka}(t,
  u)\\+C_2t^{-1}\left(\log \frac 1 t\right)^{-\si\g}W_{\ka}(t,
  u)^{\g+1}-C_3t^{\al/2-1},
\end{multline}
where $C_i=C_i(\sigma)$. Summarizing, we have that at every
$0<t<t_0(\sigma)$, either \eqref{eq:ddtW-case-1},
\eqref{eq:ddtW-case-2}, or the bound \eqref{eq:ddtW-case15}
holds. Further note that by the growth estimate in
Lemma~\ref{lem:almost-opt-growth}, the bound \eqref{eq:ddtW-case-1}
implies \eqref{eq:ddtW-case-2} for sufficiently small $t$ and thus we
may assume that \eqref{eq:ddtW-case-2} holds for all $0<t<t_0$ for
which $W_\kappa(t,u)> C t^{\alpha/2}$.

To proceed, let $0<t<t_0$ be such that
$W_\kappa(t,u)\geq t^{\alpha/8}$. Then the bound
\eqref{eq:ddtW-case-2} holds and we can derive that for
$C=\frac{\g C_2}{2(1-\si\g)}$, we have
\begin{align*}
  &\frac d{dt}\left(-W_{\ka}(t, u)^{-\g}e^{-t^{\al/4}}+C\left(\log \frac 1 t\right)^{1-\si\g}\right)\\
  &\qquad= W_{\ka}(t, u)^{-\g-1}e^{-t^{\al/4}}
    \left(\g\frac{d}{dt}W_{\ka}(t, u)+\frac{\al}4 W_{\ka}(t, u)t^{\al/4-1}\right)\\
  &\qquad\qquad-C(1-\si\g)t^{-1}\left(\log \frac 1 t\right)^{-\si\g}\\
  &\qquad\ge W_{\ka}(t, u)^{-\g}e^{-t^{\al/4}}t^{\al/4-1}\left(\frac{\al}4-\g C_1t^{3\al/4}-\frac{\g C_3t^{\al/4}}{W_{\ka}(t, u)}\right)\\
  &\qquad\qquad+\left(\log\frac 1t\right)^{-\si\g}t^{-1}\left(e^{-t^{\al/4}}\g C_2-C(1-\si\g)\right)\\
  &\qquad\ge 0,
\end{align*}
$0<t<t_0=t_0(\sigma)$. Since also the function
$-{t^{-\gamma(\alpha/8)}}e^{-t^{\al/4}}+C\left(\log \frac 1
  t\right)^{1-\si\g}$ is nondecreasing for small $t$, denoting
$$\widehat{W}_{\ka}(t, u)=\max\{W_{\ka}(t, u),t^{\alpha/8}\},$$
we obtain that the function
$$
-\widehat{W}_{\ka}(t, u)^{-\g}e^{-t^{\al/4}}+C\left(\log \frac 1
  t\right)^{1-\si\g}
$$
is nondecreasing on $(0,t_0)$. Hence,
\begin{align*}
  -\widehat{W}_{\ka}(t, u)^{-\g}e^{-t^{\al/4}}+C\left(\log\frac
  1t\right)^{1-\si\g}
  &\le -\widehat{W}_{\ka}(t_0,u)^{-\g}e^{-t_0^{\al/4}}+C\left(\log\frac
    1{t_0}\right)^{1-\si\g}\\
  &\le C\left(\log\frac 1{t_0}\right)^{1-\si\g}.
\end{align*}
If $0<t<t_0^2$, then
$\left(\log\frac 1{t_0}\right)^{1-\si\g}<\left(\frac
  12\right)^{1-\si\g}\left(\log\frac 1t\right)^{1-\si\g}$, implying
that
$$
-\widehat{W}_{\ka}(t, u)^{-\g}e^{-t^{\al/4}}\le
C\left(\left(1/2\right)^{1-\si\g}-1\right)\left(\log\frac 1
  t\right)^{1-\si\g}
$$
and hence
\[
  W_{\ka}(t, u)\leq \widehat{W}_{\ka}(t, u)\le
  C\left(1-\left(1/2\right)^{1-\si\g}\right)^{-\frac1\g}\left(\log\frac
    1t\right)^{-\frac{1-\si\g}{\g}}.\qedhere
\]
\end{proof}

\begin{lemma}\label{lem:boun-est}
  If $u$ is as in Lemma~\ref{lem:W-est} with $\frac{2}{n-2}<\si\le 1$,
  then there exist positive $C=C(\sigma)$, $t_0=t_0(\sigma)$ such
  that
  $$ \int_{\pa B_t}u^2\le C\left(\log \frac 1
    t\right)^{\si-\frac{2}{n-2}}t^{n+2\kappa-1},\quad 0<t<t_0.
$$
\end{lemma}

\begin{proof}
  Going back to the proof and notations of
  Lemma~\ref{lem:almost-opt-growth}, we have that for $0<s<t<t_0$
$$
|m(t)-m(s)|\le C\left(\log\frac t
  s\right)^{1/2}\left(W_{\ka}(t)-W_{\ka}(s)\right)^{1/2}.
$$
Let now $0\le j\le i$ be such that $2^{-2^{i+1}}<t\le 2^{-2^{i}}$,
$2^{-2^{j+1}}<t_0\le 2^{-2^{j}}$. Then
\begin{multline*}
  |m(t_0)-m(t)|\\*
  \begin{aligned}
    &\le |m(t_0)-m(2^{-2^{j+1}})|+|m(2^{-2^{i}})-m(t)|+\sum_{k=j+1}^{i-1}|m(2^{-2^{k}})-m(2^{-2^{k+1}})|\\
    &\le \sum_{k=0}^i C\left[\log \left(2^{-2^{k}}\right)-\log \left(2^{-2^{k+1}}\right)\right]^{1/2}\left[W_{\ka}\left(2^{-2^{k}}\right)-W_{\ka}\left(2^{-2^{k+1}}\right)\right]^{1/2}\\
    &\le C\sum_{k=0}^i 2^{k/2}W_{\ka}\left(2^{-2^{k}}\right)^{1/2}\\
    &\le C\sum_{k=0}^i 2^{(1-\frac{1-\si\g}{\g})k/2}\\
    &\le C2^{(\si-\frac{2}{n-2})i/2}\\
    &\le C\left(\log \frac1t\right)^{\frac12(\si-\frac{2}{n-2})}.
  \end{aligned}
\end{multline*}
Note that in the fifth inequality we have used that
$1-\frac{1-\si\g}{\g}= \si-\frac{2}{n-2}>0$.  Thus
$$
m(t)\le
m(t_0)+C\left(\log\frac1t\right)^{\frac12(\si-\frac{2}{n-2})}\le
C\left(\log\frac1t\right)^{\frac12(\si-\frac{2}{n-2})}.
$$
This implies the desired result.
\end{proof}

Lemma~\ref{lem:W-est} and Lemma~\ref{lem:boun-est} imply the
following.

\begin{corollary}[Bootstraping]
  \label{cor:boun-est}
  Let $u$ be an almost minimizer for the Signorini problem in $B_1$
  such that $0\in\Gamma(u)$ and $\widehat N(0+, u)=\ka<\kappa_0$,
  $\ka=2m$, $m\in \mathbb{N}$. Suppose that for
  $\frac 2{n-2}<\si\le 1$
$$
\int_{\pa B_t}u^2\le C(\sigma)\left(\log\frac
  1t\right)^{\si}t^{n+2\ka-1},\quad 0<t<t_0(\sigma).
$$
Then
$$
\int_{\pa B_t}u^2\le C'(\sigma)\left(\log\frac 1t\right)^{\si-\frac
  2{n-2}}t^{n+2\ka-1},\quad 0<t<t_0'(\sigma).
$$
\end{corollary}

\begin{lemma}[Optimal growth estimate at sigular points]
  \label{lem:opt-est}
  Let $u$ be an almost minimizer for the Signorini problem in $B_1$
  such that $0\in\Gamma(u)$ and $\widehat N(0+, u)=\ka<\kappa_0$,
  $\ka=2m$, $m\in \mathbb{N}$. Then, for $0<t<t_0$,

\begin{align*}
  &\int_{\pa B_t}u^2\le Ct^{n+2\ka-1},\\
  &\int_{B_t}|\D u|^2\le Ct^{n+2\ka-2}.
\end{align*}
\end{lemma}

\begin{proof}
  Starting with $\sigma=1$ in Lemma~\ref{lem:almost-opt-growth} and
  repeatedly applying Corollary~\ref{cor:boun-est}, we find
  $0<\si\leq \min\{\frac{2}{n-2},1\}$ such that
$$
\int_{\pa B_t}u^2\le C\left(\log\frac 1t\right)^{\si}t^{n+2\ka-1},
\quad 0<t<t_0.
$$
In fact, we can make $\sigma$ to be strictly less than $\frac{2}{n-2}$
by noticing that in Lemma~\ref{lem:boun-est} we can replace
$\frac{2}{n-2}$ by any smaller positive number.  Then by
Lemma~\ref{lem:W-est}
$$
0\le W_{\ka}(t, u)\le C\left(\log \frac 1
  t\right)^{-\frac{1-\si\g}{\g}}.$$ Recall also that for $0<s<t<t_0$
$$
|m(t)-m(s)|\le C\left(\log \frac t
  s\right)^{1/2}\left(W_{\ka}(t)-W_{\ka}(s)\right)^{1/2}.
$$ 
We then again consider the exponentially dyadic decomposition as in
the proof of Lemma~\ref{lem:boun-est}.  Let $0\le j\le i$ be such that
$2^{-2^{i+1}}\le s/{t_0}<2^{-2^{i}}$ and
$2^{-2^{j+1}}\le t/{t_0}<2^{-2^{j}}$. Then,
\begin{equation}
  \label{eq:opt-est-1}
  \begin{aligned}
    |m(t)-m(s)|&\le C\sum_{k=j}^i2^{k/2}W_{\ka}(2^{-2^k}t_0)^{1/2}\\
    &\le C\sum_{k=j}^{\infty}2^{\left(1-\frac{1-\si\g}{\g} \right)k/2}\\
    &\le C2^{\left(\si-\frac 2{n-2}\right)j/2}\\
    &\le C\left(\log\frac 1t\right)^{\left(\si-\frac 2{n-2}\right)/2}.
  \end{aligned}
\end{equation}
Particularly,
$$ m(t)\le
m(t_0)+C\left(\log\frac1{t_0}\right)^{\left(\si-\frac{2}{n-2}\right)/2}.
$$
This gives the first bound. The second bound is obtained from the
first one by arguing as at the end of
Lemma~\ref{lem:almost-opt-growth}.
\end{proof}

\begin{remark}
  The growth estimates in Lemma~\ref{lem:opt-est} enable us to
  consider \emph{$\ka$-homogeneous blowups}
$$
u^{\phi}_{t_j}\ra u^{\phi}_0\quad\text{in}\quad
C^1_{\loc}(\R^n_\pm\cup\R^{n-1}).
$$
for $t=t_j\ra 0+$, similar to $3/2$-homogeneous blowups, defined at
the beginning of Section~\ref{sec:growth-estimates}, see
Remark~\ref{rem:kappa-blowup}.
\end{remark}

\begin{proposition}\label{prop:blowup-rot-est} Let $u$ be an almost minimizer for the Signorini problem in $B_1$ such
  that $0\in\Gamma(u)$ and $\widehat N(0+, u)=\ka<\kappa_0$, $\ka=2m$,
  $m\in \mathbb{N}$. Then there exist $C>0$ and $t_0>0$ such that
  $$ \int_{\pa B_1}|u^{\phi}_t-u^{\phi}_s|\le C\left(\log\frac
    1t\right)^{-\frac{1-\g}{2\g}}, \quad 0<t<t_0.
$$
In particular the blowup $u^{\phi}_0$ is unique.
\end{proposition}

\begin{proof}
  Using Lemma~\ref{lem:opt-est}, we apply Lemma~\ref{lem:W-est} with
  $\si=0$ to obtain
  $$ 0\le W_{\ka}(t, u)\le C\left(\log\frac 1t\right)^{-\frac{1}{\g}}.
$$
Recall now the estimate
$$
\int_{\pa B_1}|u^{\phi}_t-u^{\phi}_s|\le C\left(\log \frac t
  s\right)^{1/2}\left(W_{\ka}(t)-W_{\ka}(s)\right)^{1/2},
$$
for $0<s<t<t_0$, that we proved in Lemma~\ref{lem:rotation-est} in the
case $\kappa=3/2$ -- the proof actually works for any
$0<\kappa<\kappa_0$. Then, applying the exponentially dyadic argument
as in the proof of Lemma~\ref{lem:opt-est}, we obtain
\[
  \int_{\pa B_1}|u^{\phi}_t-u^{\phi}_s|\le C\left(\log\frac
    1t\right)^{-\frac{1-\g}{2\g}}.\qedhere
\]
\end{proof}

\begin{lemma}[Nondegeneracy]\label{lem:non-deg} Let $0$ be a free
  boundary point of $u$ such that $\widehat N(0+, u)=\ka$, $\ka=2m$,
  $m\in \mathbb{N}$. Then
  $$ \liminf_{t\ra 0}\int_{\pa B_1}(u^{\phi}_t)^2=\liminf_{t\ra
    0}\frac 1{t^{n+2\ka-1}}\int_{\pa B_t}u^2>0.
$$
\end{lemma}

\begin{proof} We use the approach of \cite{ColSpoVel17}*{Lemma~7.2}.
  Assume to the contrary that for some $r_j\searrow 0+$
  $$ \lim_{j\ra \infty}\frac 1{r_j^{n+2\ka-1}}\int_{\pa B_{r_j}}u^2=0.
$$
Consider then the corresponding Almgren rescalings $u_{r_j}^A(x)$. By
Proposition~\ref{prop:exist-Alm-blowup}, over a subsequence,
$u_{r_j}^A\ra q$ for some blowup $q$. By a characterization of
singular points in Proposition~\ref{prop:sing-char}, $q$ is
$\ka$-homogeneous and is normalized by $\|q\|_{L^2(\pa B_1)}=1$. Next,
for each Almgren rescaling $u_{r_j}^A$ consider its $\ka$-almost
homogeneous rescalings
$$
[u_{r_j}^A]^{\phi}_t:=\frac{u_{r_j}^A(tx)}{\phi(t)}.
$$
Since $u_{r_j}^A$ is an almost minimizer in $B_{1/{r_j}}$ with gauge
function $\omega(t)=(r_jt)^{\al}$, we have
$$
N(0+, u_{r_j}^A)=\lim_{s\ra 0+}N(s, u_{r_j}^A)=\lim_{s\ra 0+}N(r_js,
u)=N(0+, u)=\ka.
$$
Thus, by Proposition~\ref{prop:blowup-rot-est}, over subsequences,
$[u_{r_j}^A]^{\phi}_t$ converges to a unique blowup $q_{r_j}$ and
$$
\int_{\pa B_1}\left|[u_{r_j}^A]^{\phi}_t-q_{r_j}\right|\le
C\left(\log\frac 1t\right)^{-\frac{1-\g}{2\g}}, \quad 0<t<t_0.
$$
Notice that since $\|u_{r_j}^A\|_{W^{1, 2}(B_1)}$ is uniformly
bounded, the constant $C$ is independent of $r_j$, $t$. Now we fix
$r_j$, and consider a sequence
$\{\rho_i\}_{i=1}^{\infty}=\{r_i/r_j\}_{i=1}^{\infty}$. Note that up
to subsequence, $[u_{r_j}^A]^{\phi}_{\rho_i}\ra q_{r_j}$ as
$\rho_i \ra 0$, by the uniqueness.  Then
\begin{align*} \int_{\pa B_1}q_{r_j}^2 &= \lim_{\rho_i\ra 0}\frac1{\rho_i^{n+2\ka-1}}\int_{\pa B_{\rho_i}}(u_{r_j}^A)^2\\
                                       &= \frac{r_j^{n+2\ka-1}}{\int_{\pa B_{r_j}}u^2}\lim_{i\ra \infty}\frac{1}{(r_j\rho_i)^{n+2\ka-1}}\int_{\pa B_{r_j\rho_i}}u^2\\
                                       &= \frac{r_j^{n+2\ka-1}}{\int_{\pa B_{r_j}}u^2}\lim_{i\ra \infty}\frac{1}{r_i^{n+2\ka-1}}\int_{\pa B_{r_i}}u^2\\
                                       &=0
\end{align*}
by the contradiction assumption. Thus, $q_{r_j}=0$ on $\pa B_1$, and
hence
$$
\int_{\pa B_1}\left|[u_{r_j}^A]^{\phi}_t\right|\le C\left(\log\frac
  1t\right)^{-\frac{1-\g}{2\g}}.
$$
Now for any $\rho>0$ and $r_j$,
\begin{align*}
  1&= \frac1{\rho^{n+2\ka-1}}\int_{\pa B_{\rho}}q^2\\
   &\le \frac{\|q\|_{L^\infty(\partial B_\rho)}}{\rho^{\ka}}\frac1{\rho^{n+\ka-1}}\int_{\pa B_{\rho}}|q|\\
   &\le \|q\|_{L^{\infty}(\pa B_{1})}\left[ \frac1{\rho^{n+\ka-1}}\int_{\pa B_{\rho}}|q-u_{r_j}^A|+\frac1{\rho^{n+\ka-1}}\int_{\pa B_{\rho}}|u_{r_j}^A|\right]\\
   &\le \|q\|_{L^{\infty}(\pa B_1)}\left[\frac1{\rho^{n+\ka-1}}C_n\rho^{\frac{n-1}2}\left(\int_{\pa B_{\rho}}|q-u_{r_j}^A|^2\right)^{1/2}+e^{-\left(\frac{\ka b}{\al}\right)\rho^{\al}}\int_{\pa B_1}\left|[u_{r_j}^A]^{\phi}_{\rho}\right|\right]\\
   &\le C\|q\|_{L^{\infty}(\pa B_1)}\left[\left(\frac1{\rho^{n+2\ka-1}}\int_{\pa B_{\rho}}|q-u_{r_j}^A|^2  \right)^{1/2}+\left(\log\frac1{\rho}\right)^{-\frac{1-\g}{2\g}}\right].
\end{align*}
Note that $u_{r_j}^A\ra q$ in $C^1_{\text{loc}}(B_1^{\pm}\cup
B_1')$. We choose first $\rho>0$ small and then $r_j=r_j(\rho)>0$
small to reach a contradiction.
\end{proof}

The nondegeneracy implies the following important fact, which enables
the use of the Whitney Extension Theorem in the proof of the
structural theorem on the singular set (Theorem~\ref{thm:sing} below).

\medskip For $\ka=2m<\kappa_0$, $m\in \mathbb{N}$, we denote
$$
\Sigma_{\ka}(u):=\{x_0\in \Si(u):N(0+, u, x_0)=\ka\}.
$$

\begin{lemma}\label{lem:sing-closed}
  The set $\Sigma_{\ka}(u)$ is of topological type $F_{\si}$; i.e., it
  is a countable union of closed sets.
\end{lemma}

\begin{proof} For $j\in\mathbb{N}$, $j\geq 2$, let
  $$
  E_j:=\Bigl\{x_0\in \Sigma_{\ka}(u)\cap \overline{B_{1-1/j}}:
  \frac1j\le \frac{1}{\rho^{n+2\ka-1}}\int_{\pa B_{\rho}(x_0)}u^2 \le
  j\ \text{for}\ 0<\rho<\frac{1}{2j}\Bigr\}.
  $$
  Then by Lemma~\ref{lem:opt-est} and Lemma~\ref{lem:non-deg},
  $\Sigma_{\ka}(u)=\bigcup_{j=2}^{\infty}E_j$. We now claim that $E_j$
  is closed for any $j\geq 2$. Indeed, take a sequence $x_i\in E_j$
  such that $x_i\ra x_0$ as $i\ra \infty$. Then
  $x_0\in \overline{B_{1-1/j}}$ and for every $0<\rho<1/(2j)$, by the
  local uniform continuity of $u$,
  \begin{equation}\label{eq:F-sigma-ineq-x0}
    \frac{1}{\rho^{n+2\ka-1}}\int_{\pa B_{\rho}(x_0)}u^2=\lim_{i\ra \infty}\frac{1}{\rho^{n+2\ka-1}}\int_{\pa B_{\rho}(x_i)}u^2\in \left[\frac1j, j\right].
  \end{equation}
  Next, since $\Gamma(u)$ is relatively closed in $B_1'$, we also know
  that $x_0\in\Gamma(u)$. Moreover, since $N(0+, u, x_i)=\ka$ and the
  function $x\mapsto \widehat N(0+, u, x)$ is upper semicontinuous, we
  have
$$
\ka=\limsup\limits_{i\ra \infty}\widehat N(0+, u, x_i)\le \widehat
N(0+, u, x_0).
$$
If $\widehat N(0+, u, x_0)=\ka'>\ka$, then by
Lemma~\ref{lem:almost-opt-growth},
$$ \frac{1}{\rho^{n+2\ka-1}}\int_{\pa B_{\rho}(x_0)}u^2\le
C\rho^{2(\ka'-\ka)}\left(\log\frac1\rho\right)\ra 0\quad\text{as
}\rho\ra 0,
$$
which contradicts \eqref{eq:F-sigma-ineq-x0}. Therefore,
$\widehat N(0+, u, x_0)=\ka$ and consequently $x_0\in E_j$. Hence,
$E_j$ is closed, $j=2,3,\ldots$, implying that $\Sigma_\kappa(u)$ is
$F_\sigma$.
\end{proof}

To state the main result of this paper concerning the singular points,
we need to introduce the following notations. For $\ka=2m<\kappa_0$,
$m\in \mathbb{N}$ and $x_0\in \Sigma_\kappa(u)$, we define
$$
d^{(\ka)}_{x_0}:=\dim\{\xi\in \R^{n-1}:\xi\cdot
\D_{x'}u_{x_0}^\phi(x', 0)\equiv 0\ \text{on}\ \R^{n-1}\},
$$
which has the meaning of the dimension of $\Sigma_\kappa(u)$ at $x_0$,
and where $u_{x_0}^\phi$ is the unique $\kappa$-homogeneous blowup at
$x_0$. In fact, $d_{x_0}^{(\kappa)}$ is the dimension of the linear
subspace $\Sigma_\kappa (u_{x_0}^\phi)\subset \R^{n-1}$. Since
$u_{x_0}^\phi$ is a nonzero solution of the Signorini problem, it
cannot vanish identically on $\R^{n-1}$ (see \cite{GarPet09}) and
therefore $d_{x_0}^{(\kappa)}<n-1$.

For $d=0,1,\ldots,n-2$, we denote
$$
\Si^d _{\ka}(u):=\{x_0\in \Si_{\ka}(u): d^{(\ka)}_{x_0}=d\}.
$$

\begin{theorem}[Structure of the singular set]\label{thm:sing}
  Let $u$ be an almost minimizer for the Signorini problem in
  $B_1$. Then for every $\ka=2m<\kappa_0$, $m\in \mathbb{N}$, and
  $d=0,1,\ldots,n-2$, the set $\Sigma_{\ka}^d(u)$ is contained in the
  union of countably many submanifolds of dimension $d$ and class
  $C^{1, \log}$.
\end{theorem}

\begin{proof} Let $\ka=2m$, $m\in \mathbb{N}$. For
  $x\in \Sigma_{\kappa}(u)\cap B_{1/2}'$, let
  $q_{x}\in\mathcal{Q}_\kappa$ denote the unique $\kappa$-homogeneous
  blowup of $u$. By the optimal growth (Lemma~\ref{lem:opt-est}) and
  the nondegeneracy (Lemma~\ref{lem:non-deg}), we can write
  $$
  q_{x}=\lambda_x q_x^A,\quad \lambda_x>0,\quad
  \|q_x^A\|_{L^2(\partial B_1)}=1,
  $$
  where $q_{x}^A\in\mathcal{Q}_\kappa$ is the corresponding Almgren
  blowup. We want to show that the $q_x$, $q_x^A$, $\lambda_x$ depend
  continuously on $x\in \Sigma_\kappa$, with a logarithmic modulus of
  continuity.

  Let $x_1, x_2\in \Si_{\ka}(u)\cap B_{1/2}$. Then for $t>0$, to be
  chosen below, we can write
  \begin{equation}\label{eq:sing-2}
    \begin{multlined}
      \|q_{x_1}-q_{x_2}\|_{L^1(\pa B_1)}\le \|q_{x_1}-u^{\phi}_{x_1,
        t}\|_{L^1(\pa B_1)}\\\qquad+\|u^{\phi}_{x_1, t}-u^{\phi}_{x_2,
        t}\|_{L^1(\pa B_1)} +\|u^{\phi}_{x_2, t}-q_{x_2}\|_{L^1(\pa
        B_1)}.
    \end{multlined}
  \end{equation}
  By Proposition~\ref{prop:blowup-rot-est}, we have
  \begin{equation}\label{eq:sing-4}
    \|q_{x}-u^{\phi}_{x, t}\|_{L^1(\pa B_1)}\le C\left(\log\frac 1t\right)^{-\frac 1{n-2}}
  \end{equation}
  for $x\in \Sigma_\kappa(u)\cap B_{1/2}'$. This controls the first
  and third term on the right hand side of \eqref{eq:sing-2} To
  estimate the middle term, we observe that
  \begin{multline*}
    \|u^{\phi}_{x_1, t}-u^{\phi}_{x_2, t}\|_{L^1(\pa B_1)}\\\le
    \frac{e^{\left(\frac{\ka b}{\al}
        \right)t^{\al}}}{t^{\ka}}\int_{\pa B_1}\int_0^1\left|\D
      u(x_1+tz+r(x_2-x_1))\right| |x_1-x_2|\,dr\,dS_z
  \end{multline*}
  for any $0<t<1/2$. Recalling that $\D u(x_1)=0$ and
  $u\in C^{1, \be}(B_1^{\pm}\cup B'_1)$, we have
  \begin{multline*}
    |\D u(x_1+tz+r(x_2-x_1)|\\*\le C|tz+r(x_2-x_1)|^{\be}\le
    C(t+|x_1-x_2|)^{\be}\le C|x_1-x_2|^{\frac{\be}{2(\ka-\be)}}
  \end{multline*}
  if we choose $t=|x_1-x_2|^{\frac {1}{2(\ka-\be)}}$ and have
  $|x_1-x_2|<\left(1/2\right)^{2(\ka-\be)}$. This gives
  \begin{equation}
    \|u^{\phi}_{x_1, t}-u^{\phi}_{x_2, t}\|_{L^1(\pa B_1)}\le \frac{C}{t^{\ka}}|x_1-x_2|^{\frac{\be}{2(\ka-\be)}}|x_1-x_2|\le C|x_1-x_2|^{1/2}.
    \label{eq:sing-3}
  \end{equation}
  Combining \eqref{eq:sing-2}, \eqref{eq:sing-3}, and
  \eqref{eq:sing-4}, we obtain
  \begin{equation}\label{eq:sing-5}
    \|q_{x_1}-q_{x_2}\|_{L^1(\pa B_1)}\le C\left(\log\frac 1{|x_1-x_2|}\right)^{-\frac 1{n-2}}.
  \end{equation}
  Next, by Lemma~\ref{lem:opt-est}, for any
  $x\in \Sigma_\kappa(u)\cap B_{1/2}'$ and small $t$
$$
\int_{\pa B_1}(u^{\phi}_{x, t})^2\le C
$$
with $C$ independent of $x$, and passing to the limit as $t\to \infty$
obtain the bound
$$
\lambda_x^2=\int_{\pa B_1}q_x^2\le C
$$
Moreover, since $q_x$ is a $\kappa$-homogeneous harmonic polynomial,
we also have
\begin{equation}\label{eq:sing-15}
  \|q_{x}\|_{L^\infty(B_1)}\leq C(n,\kappa)\|q_{x}\|_{L^2(\partial
    B_1)}\leq C.
\end{equation}
Then, by combining \eqref{eq:sing-5} and \eqref{eq:sing-15}, we have
\begin{equation}
  \begin{aligned}
    |\la_{x_1}-\la_{x_2}| &\le |\la_{x_1}^2-\la_{x_2}^2|^{1/2}\leq
    \left(\int_{\partial
        B_1}|q_{x_1}^2-q_{x_2}^2|\right)^{1/2}\\
    &\leq
    \|q_{x_1}+q_{x_2}\|_{L^\infty(B_1)}^{1/2}\|q_{x_1}-q_{x_1}\|_{L^1(\partial
      B_1)}^{1/2}\\
    &\le C\left(\log\frac 1{|x_1-x_2|}\right)^{-\frac1{2(n-2)}}.
  \end{aligned}
  \label{eq:sing-8}
\end{equation}
Finally, we want to estimate $q_{x_1}^A-q_{x_2}^A$. By writing
\begin{align*}
  \|q_{x_1}-q_{x_2}\|_{L^1(\pa B_1)}&= \int_{\pa B_1}|\la_{x_1}q_{x_1}^A-\la_{x_2}q_{x_2}^A|\\
                                    &= \int_{\pa B_1}|\la_{x_1}(q_{x_1}^A-q_{x_2}^A)+(\la_{x_1}-\la_{x_2})q_{x_2}^A|\\
                                    &\ge \la_{x_1}\int_{\pa B_1}|q_{x_1}^A-q_{x_2}^A|-|\la_{x_1}-\la_{x_2}|\int_{\pa B_1}|q_{x_2}^A|,
\end{align*}
we estimate
\begin{equation}
  \begin{aligned}
    \la_{x_1}\int_{\pa B_1}|q_{x_1}^A-q_{x_2}^A|&\le \|q_{x_1}-q_{x_2}\|_{L^1(\pa B_1)}+|\la_{x_1}-\la_{x_2}|\int_{\pa B_1}|q_{x_2}^A|\\
    &\le \|q_{x_1}-q_{x_2}\|_{L^1(\pa B_1)}+C(n)|\la_{x_1}-\la_{x_2}|\\
    &\le C\left(\log\frac 1{|x_1-x_2|}\right)^{-\frac1{2(n-2)}},
  \end{aligned}
  \label{eq:sing-9}
\end{equation}
where we used $\|q_{x_2}^A\|_{L^2(\pa B_1)}=1$ in the second
inequality and \eqref{eq:sing-5} and the bound \eqref{eq:sing-8} in
the third inequality. Next, using that $q_x^A$ are
$\kappa$-homogeneous harmonic polynomials, we
have
$$\|q_{x_1}^A-q_{x_2}^A\|_{L^\infty(B_1)}\leq
C\|q_{x_1}^A-q_{x_2}^A\|_{L^1(\partial B_1)},$$ which combined with
\eqref{eq:sing-9} gives
\begin{equation}\label{eq:sing-20}
  \lambda_{x_1}\|q_{x_1}^A-q_{x_2}^A\|_{L^\infty(B_1)}\leq C\left(\log\frac 1{|x_1-x_2|}\right)^{-\frac1{2(n-2)}}. 
\end{equation}
Now we fix $x_0\in \Si_{\ka}(u)\cap B_{1/4}'$. Then by
\eqref{eq:sing-8}, there exists
$\de=\de(x_0)\in \left(0, \left(1/2\right)^{2(\ka-\be)+1}\right)$ such
that $\la_x\ge 1/2\la_{x_0}$ if
$x\in\Sigma_\kappa(u)\cap B_\delta'(x_0)$. Then by \eqref{eq:sing-20},
we conclude that
\begin{equation}\label{eq:sing-10}
  \|q_{x_1}^A-q_{x_2}^A\|_{L^{\infty}(B_1)}\le C\left(\log\frac 1{|x_1-x_2|}\right)^{-\frac1{2(n-2)}},\quad x_1, x_2\in \Si_{\ka}(u)\cap B_{\de}(x_0).
\end{equation}
Notice that the constant $C$ does not depend on $x_1$, $x_2$, but both
$C$ and $\delta$ do depend on $x_0$.

Once we have the estimates \eqref{eq:sing-8} and \eqref{eq:sing-10},
as well as Lemma~\ref{lem:sing-closed}, we can apply the Whitney
Extension Theorem of Fefferman's \cite{Fef09}, to complete the proof,
see e.g., the proof of Theorem~5 in \cite{ColSpoVel17}.
\end{proof}

\appendix


\section{Some examples of almost minimizers}
\label{sec:ex-drift}

\begin{example} If $u$ is a minimizer of the functional
$$
\int_D a(x)|\nabla u|^2
$$
over the set $\mathfrak{K}_{\psi,g}(D,\cM)$ with strictly positive
$a\in C^{0,\alpha}(\overline{D})$, $0<\alpha\leq1$, then $u$ is an
almost minimizer for the Signorini problem with a gauge function
$\omega(r)=C r^\alpha$.
\end{example}
\begin{proof} This is rather immediate.
\end{proof}
\begin{example}\label{ex:drift} Let $u$ be a solution of the Signorini
  problem for the Laplacian with drift with the velocity field
  $b\in L^p(B_1)$, $p>n$:
  \begin{align*}
    -\Delta u + b(x)\nabla u=0&\quad\text{in }B_1^\pm\\
    -\partial_{x_n}u\geq 0,\quad u\geq 0,\quad
    u\partial_{x_n}u=0&\quad\text{on }B_1',
  \end{align*}
  even in $x_n$-variable. We understand this in the weak sense that
  $u$ satisfies the variational inequality
$$
\int_{B_1}\nabla u\nabla(w-u)+(b(x)\nabla u) (w-u)\geq 0,
$$
for any competitor $w\in\mathfrak{K}_{0,u}(B_1,B_1')$, i.e.
$w\in u+W^{1,2}_0(B_1)$ such that $w\geq 0$ on $B_1'$ in the sense of
traces. Then $u$ is an almost minimizer for the Signorini problem with
$\psi=0$ on $\cM=\R^{n-1}\times\{0\}$ and a gauge function
$\omega(r)=Cr^{1-n/p}$.
\end{example}
\begin{proof} Let $B_r(x_0)\Subset B_1$ and
  $w\in\mathfrak{K}_{0,u}(B_r(x_0),B_1')$. Extending $w$ as equal to
  $u$ in $B_1\setminus B_r(x_0)$, and applying the variational
  inequality for $u$, we obtain
  \begin{equation}\label{eq:ex-u-var-ineq}
    \int_{B_r(x_0)}\nabla u\nabla(w-u)+b(x)\nabla u(w-u)\geq 0.
  \end{equation}
  Let $v$ be the Signorini replacement of $u$ on $B_r(x_0)$. Then $v$
  satisfies the variational inequality
  \begin{equation}\label{eq:ex-v-var-ineq}
    \int_{B_r(x_0)} \nabla v\nabla (w-v)\geq 0,
  \end{equation}
  for all $w$ as above. Now, taking $w=u\pm (u-v)^+$ in
  \eqref{eq:ex-u-var-ineq} we will have
$$
\int_{B_r(x_0)}\nabla u\nabla (u-v)^+ +(b(x)\nabla u)(u-v)^+=0.
$$
Next, taking $w=v+(u-v)^+$ in \eqref{eq:ex-v-var-ineq}, we have
$$
\int_{B_r(x_0)}\nabla v\nabla (u-v)^+\geq 0.
$$
Taking the difference, we then obtain
$$
\int_{B_r(x_0)}|\nabla(u-v)^+|^2\leq -\int_{B_r(x_0)}b(x)\nabla
u(u-v)^+.
$$
Similarly, taking $w=v\pm (v-u)^+$ in \eqref{eq:ex-v-var-ineq} and
$w=u+(v-u)^+$ in \eqref{eq:ex-u-var-ineq} and subtracting the
resulting inequalities, we obtain
$$
\int_{B_r(x_0)}|\nabla(v-u)^+|^2\leq \int_{B_r(x_0)}b(x)\nabla
u(v-u)^+.
$$
Hence, combining the inequalities above, we arrive at
$$
\int_{B_r(x_0)}|\nabla(v-u)|^2\leq \int_{B_r(x_0)}|b(x)| |\nabla u|
|v-u|.
$$
Then, applying H\"older's inequality, we have for $p>n$
$$
\int_{B_r(x_0)}|\nabla(v-u)|^2\leq \|b\|_{L^p(B_r(x_0))}\|\nabla
u\|_{L^2(B_r(x_0))}\|v-u\|_{L^{p^*}(B_r(x_0))},
$$
with $p^*=2p/(p-2)$. Next, since $v-u\in W^{1,2}_0(B_1)$, from the
Sobolev's inequality we have
$$
\|v-u\|_{L^{p^*}(B_r(x_0))}\leq C_{n,p}
r^{1-n/p}\|\nabla(v-u)\|_{L^2(B_r(x_0))}
$$
and hence we can conclude that
$$
\int_{B_r(x_0)}|\nabla(v-u)|^2\leq Cr^{2(1-n/p)}\int_{B_r(x_0)}|\nabla
u|^2
$$
with $C=C_{n,p}\|b\|_{L^p(B_1)}^2$. This implies
\begin{align*}
  &\int_{B_r(x_0)}|\nabla u|^2-\int_{B_r(x_0)}|\nabla
    v|^2=\int_{B_r(x_0)} (\nabla u+\nabla v)(\nabla u-
    \nabla v)\\
  &\qquad\leq Cr^\gamma \int_{B_r(x_0)}(
    |\nabla u|^2+|\nabla
    v|^2)+Cr^{-\gamma}\int_{B_r(x_0)}|\nabla (v-u)|^2\\
  &\qquad\leq Cr^\gamma \int_{B_r(x_0)}(|\nabla u|^2+|\nabla
    v|^2)+Cr^{2(1-n/p)-\gamma}\int_{B_r(x_0)}|\nabla u|^2,
\end{align*}
where we have used Young's inequality in the second line.  Choosing
$\gamma=1-n/p$ we then deduce that for small enough
$0<r<r_0(n,p,\|b\|_{L^p(B_1)})$
$$
\int_{B_r(x_0)}|\nabla u|^2\leq (1+Cr^{1-n/p})\int_{B_r(x_0)}|\nabla
v|^2
$$
with $C=C_{n,p}\|b\|_{L^p(B_1)}^2$.
\end{proof}


\end{document}